\documentclass[a4paper,onecolumn,oneside,10pt]{article}

\usepackage{geometry}
\usepackage{xcolor, tikz}

\usepackage[hyphens]{url}

\usepackage[title]{appendix}
\usepackage{soul}

\usepackage[T1]{fontenc}
\usepackage{hyperref}
\usepackage{lineno, amsthm}
\usepackage{mathtools, comment}
\usepackage{mathdots}
\usepackage{enumerate}
\usepackage{textcomp}
\usepackage{amssymb,amsfonts}
\usepackage{amsmath,paralist,enumitem,mathrsfs,graphicx}
\usepackage{epstopdf}
\usepackage{pgfplots}

\theoremstyle{definition}
\newtheorem{ex}{Example}
\theoremstyle{plain}
\newtheorem{rem}{Remark}
\newtheorem{theorem}{Theorem}
\newtheorem{lemma}[theorem]{Lemma}
\newtheorem{prop}[theorem]{Proposition}
\newtheorem{cor}[theorem]{Corollary}
\newtheorem{ass}{Assumption}
\newcommand\ts{\scriptscriptstyle\square}
\newcommand\norm[1]{\left\lVert#1\right\rVert}
\makeatletter
\let\save@mathaccent\mathaccent
\newcommand*\if@single[3]{%
	\setbox0\hbox{${\mathaccent"0362{#1}}^H$}%
	\setbox2\hbox{${\mathaccent"0362{\kern0pt#1}}^H$}%
	\ifdim\ht0=\ht2 #3\else #2\fi
}
\newcommand*\rel@kern[1]{\kern#1\dimexpr\macc@kerna}
\newcommand*\widebar[1]{\@ifnextchar^{{\wide@bar{#1}{0}}}{\wide@bar{#1}{1}}}
\newcommand*\wide@bar[2]{\if@single{#1}{\wide@bar@{#1}{#2}{1}}{\wide@bar@{#1}{#2}{2}}}
\newcommand*\wide@bar@[3]{%
	\begingroup
	\def\mathaccent##1##2{%
		\let\mathaccent\save@mathaccent
		\if#32 \let\macc@nucleus\first@char \fi
		\setbox\z@\hbox{$\macc@style{\macc@nucleus}_{}$}%
		\setbox\tw@\hbox{$\macc@style{\macc@nucleus}{}_{}$}%
		\dimen@\wd\tw@
		\advance\dimen@-\wd\z@
		\divide\dimen@ 3
		\@tempdima\wd\tw@
		\advance\@tempdima-\scriptspace
		\divide\@tempdima 10
		\advance\dimen@-\@tempdima
		\ifdim\dimen@>\z@ \dimen@0pt\fi
		\rel@kern{0.6}\kern-\dimen@
		\if#31
		\overline{\rel@kern{-0.6}\kern\dimen@\macc@nucleus\rel@kern{0.4}\kern\dimen@}%
		\advance\dimen@0.4\dimexpr\macc@kerna
		\let\final@kern#2%
		\ifdim\dimen@<\z@ \let\final@kern1\fi
		\if\final@kern1 \kern-\dimen@\fi
		\else
		\overline{\rel@kern{-0.6}\kern\dimen@#1}%
		\fi
	}%
	\macc@depth\@ne
	\let\math@bgroup\@empty \let\math@egroup\macc@set@skewchar
	\mathsurround\z@ \frozen@everymath{\mathgroup\macc@group\relax}%
	\macc@set@skewchar\relax
	\let\mathaccentV\macc@nested@a
	\if#31
	\macc@nested@a\relax111{#1}%
	\else
	\def\gobble@till@marker##1\endmarker{}%
	\futurelet\first@char\gobble@till@marker#1\endmarker
	\ifcat\noexpand\first@char A\else
	\def\first@char{}%
	\fi
	\macc@nested@a\relax111{\first@char}%
	\fi
	\endgroup
}

\newcommand{\upperRomannumeral}[1]{\uppercase\expandafter{\romannumeral#1}}
\newcommand{\lowerRomannumeral}[1]{\lowercase\expandafter{\romannumeral#1}}

\usetikzlibrary{pgfplots.dateplot}

\newcommand\restr[2]{{% we make the whole thing an ordinary symbol
		\left.\kern-\nulldelimiterspace % automatically resize the bar with \right
		#1 % the function
		\vphantom{\big|} % pretend it's a little taller at normal size
		\right|_{#2} % this is the delimiter
}}
\newcommand{\dd}{\mathrm{d}}

\makeatletter
\def\ps@pprintTitle{%
	\let\@oddhead\@empty
	\let\@evenhead\@empty
	\def\@oddfoot{\footnotesize\itshape
		\ifx\@empty\@empty
		\else\@journal\fi\hfill\today}%
	\let\@evenfoot\@oddfoot	
}

\title{On the Kolmogorov equation associated with Volterra equations and Fractional Brownian Motion\footnote{The research of Alessandro Bondi benefited from the financial support of the chair ``Statistiques et Modèles pour le Régulation'' of École Polytechnique. The study for this paper began during the Ph.D. of Alessandro Bondi at Scuola Normale Superiore di Pisa, which the author thanks.}}  
\author{Alessandro Bondi\thanks{Centre de Mathématiques Appliquées (CMAP), CNRS, École Polytechnique, Institut Polytechnique de Paris. \textbf{Email:} alessandro.bondi@polytechnique.edu} 
	\thanks{Classe di Scienze, Scuola Normale Superiore di Pisa}
	\and {Franco Flandoli	\thanks{Classe di Scienze, Scuola Normale Superiore di Pisa. \textbf{Email:} franco.flandoli@sns.it} }}
	
\begin{document}
\maketitle
		\begin{abstract}
		We consider a Volterra convolution equation in $\mathbb{R}^d$ perturbed with an additive fractional Brownian motion of Riemann–Liouville type with Hurst parameter $H\in (0,1)$.
		We show that its solution solves a stochastic partial differential equation (SPDE) in the Hilbert space of square–integrable functions. Such an equation motivates our  study of an unconventional class of SPDEs requiring an original extension of the drift operator and its Fréchet differentials. We prove that these SPDEs generate a Markov stochastic flow which is twice Fréchet differentiable with respect to the initial data. This stochastic flow is then employed to solve, in the classical sense of  infinite dimensional calculus, the  path–dependent Kolmogorov equation corresponding to the SPDEs. In particular, we associate a time–dependent infinitesimal generator with the fractional Brownian motion.
		In the final section, we show some obstructions  in the analysis of the mild formulation of the Kolmogorov equation for SPDEs driven by the same infinite dimensional noise. This problem, which is relevant to the theory of regularization-by-noise, remains open for future research.
			\\[1ex] 
			\\[1ex] 
		\end{abstract}

	\section{Introduction}\label{intro}
Consider the stochastic differential equation (SDE) in $\mathbb{R}^{d}$ with additive noise
\begin{equation}
	X_{t}=x_{0}+\int_{0}^{t}k_{1}\left(  t-s\right)  b\left(  s,X_{s}\right)
	\dd s
	+
	\frac{1}{\Gamma\left(\alpha\right)}\int_{0}^{t}\left(  t-s\right)^{\alpha-1}  \dd W_{s}\label{SDE Volterra},%
\end{equation}
where $x_{0}\in\mathbb{R}^{d}$, $\alpha\in(\frac12,1)$, $W=\left(  W_{t}\right)  _{t\geq0}$ is a standard
Brownian motion in $\mathbb{R}^{d}$, $b:\left[  0,T\right]  \times
\mathbb{R}^{d}\rightarrow\mathbb{R}^{d}$ is a measurable vector field and{ \color{black}$k_{1}$ is a locally square--integrable, $\mathbb{R}-$valued kernel that is continuous in $(0,\infty)$}. This equation belongs to the class of  stochastic Volterra equations (of convolution type), which is characterized by a wide and continuously expanding body of literature, see for instance \cite{ACLP, sergio, aff_jumps,  HL, MS,  Pr}.
The additive noise driving the SDE \eqref{SDE Volterra} is a fractional Brownian motion (henceforth, fBm) of Riemann–Liouville type, with Hurst parameter $H=\alpha-\frac{1}{2}\in(0,\frac12)$. Our motivation for studying this random perturbation stems from its relevance in mathematical finance, particularly in  the field of rough volatility models, see \cite{rHH, RH,qr}. However, the theory that we develop in this paper encompasses also the case $\alpha\in [1,\frac32)$, corresponding to a fBM with Hurst parameter $H\in [\frac{1}{2},1)$, exhibiting smoother trajectories and longer memory.

%{\color{blue}\st{Our aim is to associate an infinitesimal generator with the problem (1) %\eqref{SDE Volterra} 
%	and to prove,
%under  additional regularity assumptions on $b$, existence of
%a classical solution to the corresponding Kolmogorov equation, which is introduced below
%after some preliminary notations.}}

Inspired by \cite{FZ}, our aim is to insert the Volterra SDE \eqref{SDE Volterra} in a class of infinite
dimensional SDEs in a separable Hilbert space $(H, \langle\cdot, \cdot\rangle_{H})$ of the form%
\begin{equation}\label{SDE Hilbert}
	w_{t}=\phi+\int_{0}^{t}B\left(  s,w_{s}\right)  \dd s+\int_{0}^{t}\sigma\left(s\right)\,\dd W_{s},\quad \phi\in H, 
\end{equation}
where $\sigma\colon[0,T]\to \mathcal{L}(\mathbb{R}^d;H) $ and $B\colon [0,T]\times H\to H$. In order to  achieve this objective,  we need to consider a drift $B$ with an unconventional structure. This motivates the study  carried out in this paper of a novel class of stochastic partial differential equations (SPDEs) and of the associated stochastic flow's regularity. Notably, these SPDEs require an extension of the drift operator and its Fréchet differentials.\\
Given $\Phi\colon H\to \mathbb{R}$, we then study the following backward Kolmogorov equation associated with \eqref{SDE Hilbert}:
\begin{align*}
	\begin{cases}
		\partial_{t}u\left(  t,\phi\right)  +  \mathcal{A}_tu  \left(
		t,\phi\right)  =0,\quad t\in\left[  0,T\right] ,\,\phi\in H,\\
		u\left(  T,\phi\right)    =\Phi\left(  \phi\right),
	\end{cases}	
\end{align*}
which will be interpreted in integral form, see \eqref{kolm}. 
Here $\mathcal{A}_t$, the time–dependent infinitesimal generator, is given by
\[
  \mathcal{A}_tu  \left(  t,\phi\right)  =\frac{1}{2}\text{Tr}
\left(
D^{2}u\left(  t,\phi\right)  \sigma(t)\sigma(t) ^\ast
\right)
+\left\langle
B\left(  t,\phi\right)  ,\nabla u\left(  t,\phi\right)  \right\rangle_{H}.
\]
As
in \cite[Chapter 9]{DP}, the approach that we adopt for the existence of classical solutions of the
Kolmogorov equation is based on  a careful analysis of \eqref{SDE Hilbert} and on the formula%
\begin{equation}\label{dais}
u\left(  t,\phi\right)  =\mathbb{E}\left[  \Phi\left(  w_{T}^{t,\phi}\right)
\right],
\end{equation}
where $w_{t}^{t_{0},\phi}$, $t\in\left[  t_{0},T\right]  $, is the solution
of an analogue of \eqref{SDE Hilbert} starting at time $t_{0}$
instead of $0$.

It is worth noting that we use classical tools of  infinite dimensional calculus, such as the Frech\'{e}t
derivative, when analyzing the Kolmogorov equation. This is a novelty compared to other studies  addressing  path–dependent PDEs related to Volterra SDEs, particularly \cite{VZ} (see also \cite{BR22} for a similar subject). 
In a sense, then, we unify the study of stochastic Volterra equations and fBm of Riemann–Liouville type to other
infinite dimensional systems. However, the assumptions imposed on $B$ are not
entirely classical, resulting in  an innovative abstract formulation of the problem. Consequently, the analysis developed here is only analogous to the
classical one, not included into it.

A more direct approach to the Kolmogorov equation would be also of great
interest for two reasons. Firstly, it would contribute to complete the comparison with the classical theory developed for
other classes of problems, see  \cite{DP}. Secondly, it could be used  to study regularization-by-noise phenomena for  SDEs driven by fractional Brownian motion, which are investigated in literature using different techniques, see, e.g., \cite{GH, GHM, HL, MS,  NO02}.  In fact, studying the Kolmogorov equation in mild form might allow to prove weak uniqueness of solutions of the
underlying SDE when the drift is not smooth, see  \cite{DFPR,DFRV, FlaSF,Stroock,Veret}. In an attempt to
develop such a direct approach, we have identified obstructions that we report in Section \ref{mde}, so this problem remains open.

The paper is structured as follows. In Section \ref{section abstract} we show the connection between the Volterra SDE \eqref{SDE Volterra} and the SDE  \eqref{SDE Hilbert}, specifying the Hilbert space $H$ considered in our study. Moreover, due to the particular structure of the drift $B$, we introduce another infinite dimensional reformulation for \eqref{SDE Volterra} (see \eqref{abstr2} in Proposition \ref{pro2}), which is at the core of our analysis. In  Section \ref{sec_abstract}  we study the reformulation given by \eqref{abstr2} in an abstract setting (see \eqref{eq_B1}), focusing also on the regularity of its  solution with respect to the initial data, see Subsections \ref{first_order}-\ref{second_order}. The related backward Kolmogorov equation in integral form is then investigated in Section \ref{sec_KO}. In Section \ref{mde} we discuss the mild formulation of the Kolmogorov equation and its importance for the theory of regularization by noise, see Subsection \ref{reg_noise}. The challenges that  we previously mentioned regarding the analysis  of such a mild formulation are explained  in  Subsection \ref{grad_est}.
Finally, in Appendix \ref{appendix1} we study the regularity of the solution of the Kolmogorov equation constructed as in \eqref{dais}.

\section{Infinite dimensional reformulations for the stochastic Volterra equation\label{section abstract}}

Let $\left(  \Omega,\mathcal{F},\mathbb{P},\mathbb{F}\right)  $ be a complete
filtered probability space, with expectation denoted by $\mathbb{E}$, where
the filtration $\mathbb{F}$\ $=\left(  \mathcal{F}_{t}\right)  _{t\in\left[
	0,T\right]  }$ satisfies the usual conditions. Fix $d\in\mathbb{N}$ and consider  an $\mathbb{{R}}^d-$valued standard Brownian motion $W=\left(  W_{t}\right)
_{t\geq0}$ defined on $\left(
\Omega,\mathcal{F},\mathbb{P},\mathbb{F}\right)$. In what follows, we denote by $k_2\colon(0,\infty)\to (0,\infty)$ the fractional kernel which controls the noise in the Volterra SDE \eqref{SDE Volterra}, namely 
\begin{equation}\label{k_f}
	k_2(t)=\frac{1}{\Gamma(\alpha)}t^{\alpha-1},\quad t>0,\text{ for some }\alpha\in\left(\frac{1}{2},1\right).
\end{equation}
As already mentioned in Introduction \ref{intro}, we note that the arguments and results of this paper continue to hold even when $\alpha\in [1,\frac{3}{2})$, i.e., when the fBM governing \eqref{SDE Volterra} has Hurst parameter in $[1/2,1)$, see also Remark~\ref{vabene}.

Fix $T>0$. Suppose that the measurable vector field $b:\left[  0,T\right]
\times\mathbb{R}^{d}\rightarrow\mathbb{R}^{d}$ 
satisfies, for some $L>0$,%
\begin{equation*}
\left|b\left(t,x\right)\right|\le L\left(1+\left|x\right|\right), \qquad 
\left| b\left(  t,x\right)  -b\left(  t,y\right)  \right| \leq
L\left| x-y\right|,
\end{equation*}
for every $t\in\left[  0,T\right]$ and   $x,y\in\mathbb{R}^{d}$.
By (strong)
solution of \eqref{SDE Volterra} we mean a {{continuous}} adapted process
satisfying the identity for every $t\in\left[  0,T\right],\, \mathbb{P}-$a.s. Existence and pathwise uniqueness
of strong solutions of  \eqref{SDE Volterra} have been studied in literature  under additional requirements on $k_1$, see, e.g., Equation (2.5) and Theorem 3.3 in \cite{sergio}.

Let $H$ be the Hilbert space 
$L^{2}\left(  0,T;\mathbb{R}^{d}\right)$ and denote  by  $\langle\cdot,\cdot\rangle_{H}$ the usual inner product.  Denoting  by $\mathcal{L}(\mathbb{R}^d;H)$ the space of linear and bounded operators from $\mathbb{R}^d$ to $H$, define $\sigma\colon[0,T]\to\mathcal{L}(\mathbb{R}^d;H)  $ by%
\begin{equation}\label{kernel}
	\left[\sigma\left(t\right)x\right]\left(\xi\right)=k_2\left(\xi-t\right)1_{\{t<\xi\}}x,
	\quad\text{}x\in\mathbb{R}^d,\,t,\xi\in\left[  0,T\right]  .
\end{equation}
For every $q\ge 2$, we denote by $\mathcal{H}^q$ the space $L^q\big(\Omega;H\big),$ endowed with the usual norm $\norm{\cdot}_{\mathcal{H}^q},$ and by $\mathcal{H}^q_t\subset \mathcal{H}^q$ the subspace of $\mathcal{F}_t-$measurable functions,  $t\in\left[0,T\right]$. Notice that 
\[
\norm{\sigma(t)}^2_{\text{HS}}\le d\norm{k_2}^2_2,\quad t\in[0,T],
\]
where $\norm{\cdot}_{\text{HS}}$ represents the Hilbert–Schmidt norm and $\norm{\cdot}_2$ the norm in $L^2(0,T;\mathbb{R})$.
As a consequence, since $\int_{0}^{T}\norm{\sigma(s)}^2_{\text{HS}}\dd s<\infty$, we can construct the stochastic integral 
\begin{equation}\label{serve_rem}
	\Sigma_{s,t}=\int_{s}^{t}\sigma\left(r\right)\dd W_r\in\mathcal{H}^q_t,\quad 0\le s \le t \le T.
\end{equation}
By \cite[Theorem $4.36$]{DP}, there exists a constant $C_{d,q}>0$ such that 
\begin{equation}\label{iso}
	\norm{\Sigma_{{s},t}}_{\mathcal{H}^q}\le C_{d,q}\norm{k_2}_2\sqrt{t-s},\quad 0\le s \le t \le T.
\end{equation}

Let $\Lambda$ be the space $C\left(  \left[  0,T\right]  ;\mathbb{R}^{d}\right)  $, and  define $B:\left[  0,T\right]  \times\Lambda\rightarrow H$  by%
\begin{equation}\label{B}
	\left[B\left(  t,w\right)\right]  \left(  \xi\right)  =k_{1}\left(  \xi-t\right)
	1_{\left\{   t<\xi\right\}  }b\left(  t,w\left(  t\right)  \right),
	\quad\text{}t,\xi\in\left[  0,T\right]  .
\end{equation}
In the sequel, a stochastic process taking values in $H$
will be denoted by, e.g., $\left(  w_{t}\right)  _{t\in\left[
	0,T\right]  }$, namely with the time variable as a subscript. Then, for a
fixed  $t_{0}\in\left[  0,T\right]  $, $w_{t_{0}}$ is a random function,
denoted by $w_{t_{0}}\left(  \xi\right)  $, $\xi\in\left[  0,T\right]  $.

In the following proposition, we show that it is possible to construct a solution to \eqref{SDE Hilbert}, i.e., an $\mathbb{F}-$adapted process with values in $H$ satisfying \eqref{SDE Hilbert} $\mathbb{P}-$a.s., for every $t\in [0,T]$,  using a 
solution of \eqref{SDE Volterra}.  
\begin{prop}\label{pr1}
	Let $X=\left(  X_{t}\right)  _{t\in\left[  0,T\right]  }$ be a solution
	of (\ref{SDE Volterra}). For every  $t\in\left[  0,T\right] $, define the $\mathbb{R}^d-$valued stochastic process $\theta_t=(\theta_t(\xi))_{\xi\in[t,T]}$ by 
	\[
	\theta_{t}\left(  \xi\right)  =x_{0}+\int_{0}^{t}k_{1}\left(  \xi-s\right)
	b\left(  s,X_{s}\right)  \dd s
	+\int_{0}^{t}k_{2}\left(  \xi-s\right)  \dd W_{s},\quad \xi\in [t,T].
	\]
	Define the $H-$valued stochastic process $\left(  w_{t}\right)
	_{t\in\left[  0,T\right]  }$ by setting, for each $t\in\left[  0,T\right]  $,%
	\begin{align}\label{mettiamola}
		w_{t}\left(  \xi\right) =\begin{cases}
			X_{\xi},&\xi\leq t,\\
			\theta_{t}\left(  \xi\right),&\xi>t.
		\end{cases} 
	\end{align}
	Then
	$\left(  w_{t}\right)  _{t\in\left[  0,T\right]  }$ is a solution of \eqref{SDE Hilbert} with $\phi\in H$ being the function identically equal to $x_{0}$.
\end{prop}

\begin{proof}
	Fix $t\in[0,T]$. Note that, by the Kolmogorov–Chentsov continuity criterion, there exists a continuous version of the stochastic process $(\int_{0}^{t}k_2(\xi-s)\,\dd W_s)_{\xi \in [t,T]}.$ Hence, also employing the dominated convergence theorem, we deduce that the process $\theta_t$ has continuous trajectories $\theta_t(\cdot)$ in $[t,T]$. It follows that $w_t$ defined in \eqref{mettiamola} takes values in $H$. \\In addition, by \cite[Proposition 3.18]{DP}, we observe that  $w_t$  is an $\mathcal{F}_t-$measurable random variable, because $X$ is continuous and $\mathbb{{F}}-$adapted, $\theta_t(\cdot)$ is continuous and $\theta_t(\xi)$ is $\mathcal{F}_t-$measurable for every $\xi \in [t,T]$. Thus, the $H-$valued stochastic process  $(w_t)_{t\in[0,T]}$  is $\mathbb{F}-$adapted.
	
	We now want to prove that $w_t$ satisfies \eqref{SDE Hilbert}. By \eqref{SDE Volterra} and the definition of $\theta_t$,  we have, $\mathbb{{P}}-$a.s.,
	\begin{align}\label{abstr_1}
		\notag w_{t}\left(  \xi\right)   & =X_\xi1_{\{\xi\le t\}}+ \theta_t(\xi)1_{\{\xi> t\}}=
		x_{0}+\int_{0}^{t\wedge\xi}k_{1}\left(
		\xi-s\right)  b\left(  s,X_{s}\right)  \dd s+\int_{0}^{t\wedge\xi}k_{2}\left(
		\xi-s\right)  \dd W_{s}\\
		& =x_{0}+\int_{0}^{t}k_{1}\left(  \xi-s\right)  1_{\left\{  \xi> s\right\}
		}b\left(  s,X_{s}\right)  \dd s
		+\int_{0}^{t}k_{2}\left(  \xi-s\right)
		1_{\left\{  \xi> s\right\}  }\dd W_{s},\quad \xi \in [0,T].
	\end{align}
	We focus on the integral in $\dd W$, with the aim of understanding its relation with $\Sigma_{0,t}=\int_{0}^{t}\sigma(s)\dd W_s$, see \eqref{serve_rem}. By \eqref{kernel} and  \cite[Proposition 4.30]{DP}, 
	\begin{equation*}
		\left\langle \int_{0}^{t}\sigma\left(s\right) \dd W_{s},h\right\rangle_{H}
		=
		\int_{0}^{t}\left(\int_{0}^{T} k_2(\xi-s)1_{\{\xi >s\}}h(\xi)\,\dd\xi\right)^\top\dd W_s
		,\quad \mathbb{P}-\text{a.s., for every }h\in H.
	\end{equation*}
	Moreover, an application of the stochastic Fubini's theorem yields
	\begin{align*}
		\left\langle \int_{0}^{t}k_2(\cdot- s)1_{\{\cdot>s\}}\dd W_s,\,h \right\rangle_{H}
		&=
		\int_{0}^{T}\left(\int_{0}^{t} k_2(\xi-s)1_{\{\xi >s\}}\,\dd W_s\right)^\top h(\xi)\,\dd \xi
		\\&=
		\int_{0}^{t}\left(\int_{0}^{T} k_2(\xi-s)1_{\{\xi >s\}}h(\xi)\,\dd\xi\right)^\top\!\!\dd W_s,\quad \mathbb{P}-\text{a.s., for every }h\in H.
	\end{align*}
	Considering that $H$ is separable, combining the two previous equations  we deduce that 
	\[
	\left\langle \int_{0}^{t}\sigma\left(s\right) \dd W_{s},h\right\rangle_{H}=\left\langle \int_{0}^{t}k_2(\cdot- s)1_{\{\cdot>s\}}\dd W_s,\,h \right\rangle_{H},\quad h\in {H},\,\mathbb{P}-\text{a.s.,}
	\]
	which in turn implies that 
	\begin{equation}\label{st_integral}
		\left(\int_{0}^{t}\sigma\left(s\right) \dd W_{s}\right)(\xi)=
		\int_{0}^{t}k_2(\xi- s)1_{\{\xi>s\}}\dd W_s
		,\quad \text{for a.e. }\xi \in [0,T],\, \mathbb{P}-\text{a.s.}
	\end{equation}
	Going back to \eqref{abstr_1}, recalling  the definition of $B$ in \eqref{B} and denoting by $\phi\in H$  the function identically equal to $x_0$, by the standard properties of Bochner's integral we conclude that
	\[
	w_t=
	\phi +\int_{0}^{t}B\left(  s,w_s\right) \dd s+\int_{0}^{t}\sigma(s)\,\dd W_{s},\quad \mathbb{P}-\text{a.s.}
	\]
	Therefore $(w_t)_{t\in [0,T]}$ satisfies \eqref{SDE Hilbert}, completing the proof.
\end{proof}
The previous proposition gives us the classical infinite dimensional reformulation of the Volterra SDE \eqref{SDE Volterra}, quoted by
Equation (\ref{SDE Hilbert}) in  Introduction \ref{intro}. However, for the procedure carried out in Section \ref{sec_abstract}, it turns out that a second reformulation is more convenient.
\begin{prop}\label{pro2}
	Let $\left(  X_{t}\right)  _{t\in\left[  0,T\right]  }$ be a solution
	of \eqref{SDE Volterra} and $\phi\in H$ be the function identically equal to $x_0$. Let $\theta_{t}\left(  \xi\right)  $ and
	$w_{t}\left(  \xi\right)  $ be defined as in Proposition \ref{pr1}. Then,
	for every $t\in\left[  0,T\right]  $, the following identity holds:
	\begin{equation}\label{abstr2}
		w_{t}=\phi+\int_{0}^{t}B\left(  s,w_{t}\right)  \dd s+\int_{0}^{t}\sigma(s)\,\dd W_{s},\quad \mathbb{P}-\text{a.s.}
	\end{equation}
\end{prop}
\begin{proof}
	Observing that, for a.e. $\xi \in [0,T]$,
	\[
	\int_{0}^{t}k_{1}\left(  \xi-s\right)  1_{\left\{  \xi> s\right\}
	}b\left(  s,X_{s}\right) \dd s=\int_{0}^{t}B\left(  s,w_{t}\right)  \left(
	\xi\right)  \dd s
	=\left[\int_{0}^{t}B\left(  s,w_{t}\right)    \dd s\right]\left(
	\xi\right),
	\]
	the proof is the same as the one of Proposition \ref{pr1},
\end{proof}
Motivated by the infinite dimensional reformulation  of Proposition \ref{pro2}, in Section \ref{sec_abstract} we focus on studying Equation \eqref{abstr2}. Our aim is to investigate the property of its solutions and the associated Kolmogorov equation, which is the subject of Section \ref{sec_KO}.  However, the implementation of this plan is challenging, due to the particular structure of the drift function $B\colon [0,T] \times \Lambda\to H$. More precisely, the issue with  the expression of $B$ in \eqref{B} is that it is meaningful only for continuous functions, as it involves a punctual evaluation. Consequently, unlike the classical case, the functional space $\Lambda$  in the domain of $B$ is different from the arrival Hilbert space $H$. This requires an abstract formulation of the problem that, to the best of our knowledge, is not covered by the existing literature.

\section{Abstract formulation and differentiability of the stochastic flow}\label{sec_abstract}
In this section, we introduce and study an abstract formulation for the equation \eqref{abstr2}, with a particular attention devoted to the differentiability of its solution with respect to the initial data, see Subsections \ref{first_order}-\ref{second_order}. In our reasoning, we introduce an extension of the drift operator $B$, denoted by $\widebar B$, which is a characterizing and original feature of the approach that we propose.

For every $k,\,p\in\mathbb{N}$, we denote by $\norm{\cdot}_{p}$ the usual norm on the Banach space $L^p\big(0,T;\mathbb{R}^k\big)$. We denote by  
\[
\text{$H_{\ts}$ the Hilbert space  $L^2\big(\left(0,T\right)\times \left(0,T\right);\mathbb{R}^d\big)$ endowed with the norm $\norm{\cdot}_{2,\ts}$.}
\]
  Recall  $H=L^2\big(0,T;\mathbb{R}^d\big)$ and  $\Lambda=C\left(\left[0,T\right];\mathbb{R}^d\right)$. For every $w\in\Lambda$, we consider a map $B\left(w\right)\colon\left[0,T\right]\times [0,T]\to \mathbb{R}^d$ subject to the next requirement.
\begin{ass}\label{B_1}
	The function $B\colon\Lambda\to H_{\ts}$ satisfies
	\begin{equation}\label{bound_B}
		\norm{B\left(w_1\right)}_{2,\ts}\le C_0\left(1+\norm{w_1}_2\right),\qquad
		\norm{B\left(w_1\right)-B\left(w_2\right)}_{2,\ts}\le C_0\norm{w_1-w_2}_2,
	\end{equation}
	for every $w_1,\,w_2\in \Lambda,$ for some constant $C_0=C_0\left(d,T\right)>0$. \\
	Moreover, given $w\in\Lambda$ and $0< t \le T$, for a.e. $r\in\left(0,t\right)$ the function $B\left(w\right)\left(r,\cdot\right)\in H$ is of Volterra--type, namely $B\left(w\right)\left(r,\xi\right)=0$ for a.e. $\xi\in\left(0,r\right)$, and depends on $w$ only via its restriction ${w}|_{\left(0,t\right)}$ to $\left(0,t\right)$.
\end{ass}
In the sequel, we are going to progressively introduce stricter hypotheses on the drift map $B$ (see, in particular, Assumptions \ref{B_2}-\ref{B_3}), which will  allow us to prove the main result on the Kolmogorov equation, see Theorem \ref{Kolm_thm} in Section \ref{sec_KO}. In Example \ref{ex1}, we show a function $B$, obtained by choosing $b$ in \eqref{B} with an affine structure, that satisfies these requirements.

Under Assumption \ref{B_1}, we can invoke the theorem of extension of uniformly continuous functions to uniquely define a continuous map $\widebar{B}\colon H\to H_{\ts}$ such that $\widebar{B}\big|_{\Lambda}=B$. Note that $\widebar{B}$ satisfies \eqref{bound_B} for every $w_1,\,w_2\in H$. Given $w\in H$ and $r\in\left(0,T\right)$, we are going to write $\widebar{B}\left(r,w\right)=\widebar{B}\left(w\right)\left(r,\cdot\right)\in H$: these maps are well defined for a.e. $r\in\left(0,T\right)$. \\
For a fixed $0<t\le T$, we remark that also $\widebar{B}\left(r,w\right)$ is of Volterra--type in the sense of Assumption \ref{B_1} for a.e. $r\in\left(0,t\right)$, and that it depends on $w$ only via $w|_{\left(0,t\right)}$. For these reasons, in the sequel we will refer to Assumption \ref{B_1} while talking about $\widebar{B}.$

Recall the spaces $\mathcal{H}^q=L^q\big(\Omega;H\big),\,q\ge 2$, and the subspaces $\mathcal{H}^q_t\subset \mathcal{H}^q$ of $\mathcal{F}_t-$measurable functions introduced in Section \ref{section abstract}, as well as  the random variables $\Sigma_{s,t}\in\mathcal{H}_t^q$ in \eqref{serve_rem}.
 For every $0\le s \le t \le T$ and $\phi\in \mathcal{H}^q$, we are interested in the  equation
	\begin{equation}\label{eq_B1}
	w=\phi+\int_{s}^{t}\widebar{B}\left(r, w\right)\dd r
	+
	\int_{s}^{t}\sigma\left(r\right)\dd W_r,
\end{equation}
whose well--posedness in $\mathcal{H}^q$ is given by the next result.
\begin{theorem}\label{well_pos}
	Under Assumption \ref{B_1}, for every $q\ge 2$, $\phi\in \mathcal{H}^q$ and $s,t\in\left[0,T\right]$, with $s\le t$, there exists a unique solution $w^{s,\phi}_t\in \mathcal{H}^q$ of \eqref{eq_B1}. In particular, if $\phi\in \mathcal{H}^q_s$ then $w^{s,\phi}_t\in \mathcal{H}^q_t$.

Furthermore, the following cocycle property holds in $\mathcal{H}^q$:
\begin{equation}\label{cocycle}
	w^{s,\phi}_t=w_t^{u,w^{s,\phi}_u},\quad 0\le s < u < t \le T,\,\phi\in \mathcal{H}^q.
\end{equation}
\end{theorem}
\begin{proof}
	Fix $q\ge 2$, $0\le s\le t\le T$ and $\phi\in \mathcal{H}^q_s$. 
%	We now prove the existence of such a solution to \eqref{eq_B1}, which will be denoted by $w^{s,\phi}_t$. 
Consider $N=N\left(d,T\right)\in \mathbb{N}$ so big that $C_0\sqrt{T/N}<1,$ where $C_0=C_0\left(d,T\right)$ is the constant in \eqref{bound_B}.  Let us introduce an equispaced partition $\{t_k\}_{k=0}^N$ of $\left[s,t\right]$ where $t_0=s$ and $t_N=t$: its mesh $\Delta\le T/N$. Define the mapping  $\Gamma_{t_0}^{t_1}\colon \mathcal{H}^q_{t_1}\to \mathcal{H}^q_{t_1}$ by 
	\begin{equation}\label{fix_point_B}
		\Gamma_{t_0}^{t_1} w=\phi+\int_{t_0}^{t_1} \widebar{B}\left(r, w\right)\dd r+
		\int_{t_0}^{t_1}\sigma\left(r\right)\dd W_r
		,\quad w\in \mathcal{H}^q_{t_1}.
	\end{equation}
	Under Assumption \ref{B_1}, $\Gamma_{t_0}^{t_1}$ is well defined. Indeed, for every $w\in \mathcal{H}^q_{t_1}$,
	\begin{align*}
		\norm{\Gamma_{t_0}^{t_1}w}_{\mathcal{H}^q}^q&=\mathbb{E}\left[\norm{\Gamma_{t_0}^{t_1}w}^q_2\right]\le3^{q-1}\mathbb{E}\left[ \norm{\phi}^q_2+\Bigg(\int_{t_0}^{t_1}\left(\int_{0}^{T}\left|\widebar{B}\left(w\right)\left(r,\xi\right)\right|^2\dd\xi\right)^{\frac{1}{2}}\dd r\Bigg)^q+\norm{\int_{t_0}^{t_1}\sigma\left(r\right)\dd W_r}_2^q\right]\\
		&\le 3^{q-1}\mathbb{E}\left[\norm{\phi}^q_2+C_0^q{\Delta}^\frac{q}{2}\left(1+\norm{w}_2\right)^q+\norm{\int_{t_0}^{t_1}\sigma\left(r\right)\dd W_r}_2^q\right]<\infty,
	\end{align*}
	where we use Bochner's theorem in the first inequality  and the first bound in \eqref{bound_B}, coupled with Jensen's inequality, in the second one. Analogously, using the second inequality in \eqref{bound_B}, we write
		\begin{multline}\label{contr_1}
	\norm{\Gamma_{t_0}^{t_1} w_1-\Gamma_{t_0}^{t_1} w_2}_{\mathcal{H}^q}\le \mathbb{E}\left[\Bigg(\int_{t_0}^{t_1}\left(\int_{0}^{T}\left|\widebar{B}\left(w_1\right)-\widebar{B}\left(w_2\right)\right|^2\left(r,\xi\right)\dd\xi\right)^{\frac{1}{2}}\dd r\Bigg)^q\right]^{\frac{1}{q}}\\
	\le C_0\sqrt{\Delta}\norm{w_1-w_2}_{\mathcal{H}^q},\quad w_1,\,w_2\in \mathcal{H}^q_{t_1}.
	\end{multline}
	Hence, for our choice of $N\in \mathbb{N}$, the map $\Gamma_{t_0}^{t_1}$ is a contraction in $\mathcal{H}^q_{t_1}$, whose unique fixed point is $\widebar{w}_1$. Noting that $\widebar{w}_1$ is the unique solution of \eqref{eq_B1} with $t_1$ instead of $t$, we denote it by ${w}^{s,\phi}_{t_1}$.

	Since the relation between constants in \eqref{contr_1}, which is necessary to make  $\Gamma_{t_0}^{t_1}$ a contraction, does not depend on the initial condition, under Assumption \ref{B_1} the previous argument can be iterated to construct the solution $w^{s,\phi}_{t}$ of \eqref{eq_B1}. More precisely, define the map $\Gamma_{t_1}^{t_2}\colon \mathcal{H}^q_{t_2}\to \mathcal{H}^q_{t_2}$ by
	\begin{equation*}
	\Gamma_{t_1}^{t_2} w=\widebar{w}_1+\int_{t_1}^{t_2} \widebar{B}\left(r, w\right)\dd r+\int_{t_1}^{t_2}\sigma\left(r\right)\dd W_r,\quad w\in \mathcal{H}^q_{t_2}.
	\end{equation*}
	Computations similar to those above show that $\Gamma_{t_1}^{t_2}$ is well defined. Moreover,
	\begin{equation*}
	\norm{\Gamma_{t_1}^{t_2} w_{1}-\Gamma_{t_1}^{t_2} w_{2}}_{\mathcal{H}^q}\le C_0\sqrt{\Delta}
	\norm{w_1-w_2}_{\mathcal{H}^q},\quad w_1,\,w_2\in \mathcal{H}^q_{t_2}.
	\end{equation*}
	Thus,  $\Gamma_{t_1}^{t_2}$ is a contraction in $\mathcal{H}^q_{t_2}$, whose unique fixed point is $\widebar{w}_{2}={w}^{t_1,{w}^{s,\phi}_{t_1}}_{t_2}$. 
	Now, by the Volterra--type property of $\widebar{B}$ and $\sigma$, together with the standard features of the Bochner's and stochastic integrals (see \eqref{st_integral}), we infer that 
	\begin{equation}\label{sotto}
		\left(\int_{t_1}^{t_2} \widebar{B}\left(r, \widebar{w}_{2}\right)\dd r\right)\left(\xi\right)= \left(\int_{t_1}^{t_2}\sigma\left(r\right)\dd W_r\right)\left(\xi\right)
		=0,\quad \text{for a.e. } \xi \in(0,t_1),\,\mathbb{P}-\text{a.s.,}
	\end{equation}
	whence 
	\[
		\restr{\widebar{w}_{2}}{\left(0,t_1\right)}=\restr{\widebar{w}_{1}}{\left(0,t_1\right)},\quad \mathbb{P}-\text{a.s.}
	\] 
	Furthermore, $\mathbb{P}-$a.s., for a.e. $r\in\left(s,t_1\right)$, $\widebar{B}\left(r, \widebar{w}_{1}\right)$ depends on $\widebar{w}_{1}$ only via $\restr{\widebar{w}_{1}}{\left(0,r\right)}$, which yields
	\begin{equation}\label{sotto1}
		\widebar{B}\left(r, \widebar{w}_{1}\right)=\widebar{B}\left(r, \widebar{w}_{2}\right),\quad \text{for a.e. }r\in (s,t_1),\,\mathbb{P}-\text{a.s.}
	\end{equation}
	Therefore, recalling \eqref{fix_point_B},
	\begin{equation}\label{mett}
	\widebar{w}_{2}=\phi+\int_{s}^{t_1}\widebar{B}\left(r,\widebar{w}_{1}\right)\dd r+ \int_{t_1}^{t_2}\widebar{B}\left(r,\widebar{w}_{2}\right)\dd r
	+\int_{s}^{t_2}\sigma\left(r\right)\dd W_r
	=\phi+\int_{s}^{t_2}\widebar{B}\left(r,\widebar {w}_{2}\right)\dd r+\int_{s}^{t_2}\sigma\left(r\right)\dd W_r.
	\end{equation}
	This shows that $\widebar{w}_{2}$ is a solution of \eqref{eq_B1} with $t_2$ instead of $t$.	\\ 
	To prove that $\widebar{w}_2$ is in fact the unique solution of this equation, we consider another random variable $\widetilde{w}\in\mathcal{H}^q_{t_2}$ satisfying \eqref{mett}. Then, relying on the same properties of $\widebar{B}$ and $\sigma$ as those used above, we deduce that
	\begin{equation}\label{uni}
		1_{(0,t_1)}\widetilde{w}=1_{(0,t_1)}\left(\phi+\int_{s}^{t_1}\widebar{B}\left(r,1_{(0,t_1)}\widetilde{w}\right)\dd r
		+\int_{s}^{t_1}\sigma\left(r\right)\dd W_r\right).
	\end{equation}
	  Moreover, we observe that also $1_{(0,t_1)}\widebar{w}_1\in\mathcal{H}^q$ satisfies \eqref{uni}. Therefore, using Bochner's theorem and Jensen's inequality, by Assumption \ref{B_1} we can compute
	  	\begin{align*}
	  	\norm{1_{(0,t_1)}\left(\widebar{w}_1-\widetilde{w}\right)}_{\mathcal{H}^q}^q&\le 	\mathbb{E}\left[\norm{\int_{s}^{t_1}\left(\widebar{B}\left(r, 1_{(0,t_1)}\widebar{w}_1\right)-\widebar{B}\left(r, 1_{(0,t_1)}\widetilde{w}\right)\right)\dd r}_2^q\right]
	  			\\&\le 
	  	\mathbb{E}\left[\left(\int_{s}^{t_1}\norm{\widebar{B}\left(r, 1_{(0,t_1)}\widebar{w}_1\right)-\widebar{B}\left(r, 1_{(0,t_1)}\widetilde{w}\right)}_2\dd r\right)^q\right]
	  	\\
	  	%=
	  	%\mathbb{E}\left[\Bigg(\int_{s}^{t_1}\left(\int_{0}^{T}\left|\widebar{B}\left(1_{(0,t_1)}w_1\right)-\widebar{B}\left(w_2\right)\right|^2\left(r,\xi\right)\dd\xi\right)^{\frac{1}{2}}\dd r\Bigg)^q\right]
	  	&\le 
	  	\Delta^{\frac{q}{2}}\mathbb{E}\left[\norm{\widebar{B}\left(1_{(0,t_1)}\widebar{w}_1\right)-\widebar{B}\left(1_{(0,t_1)}\widetilde{w}\right)}^q_{2,\ts}\right]\le \Delta^{\frac{q}{2}}C_0^q\norm{1_{(0,t_1)}\left(\widebar{w}_1-\widetilde{w}\right)}_{\mathcal{H}^q}^q,
	  \end{align*} 
	 which allow us to conclude, recalling that $\sqrt{\Delta}C_0<1$,
	  \[
	  	1_{(0,t_1)}\widetilde{w}=1_{(0,t_1)}\widebar{w}_1,\quad \mathbb{P}-\text{a.s.}
	  \]
	  Going back to \eqref{mett}, by \eqref{fix_point_B} and the previous equality we have, $\mathbb{P}-$a.s.,
	  	\begin{equation*}
	  	\widetilde{w}=\phi+\int_{s}^{t_1}\widebar{B}\left(r,\widebar{w}_{1}\right)\dd r
	  	+\int_{s}^{t_1}\sigma\left(r\right)\dd W_r
	  	+ \int_{t_1}^{t_2}\widebar{B}\left(r,\widetilde{w}\right)\dd r
	  	+\Sigma_{t_1,t_2}
	  	=
	  	\widebar{w}_1+ \int_{t_1}^{t_2}\widebar{B}\left(r,\widetilde{w}\right)\dd r
	  	+\Sigma_{t_1,t_2}.
	  	\end{equation*}
  	It follows that $\widetilde{w}$ is a fixed point of the map $\Gamma_{t_1}^{t_2}$ in $\mathcal{H}_{t_2}^{q}$: by uniqueness, we obtain $\widetilde{w}=\widebar{w}_2$. Hence $\widebar{w}_2$ is the unique solution of \eqref{eq_B1} with $t_2$ instead of $t$, which we denote by $w_{t_2}^{s,\phi}$.
  	
	This argument by steps can be repeated to cover the whole interval $\left[s,t\right]$. In this way, we obtain the unique solution $w^{s,\phi}_t$of  \eqref{eq_B1} in $\mathcal{H}^q_t$. The same procedure also works when the initial condition $\phi\in \mathcal{H}^q$, i.e., when $\phi$ is not necessarily $\mathcal{F}_s-$measurable. In such a case, it provides a unique solution $w_{t}^{s,\phi}\in\mathcal{H}^q$.
	
	The cocycle property in \eqref{cocycle} follows by a similar reasoning. Indeed, if we fix $u\in\left(s,t\right)$,  then  by the Volterra--type property of $\widebar{B}$ and $\sigma$ (cfr. \eqref{sotto}) we have 
	\begin{equation}\label{flow1}
		\restr{w^{u,w^{s,\phi}_u}_t}{\left(0,u\right)}=\restr{w^{s,\phi}_u}{\left(0,u\right)},\quad \mathbb{P}-\text{a.s.}
				\end{equation}
Invoking again Assumption \ref{B_1} as in \eqref{sotto1},
	\begin{multline*}
		w_t^{u,w^{s,\phi}_u}=
		\phi+\int_{s}^{u}\widebar{B}\left(r, w^{s,\phi}_u\right)\dd r+\int_{u}^{t}\widebar{B}\left(r,w_t^{u,w^{s,\phi}_u}\right) \dd r+\int_{s}^{t}\sigma\left(r\right)\dd W_r
		\\
		=\phi+\int_{s}^{t}\widebar{B}\left(r,w_t^{u,w^{s,\phi}_u}\right) \dd r+\int_{s}^{t}\sigma\left(r\right)\dd W_r,
	\end{multline*}
	hence the equality in \eqref{cocycle} is inferred by the  uniqueness of the solution of \eqref{eq_B1}.
	The proof is now complete.
\end{proof}

\begin{rem}\label{classical}
	The cocycle property in \eqref{cocycle} (see also \eqref{flow1}) yields  $w^{s,\phi}_t\left(\xi\right)=w^{s,\phi}_u\left(\xi\right)$ for a.e. $\xi\in\left(0,u\right)$, $\mathbb{P}-$a.s., for every $0\le s\le u \le t \le T$ and $\phi\in \mathcal{H}^q,\,q\ge2$. 
	%Therefore, if we define $w^{s,\phi}\colon \left[s,T\right]\to \mathcal{H}^q$ by $w^{s,\phi}\left(t\right)=w^{s,\phi}_t,\,t\in\left[s,T\right],$ then this mapping solves the following SPDE in the strong sense (cfr. \eqref{SDE Hilbert}):
	%	\begin{equation*}
		%	\begin{cases}
			%	\dd {w^{}}\left(t\right)=\widebar{B}\left(t,w^{}\left(t\right)\right)\dd t+\sigma\left(t\right)\dd W_t,& t\in\left[s,T\right],\\
			%			w\left(s\right)=\phi\in \mathcal{H}^q_s.
			%	\end{cases}
		%\end{equation*}
	\end{rem}

\begin{rem}\label{rem_p}
	For every  $p\in(2,(1-\alpha)ì^{-1})$, the fractional kernel $k_2$ in \eqref{k_f} belongs to the space $L^p\big(0,T;\mathbb{R}\big)$. 
	
	 As a consequence, according to \cite[Lemma $8.27$, Theorem $8.29$]{PZ}, the stochastic integral $\Sigma_{s,t}$ in \eqref{serve_rem} belongs to the space
	  $$\mathcal{L}_t^p= \left(L_t^p\left(\Omega;L^p\right),\norm{\cdot}_{\mathcal{L}^p}\right),\quad \text{ where }\quad L^p=L^p\big(0,T;\mathbb{R}^d\big).$$ 
As before, the subscript $t$ in the previous expression indicates a space of $\mathcal{F}_t-$measurable random variables. Moreover, the following inequality holds (cfr. \eqref{iso}):
	\begin{equation}\label{PZbanach}
		\norm{\Sigma_{s,t}}_{\mathcal{L}^p}\le C_{d,p}\norm{k_2}_p\sqrt{t-s},\quad \text{for some }C_{d,p}>0.
	\end{equation}
	
	We denote by $$\text{$L^p_{\ts}$ the Banach space  $L^p\big(\left(0,T\right)\times \left(0,T\right);\mathbb{R}^d\big)$, endowed with the  norm $\norm{\cdot}_{p,\ts}$}.$$ In addition to Assumption \ref{B_1}, suppose that $B\colon\Lambda\to L^p_{\ts}$  and that it satisfies
		\begin{equation}\label{ass_b_1'}
		\norm{B\left(w_1\right)}_{p,\ts}\le C_{0,p}\left(1+\norm{w_1}_p\right),\qquad
		\norm{B\left(w_1\right)-B\left(w_2\right)}_{p,\ts}\le C_{0,p}\norm{w_1-w_2}_p,
	\end{equation}
for every $w_1,\,w_2\in \Lambda,$ for some constant $C_{0,p}=C_{0,p}(d,T)>0$. 
 Note that $\widebar{B}\colon H\to H_{\ts}$ satisfies \eqref{ass_b_1'} for every $w_1,\,w_2\in L^p$. 
 
 In this framework, one can argue as in the proof of Theorem \ref{well_pos} to infer that, for every $\phi\in \mathcal{L}_s^p$, there exists a unique solution $w^{s,\phi}_t$ of \eqref{eq_B1} belonging to the space $\mathcal{L}_t^p$.
\end{rem}

The following corollary to Theorem \ref{well_pos} gives a Lipschitz–type dependence of the solution $w^{s,\phi}_t$ of \eqref{eq_B1} on the initial condition $\phi$, which combined with \eqref{cocycle} allows to prove the $\mathbb{F}-$Markov property of the process $(w^{s,\phi}_t)_{t\in[s,T]}$.
\begin{cor}\label{lip}
Let $q\ge 2$. Under Assumption \ref{B_1},  there exists a constant $C_1=C_1\left(d, q,T\right)>0$ such that, for every $0\le s <t\le T$,
	 \begin{equation}\label{initial_c}
		\norm{w^{s,\phi}_t-w^{s,\psi}_t}_{\mathcal{H}^q}\le C_1 \norm{\phi-\psi}_{\mathcal{H}^q},\quad \phi,\,\psi\in \mathcal{H}^q.
	\end{equation}
In addition, for all   $s\in [0,T]$ and  $\phi\in \mathcal{H}^q_s$,  the process $(w^{s,\phi}_t)_{t\in[s,T]}$ is $\mathbb{F}-$Markov, and 
\begin{equation}\label{markoviala}
	\mathbb{E}\left[\Phi\left(w^{s,\phi}_u\right) \Big| \mathcal{F}_t\right]
	=
	\restr{\mathbb{E}\left[\Phi\left(w_u^{t,\psi}\right)\right]}{\psi=w^{s,\phi}_t}
	,\quad \mathbb{P}-\text{a.s., }s\le t\le u\le T,\,\Phi\in\mathcal{B}_b(H),
\end{equation} 
where $\mathcal{B}_b(H)$ denotes the space of bounded Borel measurable functions from $H$ to $\mathbb{R}$.
\end{cor}
\begin{proof}
 Fix $q\ge 2$, $0\le s < t\le T$ and consider $N=N\left(d,T\right)\in \mathbb{N}$ so big that $2C_0\sqrt{T/N}<{2}^{1/q},$ where $C_0=C_0\left(d,T\right)$ is the constant in \eqref{bound_B}.  Moreover, take an equispaced partition $\{t_k\}_{k=0}^N$ of $\left[s,t\right]$ where $t_0=s$ and $t_N=t$. By \eqref{bound_B}-\eqref{eq_B1}, for every $\phi,\,\psi\in \mathcal{H}^q,$
\begin{multline*}
\norm{w^{s,\phi}_{t_1}-w^{s,\psi}_{t_1}}^q_2\le 2^{q-1}\norm{\phi-\psi}^q_2+2^{q-1}\left(\frac{T}{N}\right)^{\frac{q}{2}}\norm{\widebar{B}\left(w^{s,\phi}_{t_1}\right)-\widebar{B}\left(w^{s,\psi}_{t_1}\right)}^q_{2,\ts}\\
\le 2^{q-1}\norm{\phi-\psi}^q_2+2^{q-1}C_0^q\left(\frac{T}{N}\right)^{\frac{q}{2}}\norm{w^{s,\phi}_{t_1}-w^{s,\psi}_{t_1}}^q_{2},\quad \text{$\mathbb{P}-$a.s.,}
\end{multline*}
hence 
\begin{equation*}
	\norm{w^{s,\phi}_{t_1}-w^{s,\psi}_{t_1}}^q_2\le 
	2^{q-1}\left(1-2^{q-1}C_0^q\left(\frac{T}{N}\right)^{\frac{q}{2}}\right)^{-1}\norm{\phi-\psi}^q_2,\quad \mathbb{P}-\text{a.s.}
\end{equation*}
Thus, by the cocycle property in \eqref{cocycle}, for every $\phi,\,\psi\in \mathcal{H}^q,$
\begin{align*}
	\norm{w^{s,\phi}_t-w^{s,\psi}_t}^q_2&=
	\norm{w^{t_{N-1},w^{s,\phi}_{t_{N-1}}}_{t_{N}}-w^{t_{N-1},w^{s,\psi}_{t_{N-1}}}_{t_N}}^q_2
	\\&\le 2^{q-1}\left(1-2^{q-1}C_0^q\left(\frac{T}{N}\right)^{\frac{q}{2}}\right)^{-1}\norm{w^{t_{N-2},w^{s,\phi}_{t_{N-2}}}_{t_{N-1}}-w^{t_{N-2},w^{s,\psi}_{t_{N-2}}}_{t_{N-1}}}^q_2\\
	&\le  
	2^{N\left(q-1\right)}\left(1-2^{q-1}C_0^q\left(\frac{T}{N}\right)^{\frac{q}{2}}\right)^{-N}\norm{\phi-\psi}^q_2,\quad\text{$\mathbb{P}-$a.s.,} 
\end{align*}
which shows \eqref{initial_c} upon taking expectations and $q-$th root, as desired.

The Markov property of the process $(w^{s,\phi}_t)_{t\in[s,T]},\,\phi\in \mathcal{H}^q_s,$ is a consequence of \eqref{markoviala}. In turn, the equality in \eqref{markoviala} can be readily obtained by paralleling the monotone class argument in \cite[Theorem $9.14$]{DP}, which essentially relies on the cocycle property in \eqref{cocycle} and the Lipschitz--continuous dependence in \eqref{initial_c}. Thus, the proof is complete.
\end{proof}
\subsection{First–order differentiability in the initial data}\label{first_order}
In this subsection we focus on deterministic initial conditions for \eqref{eq_B1}, i.e., $\phi\in H$. From now on, we denote the Hilbert space $\mathcal{H}^2=L^2(\Omega;H)$ simply by $\mathcal{H}$. 

In order to study the first–order Fréchet differentiability  of $w^{s,\phi}_t$ in ${H}$, we require hypotheses on $B$ which are stronger than Assumption \ref{B_1}. In fact, we  need some  conditions on the Frechét differentiability of $B$ in the normed space $(\Lambda,\norm{\cdot}_2)$. In the sequel, we write $\Lambda_2$ for $(\Lambda,\norm{\cdot}_2)$ to have a compact notation.
\begin{ass}\label{B_2}
	The map $B\colon \Lambda\to H_{\ts}$ satisfies Assumption \ref{B_1}. Moreover, $B$ is $\Lambda_2-$Fréchet differentiable, and there exists a constant $C_0=C_0\left(d,T\right)>0$ such that
\begin{equation}\label{direct}
	\norm{DB\left(w_1\right)\left(w_2\right)}_{2,{\ts}}\le C_0\norm{w_2}_2,\quad w_1,w_2\in \Lambda, 
\end{equation}	
and 
\begin{equation}\label{ho_alpha}
	\norm{DB\left(w_1\right)-DB\left(w_2\right)}_{\mathcal{L}\left(\Lambda_2;H_{\ts}\right)}\le C_0\norm{w_1-w_2}_2^{\gamma},\quad w_1,w_2\in \Lambda,\text{ for some }\gamma\in\left(0,1\right].
\end{equation}
\end{ass}
Without loss of generality, we assume the constant $C_0$ in \eqref{direct}-\eqref{ho_alpha} to be the same as the one in \eqref{bound_B}.\\ Under Assumption \ref{B_2}, precisely by \eqref{direct} and the theorem of extension of uniformly continuous functions, for every $w_1\in\Lambda$ it is possible to extend $DB\left(w_1\right)\in\mathcal{L}\left(\Lambda; H_{\ts}\right)$ to an operator $\widebar {DB}\left(w_1\right)\in  \mathcal{L}\left(H; H_{\ts}\right)$ satisfying \eqref{direct} for all $w_2\in H$. Moreover, by \eqref{ho_alpha},
\begin{equation}\label{ho_alpha1}
	\norm{\widebar{DB}\left(w_1\right)-\widebar{DB}\left(w_2\right)}_{\mathcal{L}\left(H;H_{\ts}\right)}=
	\norm{{DB}\left(w_1\right)-{DB}\left(w_2\right)}_{\mathcal{L}\left(\Lambda_2;H_{\ts}\right)}\le C_0\norm{w_1-w_2}_2^{\gamma},\quad w_1,w_2\in \Lambda,
\end{equation}
	hence we can extend (without changing the notation) 
	\begin{align}\label{meglio}
	\text{$\widebar{DB}\colon H \to \mathcal{L}\left(H;H_{\ts}\right)$, with $\widebar{DB}$ satisfying (\ref{direct})-(\ref{ho_alpha1}) for every $w_1,w_2\in H$.}
	\end{align}
	We  want to show that  $\widebar{B}$ is $H-$Fréchet differentiable, with $D\widebar{B}=\widebar{DB}.$ 
	By Taylor's formula applied on $B$, recalling that $\restr{\widebar{B}}{\Lambda}=B$ and $\restr{\widebar{DB}(w_1)}{\Lambda}=DB(w_1),\,w_1\in\Lambda$, we write
	\begin{align}
		\label{save}
		&\widebar{B}\left(w_2\right)-\widebar{B}\left(w_1\right)-\widebar{DB}\left(w_1\right)\left(w_2-w_1\right)=r\left(w_1,w_2\right), \quad  w_1,w_2\in\Lambda,\text{ where}\\& 
		r\left(x,y\right)= \notag\int_{0}^{1}\left(\widebar{DB}\left(x+h\left(y-x\right)\right)-\widebar{DB}\left(x\right)\right)\left(y-x\right)\dd h,\quad x,\,y\in H. 
	\end{align}
Note that $r\colon H\times H\to H_{\ts}$ is continuous. Indeed, for every $x,y \in H$ and every sequence $H\times H\ni \left(x_n,y_n\right)\to \left(x,y\right)$ as $n\to \infty$, by Bochner's theorem and \eqref{meglio}, after some algebraic computations we deduce that
\begin{multline*}
	\norm{r\left(x_n,y_n\right)-r\left(x,y\right)}_{2,\ts}\le
	\frac{C_0}{\gamma+1} \norm{y_n-x_n}^{\gamma}_{2}\norm{y_n-y+x-x_n}_2
	\\+C_0\left(\frac{1}{\gamma+1}\norm{y_n-y+x-x_n}^{\gamma}_2+2\norm{x-x_n}_2^{\gamma}\right)\norm{y-x}_2\to0,\quad\text{as }n\to\infty.
\end{multline*}
It then follows from the continuity of $\widebar{B}$ in $H$ and \eqref{meglio} that  \eqref{save} holds for every $w_1,w_2\in H$.
Moreover, since by \eqref{meglio} $\norm{r\left(x,y\right)}_{2,\ts}\le C_0(\gamma+1)^{-1}\norm{y-x}^{1+\gamma}_2,\,x,\,y\in H$, we conclude that 
\[
	\widebar{B}\left(w_2\right)-\widebar{B}\left(w_1\right)-\widebar{DB}\left(w_1\right)\left(w_2-w_1\right)=\text{o}\left(\norm{w_2-w_1}_2\right),\quad w_1,w_2\in H.
\]
Therefore $\widebar{B}$ is $H-$Fréchet differentiable, with $D\widebar{B}=\widebar{DB}.$ 

We also notice that, for every $w_1,\,w_2\in H$ and $0<t\le T$, 
\begin{equation}\label{VO_D}
	\text{$D\widebar{B}(w_1)\left(r,w_2\right)\coloneqq [D\widebar{B}(w_1)(w_2)]\left(r,\cdot\right)\in H$ is of Volterra--type, for a.e. $r\in\left(0,t\right)$},
\end{equation}
and that 
\begin{equation}\label{D_P_D}
	\text{$D\widebar{B}(w_1)\left(r,w_2\right)$ depends on $w_i$ only via $w_i\big|_{\left(0,t\right)},\,i=1,2$, for a.e. $r\in\left(0,t\right)$}:
\end{equation}
these two properties are inherited from $\widebar{B}$, see Assumption~\ref{B_1}.

The next result shows that, under Assumption \ref{B_2}, the solution $w_t^{s,\phi}$  of \eqref{eq_B1}, considered as a map from $H$ to $\mathcal{H}$, is $H-$Fréchet differentiable.
\begin{comment}
{
We need a straightforward lemma on the mapping  $\Gamma_s^t\colon \mathcal{H}\times \mathcal{H}\to \mathcal{H},\,0\le s\le t \le T$, defined by 
\[
\Gamma_s^t\left(\phi, w\right)=\phi+\int_{s}^{t}\widebar{B}\left(r,w\right)\dd r+\int_{s}^{t}\sigma\left(r\right)\dd W_r,\quad \phi,\,w\in \mathcal{H},
\]
whose proof is omitted.
\begin{lemma}\label{diff_1}
	Under Assumption \ref{B_2}, for every $0\le s \le t \le T$, the mapping $\Gamma_s^t \in C^1\left(\mathcal{H}\times \mathcal{H};\mathcal{H}\right)$, with
	\begin{equation*}
		D_\phi\Gamma_s^t\left(\phi,w\right)\psi=\psi,\qquad 
		D_w\Gamma_s^t\left(\phi,w\right)w_0=\int_{s}^{t}D\widebar{B}\left(w\right)\left(r,w_0\right)\,\dd r,
\end{equation*}
for all $\phi,\,\psi,\,w,\,w_0\in \mathcal{H}$.
\end{lemma}}
\end{comment}

\begin{theorem}\label{diff_w_thm}
	Under Assumption \ref{B_2}, for every $0\le s\le t \le T$, the mapping $w^{s,\cdot}_t\in C^{1+\gamma}\left({H};\mathcal{H}\right)$. In particular, for every $\phi,\,\psi\in H$,  $Dw^{s,\phi}_{t} \psi$ is the unique solution in $\mathcal{H}$ of the following equation:
	\begin{equation}\label{frec_w}
		Dw^{s,\phi}_{t} \psi=\psi+ \int_{s}^{t}D\widebar{B}\left(w^{s,\phi}_t\right)\left(r,Dw_t^{s,\phi}\psi\right)\dd r.
	\end{equation} 
Furthermore, there exists a constant $C_2=C_2(d,T)>0$ such that, for every $\phi,\,\psi,\,\eta\in {H}$, $\mathbb{P}-\text{a.s.},$
\begin{equation}\label{alpha-h}
	\norm{Dw^{s,\phi}_t\eta}_{2}\le C_2\norm{\eta}_2,\qquad 
	\norm{Dw^{s,\phi}_t\eta-Dw^{s,\psi}_t\eta}_{2}\le C_2\norm{w^{s,\phi}_t-w^{s,\psi}_t}_{2}^{\gamma}\norm{\eta}_2. 
\end{equation}
\end{theorem}
\begin{proof}
Fix $0\le s \le t\le T$ and $\phi\in H$. Firstly, we prove the well–posedness in $\mathcal{H}$ of the equation
\begin{equation}\label{claro}
		w=\psi+ \int_{s}^{t}D\widebar{B}\left(w^{s,\phi}_t\right)\left(r,w\right)\dd r,\quad \psi\in H.
\end{equation}
Consider $N=N\left(d,T\right)\in \mathbb{N}$ so big that $C_0\sqrt{T/N}<1,$ where $C_0=C_0\left(d,T\right)$ is the constant in Assumptions \ref{B_1}-\ref{B_2}. In addition, take an equispaced partition $\{t_k\}_{k=0}^N$ of $\left[s,t\right]$ where $t_0=s$ and $t_N=t$: its mesh $\Delta\le T/N$. By \eqref{meglio} (see also \eqref{direct}) and Bochner's theorem, the following estimate holds:
\begin{multline}\label{claro1}
	\norm{\int_{t_0}^{t_1}D\widebar{B}\left(w^{s,\phi}_t\right)\left(w_1-w_2\right)\left(r,\cdot\right)\dd r}_\mathcal{H}\le
	\sqrt{\Delta}\mathbb{E}\left[\int_{t_0}^{t_1}\int_{0}^{T}\left|D\widebar{B}\left(w^{s,\phi}_t\right)\left(w_1-w_2\right)\right|^2\left(r,\xi\right)\dd \xi\,\dd r\right]^{\frac{1}{2}}
	\\
	\le \sqrt{\Delta}\mathbb{E}\left[\norm{D\widebar{B}\left(w^{s,\phi}_t\right)\left(w_1-w_2\right)}_{2,\ts}^2\right]^{\frac{1}{2}}
	\le  C_0\sqrt{\Delta}\norm{w_1-w_2}_\mathcal{H},\quad w_1,w_2\in \mathcal{H}.
\end{multline}
Thus, employing a fixed point argument as in the proof of Theorem \ref{well_pos}, we deduce the existence of a unique solution $\widebar{w}^{\psi}_1\in\mathcal{H}$ of \eqref{claro} with $t_1$ instead of $t$, for every $\psi\in H$. 

We claim that the operator $Dw^{s,\phi}_{t_1}\colon H\to \mathcal{H}$ defined by $Dw^{s,\phi}_{t_1}\psi=\widebar{w}_1^{\psi},\,\psi\in H,$ is the Fréchet differential of $w^{s,\phi}_{t_1}$. Indeed, the linearity of  $Dw^{s,\phi}_{t_1}$ is straightforward, while the continuity is ensured by the following computation, which can be argued from \eqref{claro} similarly to \eqref{claro1}: 
\begin{equation}\label{bound_derw_ste1}
	\norm{Dw^{s,\phi}_{t_1}\psi}_{2}
	\le \left(1-C_0\sqrt{T/N}\right)^{-1}\norm{\psi}_2,	\quad \mathbb{P}-\text{a.s., }\psi\in {H}.
\end{equation}
Moreover,  recalling \eqref{eq_B1}-\eqref{frec_w},
\begin{align}\label{ma1}
	\notag&\norm{w^{s,\phi+h}_{t_1}-w^{s,\phi}_{t_1}-Dw^{s,\phi}_{t_1}h}_\mathcal{H}
	%=\norm{\int_{s}^{t_1}\left(\widebar{B}\left(w^{s,\phi+h}_{t_1}\right)-\widebar{B}\left(w^{s,\phi}_{t_1}\right)-D\widebar{B}\left(w^{s,\phi}_{t_1}\right)Dw^{s,\phi}_{t_1}h\right)\left(r,\cdot\right)\dd r}_\mathcal{H}\\&
\le \sqrt{\Delta}\mathbb{E}\left[\norm{\widebar{B}\left(w^{s,\phi+h}_{t_1}\right)-\widebar{B}\left(w^{s,\phi}_{t_1}\right)-D\widebar{B}\left(w^{s,\phi}_{t_1}\right)Dw^{s,\phi}_{t_1}h}^2_{2,\ts}\right]^{\frac{1}{2}}
	\\\notag&\qquad\le \sqrt{T/N}\left(\mathbb{E}\left[\norm{D\widebar{B}\left(w^{s,\phi}_{t_1}\right)
	\left(w^{s,\phi+h}_{t_1}-w^{s,\phi}_{t_1}-Dw^{s,\phi}_{t_1}h\right)}_{2,\ts}^2\right]^{\frac{1}{2}}
\right.\\\notag&\qquad\qquad  \qquad \quad+\left.
\mathbb{E}\left[\norm{\int_{0}^{1}\left(D\widebar{B}\left(w^{s,\phi}_{t_1}+u\left(w^{s,\phi+h}_{t_1}-w^{s,\phi}_{t_1}\right)\right)-D\widebar{B}\left(w^{s,\phi}_{t_1}\right)\right)
	\left(w^{s,\phi+h}_{t_1}-w^{s,\phi}_{t_1}\right)\dd u}_{2,\ts}^2\right]^{\frac{1}{2}}
	\right)
	\\&\qquad\le \sqrt{T/N}C_0\left(\norm{w^{s,\phi+h}_{t_1}-w^{s,\phi}_{t_1}-Dw^{s,\phi}_{t_1}h}_\mathcal{H}
	+\mathbb{E}\left[\norm{w^{s,\phi+h}_{t_1}-w^{s,\phi}_{t_1}}_2^{2\left(1+\gamma\right)}\right]^{\frac{1}{2}}
	\right),\quad h\in H,
\end{align}
where we apply Taylor's formula on $\widebar{B}$ for the second inequality and \eqref{meglio} together with Bochner's theorem for the third. Notice that $H\subset \mathcal{H}^q$ for every $q\ge 2$. Therefore, by  Corollary \ref{lip} with $q=2(1+\gamma),$ from \eqref{ma1} we infer that 
\begin{equation}\label{ma2}
\norm{w^{s,\phi+h}_{t_1}-w^{s,\phi}_{t_1}-Dw^{s,\phi}_{t_1}h}_\mathcal{H}\le \sqrt{T/N} C_0C_{1}^{1+\gamma}\left(1-\sqrt{T/N}C_0\right)^{-1}
\norm{h}_2^{1+\gamma}
=\text{o}\left(\norm{h}_2\right),\quad h\in H,
\end{equation}
for some constant $C_1=C_1(\gamma,d,T)>0$. 
This shows that $Dw^{s,\phi}_{t_1}$ is the Fréchet differential of $w^{s,\phi}_{t_1}$, as desired.

Next,  consider  
\begin{equation}\label{add}
	w
	=Dw_{t_1}^{s,\phi}\psi+ \int_{t_1}^{t_2}D\widebar{B}\left(w^{s,\phi}_{t_2}\right)\left(r,	w\right)\dd r,\quad \psi \in H:
\end{equation}
the well--posedness of this equation in $\mathcal{H}$ can be obtained via a fixed–point argument as in the above step.  We denote  by $\widebar{w}_2^{\psi}\in\mathcal{H},\,\psi\in H$, the unique solution of \eqref{add}. 

We argue that $\widebar{w}_2^\psi$ is the unique solution of \eqref{claro} with $t_2$ instead of $t$, for every $\psi\in H$. 
By the Volterra–type property of $D\widebar{B}$ in \eqref{VO_D} and \eqref{add} we have,  $\mathbb{P}-$a.s.,
\[
\restr{\widebar{w}_2^{\psi}}{\left(0,t_1\right)}=\restr{Dw^{s,\phi}_{t_1}\psi}{\left(0,t_1\right)}.
\]
Furthermore, thanks to the relation $w^{s,\phi}_{t_2}=w^{t_1,w^{s,\phi}_{t_1}}_{t_2}$  in \eqref{cocycle} and the properties of $\widebar{B}$ under Assumption \ref{B_1} we can write, $\mathbb{P}-$a.s.,
\begin{equation}\label{wow}
	\restr{w^{s,\phi}_{t_2}}{\left(0,t_1\right)}=\restr{w^{s,\phi}_{t_1}}{\left(0,t_1\right)},
\end{equation}
see Remark \ref{classical}. 
Consequently, by the property of $D\widebar{B}$ in \eqref{D_P_D} and recalling  that $Dw_{t_1}^{s,\phi}\psi$ satisfies \eqref{claro} with $t_1$ instead of $t$, from \eqref{add} we conclude that, $\mathbb{P}-$a.s., 
\begin{multline}\label{dai}
\widebar{w}_2^{\psi}
	=\psi+\int_{s}^{t_1}D\widebar{B}\left(w^{s,\phi}_{t_1}\right)\left(r,Dw_{t_1}^{s,\phi}\psi\right)\dd r
	+\int_{t_1}^{t_2}D\widebar{B}\left(w^{s,\phi}_{t_2}\right)\left(r,\widebar{w}_2^{\psi}\right)\dd r
	\\=
	\psi
	+\int_{s}^{t_2}D\widebar{B}\left(w^{s,\phi}_{t_2}\right)\left(r,\widebar{w}_2^{\psi}\right)\dd r.
\end{multline}
Hence $\widebar{w}_2^{\psi}$ solves \eqref{claro} with $t$ replaced by $t_2$; to prove that it is in fact the unique solution,  we consider another random variable $\widetilde{w}\in\mathcal{H}$ satisfying \eqref{dai}. Then, by \eqref{VO_D}-\eqref{D_P_D},
\begin{equation}\label{uni1}
	1_{(0,t_1)}\widetilde{w}=1_{(0,t_1)}\left(\psi+\int_{s}^{t_1}D\widebar{B}\left(w^{s,\phi}_{t_1}\right)\left(r,1_{(0,t_1)}\widetilde{w}\right)\dd r
	\right).
\end{equation}
We observe that also $1_{(0,t_1)}\widebar{w}^{\psi}_1\in\mathcal{H}$ satisfies \eqref{uni1}. Therefore, using Bochner's theorem and Jensen's inequality, by \eqref{meglio} we can compute
\begin{align}\label{metterla:tardi}
	\notag\norm{1_{(0,t_1)}\left(\widebar{w}_1^{\psi}-\widetilde{w}\right)}_{\mathcal{H}}^2&\le 
	\mathbb{E}\left[\left(\int_{s}^{t_1}\norm{D\widebar{B}\left(w^{s,\phi}_{t_1}\right)\left(r, 1_{(0,t_1)}\left(\widebar{w}^{\psi}_1-\widetilde{w}\right)\right)}_2\dd r\right)^2\right]
	\\
	%=
	%\mathbb{E}\left[\Bigg(\int_{s}^{t_1}\left(\int_{0}^{T}\left|\widebar{B}\left(1_{(0,t_1)}w_1\right)-\widebar{B}\left(w_2\right)\right|^2\left(r,\xi\right)\dd\xi\right)^{\frac{1}{2}}\dd r\Bigg)^q\right]
	&\le 
	\Delta\mathbb{E}\left[\norm{D\widebar{B}\left(w^{s,\phi}_{t_1}\right)\left( 1_{(0,t_1)}\left(\widebar{w}^{\psi}_1-\widetilde{w}\right)\right)}^2_{2,\ts}\right]\le \Delta C_0^2\norm{1_{(0,t_1)}\left(\widebar{w}_1-\widetilde{w}\right)}_{\mathcal{H}}^2,
\end{align} 
which allow us to conclude, recalling that $\sqrt{\Delta}C_0<1$,
\[
1_{(0,t_1)}\widetilde{w}=1_{(0,t_1)}\widebar{w}^{\psi}_1,\quad \mathbb{P}-\text{a.s.}
\]
Going back to \eqref{dai}, by \eqref{claro} and the previous equality we have, $\mathbb{P}-$a.s.,
\begin{equation*}
	\widetilde{w}=\psi+\int_{s}^{t_1}D\widebar{B}\left(w^{s,\phi}_{t_1}\right)\left(r,	\widebar{w}_1^{\psi}\right)\dd r
+\int_{t_1}^{t_2}D\widebar{B}\left(w^{s,\phi}_{t_2}\right)\left(r,\widetilde{w}\right)\dd r
	=
	\widebar{w}_1^{\psi}	+\int_{t_1}^{t_2}D\widebar{B}\left(w^{s,\phi}_{t_2}\right)\left(r,\widetilde{w}\right)\dd r.
\end{equation*}
It follows that $\widetilde{w}$ satisfies \eqref{add}: by uniqueness, we obtain $\widetilde{w}=\widebar{w}^{\psi}_2$. Hence $\widebar{w}_2^{\psi}$ is the unique solution of \eqref{claro} in $\mathcal{H}$ with $t_2$ instead of $t$.

We define the operator $Dw^{s,\phi}_{t_2}\colon H\to \mathcal{H}$  by $Dw^{s,\phi}_{t_2}\psi=\widebar{w}_2^{\psi},\,\psi\in H,$ and claim that it is the Fréchet differential of $w^{s,\phi}_{t_2}$. To see this, note that the linearity of $Dw^{s,\phi}_{t_2}$ is a consequence of the well–posedness of \eqref{dai}. As for the continuity, it is ensured by the following computations, where we use \eqref{meglio}-\eqref{bound_derw_ste1}-\eqref{add}:
\begin{multline*}
	\norm{Dw_{t_2}^{s,\phi}\psi}_2\le 
	\norm{Dw_{t_1}^{s,\phi}\psi}_2+ \int_{t_1}^{t_2}\norm{D\widebar{B}\left(w^{s,\phi}_{t_2}\right)(r,Dw_{t_2}^{s,\phi}\psi)}_2 \dd r
	\\
	\le \left(1-C_0\sqrt{T/N}\right)^{-1}\norm{\psi}_2
	+
	\sqrt{\Delta}C_0\norm{Dw_{t_2}^{s,\phi}\psi}_2,\quad \mathbb{P}-\text{a.s.}, \,\psi \in H,
\end{multline*}
whence 
\begin{equation}\label{bound_ste2}
\norm{Dw_{t_2}^{s,\phi}\psi}_2\le \left(1-C_0\sqrt{T/N}\right)^{-2}\norm{\psi}_2,\quad \mathbb{P}-\text{a.s.}, \,\psi \in H.
\end{equation}
Moreover, by the cocycle property in \eqref{cocycle} and  reasoning as in \eqref{ma1}, by \eqref{eq_B1}-\eqref{add} we obtain, for some constant $c>0$,  
\begin{align}
	\notag&\norm{w^{s,\phi+h}_{t_2}-w^{s,\phi}_{t_2}-Dw^{s,\phi}_{t_2}h}_\mathcal{H}
	=
	\norm{w^{t_1,w^{s,\phi+h}_{t_1}}_{t_2}-w^{t_1,w^{s,\phi}_{t_1}}_{t_2}-Dw^{s,\phi}_{t_1}h-\int_{t_1}^{t_2}D\widebar{B}\left(w_{t_2}^{s,\phi}\right)\left(r,Dw_{t_2}^{s,\phi}h\right)\dd r}_\mathcal{H}
	\\&\qquad \notag
	\le 
	\norm{w^{s,\phi+h}_{t_1}-w^{s,\phi}_{t_1}-Dw^{s,\phi}_{t_1}h}_\mathcal{H}
	+
	\norm{\int_{t_1}^{t_2}\left(\widebar{B}\left(w^{s,\phi+h}_{t_2}\right)-\widebar{B}\left(w^{s,\phi}_{t_2}\right)-D\widebar{B}\left(w^{s,\phi}_{t_2}\right)Dw^{s,\phi}_{t_2}h\right)\left(r,	\cdot\right)\dd r}_\mathcal{H}
	\\&\label{ma4}\qquad \le  c\norm{h}_2^{1+\gamma}=
	\text{o}\left(\norm{h}_2\right),\quad h\in H,
\end{align}
where we also employ \eqref{ma2} in the last inequality.
This shows that $Dw^{s,\phi}_{t_2}$ is the Fréchet differential of $w^{s,\phi}_{t_2}$, as desired.

Repeating this argument $N-$times, we deduce that  the operator $Dw^{s,\phi}_{t}\colon H\to \mathcal{H}$ defined by $Dw^{s,\phi}_{t}\psi=\widebar{w}_N^{\psi}$,  where $\widebar{w}_N^{\psi}$ is the unique solution of \eqref{claro} in $\mathcal{H}$, for every $\psi\in H$, is the Fréchet differential of $w_t^{s,\phi}$.  In particular, the first bound in \eqref{alpha-h} is true, because (cfr. \eqref{bound_derw_ste1}-\eqref{bound_ste2})
\begin{equation}\label{bound_derw}
	\norm{Dw^{s,\phi}_t\psi}_{2}\le \left({1}-{C_0\sqrt{T/N}}\right)^{-{N}}\norm{\psi}_2\eqqcolon \widebar{C}\norm{\psi}_2,\quad \mathbb{P}-\text{a.s., }\phi, \psi\in {H}.
\end{equation}
As regards the second inequality in \eqref{alpha-h}, by \eqref{meglio}, \eqref{frec_w}  and \eqref{bound_derw} we have, for every $\phi,\,\psi,\,\eta\in {H},$ $\mathbb{P}-$a.s.,
\begin{align*}
	&\norm{Dw_{t_1}^{s,\phi}\eta-Dw_{t_1}^{s,\psi}\eta}_{2}
	= %\sup_{\norm{\eta}_{H}\le 1}
\norm{\int_{s}^{t_1}\left(D\widebar{B}\left(w^{s,\phi}_{t}\right)Dw_{t_1}^{s,\phi}\eta
		-D\widebar{B}\left(w^{s,\psi}_{t}\right)Dw_{t_1}^{s,\psi}\eta\right)\left(r,\cdot\right)\,\dd r}_2
	\\
	&\qquad \le \sqrt{\Delta}	\left(
	\norm{D\widebar{B}\left(w^{s,\phi}_{t}\right)\left(Dw_{t_1}^{s,\phi}\eta-Dw_{t_1}^{s,\psi}\eta\right)}_{2,\ts}
	+
\norm{\left(
	D\widebar{B}\left(w^{s,\phi}_{t}\right)
	-D\widebar{B}\left(w^{s,\psi}_{t}\right)\right)Dw_{t_1}^{s,\psi}\eta
	}_{2,\ts}\right)
	\\&\qquad \le
	C_0\sqrt{T/N}\left(\norm{Dw^{s,\phi}_{t_1}\eta-Dw^{s,\psi}_{t_1}\eta}_{2}+\widebar{C}\norm{w^{s,\phi}_{t}-w^{s,\psi}_{t}}_2^{\gamma}\norm{\eta}_2\right),
\end{align*}
where in the first equality we also use \eqref{D_P_D} and \eqref{wow} with $t$ instead of $t_2$. It follows that 
\[
\norm{Dw_{t_1}^{s,\phi}\eta-Dw_{t_1}^{s,\psi}\eta}_{2}
\le 
\left(1-C_0\sqrt{T/N}\right)^{-1}C_0\widebar{C}\sqrt{T/N}
\norm{w^{s,\phi}_{t}-w^{s,\psi}_{t}}_2^{\gamma}\norm{\eta}_2.
\]
By \eqref{add}, we sequentially iterate this computation to obtain the second bound in \eqref{alpha-h} with $$C_2=\max\{\widebar{C},NC_0\widebar{C}^2\sqrt{T/N}\}.$$

At this point, taking expectations and using Corollary \ref{lip} with $q=2\gamma$ (recall that $H\subset \mathcal{H}^q$), by Jensen's inequality we infer that, for some constant $C>0$,
\begin{multline*}
\norm{Dw_{t}^{s,\phi}-Dw_{t}^{s,\psi}}_{\mathcal{L}\left(H;\mathcal{H}\right)}=\sup_{\norm{\eta}_{2}\le 1}\mathbb{E}\left[\norm{Dw_{t}^{s,\phi}\eta-Dw_{t}^{s,\psi}\eta}_2^2\right]^\frac{1}{2}
\le 
C_2\mathbb{E}\left[\norm{w^{s,\phi}_{t}-w^{s,\psi}_{t}}_2^{2\gamma}\right]^{\frac{1}{2}}\\
\le C\norm{\phi-\psi}_2^\gamma,\quad \phi,\,\psi\in {H}.
\end{multline*}
This shows that $Dw^{s,\cdot}_t\in C^{\gamma}\left(H;\mathcal{L}(H;\mathcal{H})\right)$, completing the proof.
\end{proof}
\subsection{Second–order differentiability in the initial data}\label{second_order}
Recalling the normed space $\Lambda_2=\big(\Lambda,\norm{\cdot}_2\big)$, in the sequel we identify $\mathcal{L}(\Lambda_2;\mathcal{L}(\Lambda_2;H_{\ts}))$ with the space $\mathcal{L}(\Lambda_2,\Lambda_2;H_{\ts})$ of bilinear forms from $\Lambda_2 \times\Lambda_2$ to $H_{\ts}$ in the usual way.\\
For the purpose of investigating the second--order Fréchet differential in ${H}$ of $w^{s,\phi}_t$, we need to require another condition on $B$. 
\begin{ass}\label{B_3}
	The map $B\colon \Lambda\to H_{\ts}$ satisfies Assumption \ref{B_2}. Moreover, $B$ is twice  $\Lambda_2-$Fréchet differentiable, and there exists a constant $C_0=C_0\left(d,T\right)>0$ such that
	\begin{equation}\label{direct_2}
		\norm{D^2B\left(w_1\right)\left(w_2,w_3\right)}_{2,{\ts}}\le C_0\norm{w_2}_2\norm{w_3}_2,\quad w_1,w_2,w_3\in \Lambda, 
	\end{equation}	
	and 
	\begin{equation}\label{ho_alpha_2}
		\norm{D^2B\left(w_1\right)-D^2B\left(w_2\right)}_{\mathcal{L}\left(\Lambda_2,\Lambda_2;H_{\ts}\right)}\le C_0\norm{w_1-w_2}_2^\beta,\quad w_1,w_2\in \Lambda,\text{ for some }\beta\in\left(0,1\right].
	\end{equation}
\end{ass}
Once again,  we can assume that the constant $C_0$ in \eqref{direct_2}-\eqref{ho_alpha_2} is the same as the one in \eqref{bound_B} and  \eqref{direct}-\eqref{ho_alpha}.\\
By \eqref{direct_2}, we invoke the theorem of extension of uniformly continuous functions to extend, for every $w_1,w_2\in\Lambda$, the map $D^2B\left(w_1\right)(w_2,\cdot)\in\mathcal{L}\left(\Lambda_2; H_{\ts}\right)$ to an operator $\widebar {D^2B}\left(w_1\right)(w_2,\cdot)\in  \mathcal{L}\left(H; H_{\ts}\right)$ satisfying \eqref{direct_2} for all $w_3\in H$. It follows that, by linearity,
\begin{multline*}
	\norm{\widebar{D^2B}\left(w_1\right)\left(w_2\right)-\widebar{D^2B}\left(w_1\right)\left(w_3\right)}_{\mathcal{L}\left(H;H_{\ts}\right)}
	=
	\norm{\widebar{D^2B}\left(w_1\right)\left(w_2-w_3\right)}_{\mathcal{L}\left(H;H_{\ts}\right)}
	\\\le C_0\norm{w_2-w_3}_2,\quad w_1,w_2,w_3\in \Lambda,
\end{multline*}
hence we can extend (without changing notation) $\widebar{D^2B}(w_1)\in\mathcal{L}(H,H;H_{\ts})$, for all $w_1\in \Lambda$. At this point, by \eqref{ho_alpha_2} we infer that, for every $w_1,w_2\in \Lambda$,
\begin{equation}\label{ho_alpha22}
		\norm{\widebar{D^2B}\left(w_1\right)-\widebar{D^2B}\left(w_2\right)}_{\mathcal{L}\left(H,H;H_{\ts}\right)}
		=
		\norm{\widebar{D^2B}\left(w_1\right)-\widebar{D^2B}\left(w_2\right)}_{\mathcal{L}\left(\Lambda_2,\Lambda_2;H_{\ts}\right)}
		\le C_0\norm{w_1-w_2}_2^\beta,	
\end{equation}
whence, via another extension,  from now on we consider \begin{equation}\label{meglio1}
	\text{$\widebar{D^2B}\colon H\to \mathcal{L}(H,H;H_{\ts})$ satisfying (\ref{direct_2})-(\ref{ho_alpha22}) for every $w_i\in H,\,i=1,2,3.$}
\end{equation}
We want to show that $\widebar{B}$ is twice $H-$Fréchet differentiable, with $D^2\widebar{B}=\widebar{D^2B}.$ By Taylor's formula applied to $DB$, 
\begin{align}\label{save2}
	&\left(D\widebar{B}\left(w_2\right)-D\widebar{B}\left(w_1\right)-\widebar{D^2B}\left(w_1\right)\left(w_2-w_1\right)\right)w_3=r\left(w_1,w_2,w_3\right), \quad  w_1,w_2,w_3\in\Lambda,\text{ where}\\& 
\notag	r\left(x,y,z\right)= \left(\int_{0}^{1}\left(\widebar{D^2B}\left(x+h\left(y-x\right)\right)-\widebar{D^2B}\left(x\right)\right)\left(y-x\right)\dd h\right)z,\quad x,\,y,\,z\in H. 
\end{align}
We note that $r\colon H\times H\times H\to H_{\ts}$ is continuous. Indeed, for every $x,y,z \in H$ and every sequence $(\left(x_n,y_n,z_n\right))_n\subset H\times H\times H$ such that $(x_n,y_n,z_n)\to \left(x,y,z\right)$ as $n\to \infty$, with some algebraic computations we obtain, by \eqref{meglio1}, 
\begin{multline*}
	\norm{r\left(x_n,y_n,z_n\right)-r\left(x,y,z\right)}_{2,\ts}\le 2C_0\norm{y_n-x_n}_2\norm{z_n-z}_2\\+ C_0\norm{z}_2\left(2\norm{y_n-x_n+x-y}_2+\left(\frac{1}{\beta+1}\norm{y_n-y+x-x_n}^\beta_2+2\norm{x_n-x}_2^\beta\right)\norm{y-x}_2\right)\underset{n\to\infty}{\longrightarrow}0.
\end{multline*}
It then follows  from the continuity of $D\widebar{B}$ in $H$ and \eqref{meglio1} that  \eqref{save2} holds for every $w_1,w_2,w_3\in H$. Moreover, observing that, by \eqref{meglio1},  $\norm{r\left(x,y,\cdot\right)}_{\mathcal{L}(H;H_{\ts})}\le C_0(\beta+1)^{-1}\norm{y-x}^{1+\beta}_2,\,x,y\in H$, we conclude that 
\[
D\widebar{B}\left(w_2\right)-D\widebar{B}\left(w_1\right)-\widebar{D^2B}\left(w_1\right)\left(w_2-w_1\right)=\text{o}\left(\norm{w_2-w_1}_2\right),\quad w_1,w_2\in H.
\]
Therefore $\widebar{B}$ is twice $H-$Fréchet differentiable, with $D^2\widebar{B}=\widebar{D^2B}.$

We also note that, for every $w_1,\,w_2,\,w_3\in H$ and $0< t\le T$,  
		\begin{equation}\label{V0D1}
			\text{$D^2\widebar{B}(w_1)\left(w_2,w_3\right)\left(r,\cdot\right)\in H$ is of Volterra--type, for a.e. $r\in\left(0,t\right)$,}
		\end{equation} 
	and that 
	\begin{equation}\label{DPD1}
		\text{$D^2\widebar{B}(w_1)\left(w_2,w_3\right)(r,\cdot)$ depends on $w_i$ only via $\restr{w_i}{\left(0,t\right)},\,i=1,2,3$, for a.e. $r\in (0,t)$:}
	\end{equation}
 these properties are inherited from $D\widebar{B}$ (cfr.  \eqref{VO_D}-\eqref{D_P_D} in the discussion following Assumption \ref{B_2}).

In conclusion, we notice that, by \eqref{meglio1} (see also \eqref{direct_2}), 
\begin{equation}\label{massu}
	\norm{D^2\widebar{B}}_\infty=\sup_{w\in H}\norm{D^2\widebar{B}(w)}_{\mathcal{L}(H,H;H_{\ts})}\le C_0.
\end{equation}
 As a consequence, by the mean value theorem we deduce that \eqref{ho_alpha1} (see also \eqref{meglio}) holds with $\gamma=1$, i.e., under Assumption \ref{B_3} the map $D\widebar{B}\colon H\to \mathcal{L}(H;H_{\ts})$ is globally Lipschitz--continuous. Since $D\widebar B$ is also bounded (see \eqref{direct}-\eqref{meglio}), in what follows we suppose, without loss of generality, that 
 \begin{equation}\label{unavolta}
 	\text{under Assumption \ref{B_3}, $D\widebar B\colon H\to \mathcal{L}(H;H_{\ts})$ satisfies (\ref{meglio}) with $\gamma=\beta$. }
 \end{equation}
\begin{comment}

By analogy with Lemma \ref{diff_1} we have the following, stronger result, whose proof is foregone.
\begin{lemma}\label{diff_2}
	Under Assumption \ref{B_3}, for every $0\le s \le t \le T$, the mapping $\Gamma_s^t \in C^2\left(\mathcal{H}\times \mathcal{H};\mathcal{H}\right)$, with
	\begin{equation*}
		D^2_w\Gamma_s^t\left(\phi,w\right)\left(w_0,w_1\right)=\int_{s}^{t}D^2\widebar{B}\left(w\right)\left(w_0,w_1\right)\left(r,\cdot\right)\,\dd r,
	\end{equation*}
	for all $\phi,\,w,\,w_0,\,w_1\in \mathcal{H}$.
\end{lemma}
\end{comment}

The next result shows that, in the framework of this subsection, the solution $w_t^{s,\phi}$  of \eqref{eq_B1}, considered as a map from $H$ to $\mathcal{H}$, is twice $H-$Fréchet differentiable.
\begin{theorem}\label{diff_w2_thm}
	Under Assumption \ref{B_3}, for every $0\le s\le t \le T$, the mapping $w^{s,\cdot}_t\in C^{2+\beta}\left({H};\mathcal{H}\right)$. In particular, for every $\phi,\,\psi,\,\eta\in {H}$, 	$D^2w^{s,\phi}_{t} \left(\psi,\eta\right)$ is the unique solution in $\mathcal{H}$ of the following equation:  
	\begin{equation}\label{frec_w2}
		D^2w^{s,\phi}_{t} \left(\psi,\eta\right)= \int_{s}^{t}\Big(D^2\widebar{B}\left(w^{s,\phi}_t\right)\left(Dw^{s,\phi}_t\psi,Dw_t^{s,\phi}\eta\right)
		+D\widebar{B}\left(w^{s,\phi}_t\right)D^2w^{s,\phi}_t\left(\psi,\eta\right)\Big)\left(r,\cdot\right)\dd r.
	\end{equation} 
	Furthermore, there exists a constant $C_3=C_3(d,T)>0$ such that, for every $\phi,\psi,\eta,\theta\in {H}$, $\mathbb{P}-$a.s.,
	\begin{equation}\label{alpha-h2}
		\norm{D^2w^{s,\phi}_t\left(\eta,\theta\right)}_2\le C_3\norm{\eta}_2\norm{\theta}_2,\qquad 
		\norm{\left(D^2w^{s,\phi}_t-D^2w^{s,\psi}_t\right)\left(\eta,\theta\right)}_2\le C_3\norm{w^{s,\phi}_t-w^{s,\psi}_t}_{2}^\beta\norm{\eta}_2\norm{\theta}_2.
	\end{equation}
\end{theorem}

\begin{proof}
	Fix $0\le s \le t\le T$ and $\phi\in H$.  We first want to prove the well–posedness in $\mathcal{H}$ of the equation
	\begin{equation}\label{claro2}
		w=\int_{s}^{t}\Big(D^2\widebar{B}\left(w^{s,\phi}_t\right)\left(Dw^{s,\phi}_t\psi,Dw_t^{s,\phi}\eta\right)
		+
		D\widebar{B}\left(w^{s,\phi}_t\right)w\Big)\left(r,\cdot\right)\dd r,\quad \psi,\,\eta\in H.
	\end{equation}
	Consider $N=N\left(d,T\right)\in \mathbb{N}$ so big that $C_0\sqrt{T/N}<1,$ where $C_0=C_0\left(d,T\right)$ is the constant in Assumptions \ref{B_1}-\ref{B_2}-\ref{B_3}. In addition, take an equispaced partition $\{t_k\}_{k=0}^N$ of $\left[s,t\right]$ where $t_0=s$ and $t_N=t$: its mesh $\Delta\le T/N$. 
	Under Assumption \ref{B_3}, the bound in \eqref{claro1} holds and allows to employ a fixed point argument as in the proof of Theorem \ref{well_pos} (see also Theorem \ref{diff_w_thm})
	to deduce the existence of a unique solution $\widebar{w}^{\psi,\eta}_1\in\mathcal{H}$ of \eqref{claro2} with $t_1$ instead of $t$, for every $\psi,\eta\in H$. 
	
We claim that the operator $D^2w^{s,\phi}_{t_1}\colon H\times H\to \mathcal{H}$ defined by $D^2w^{s,\phi}_{t_1}(\psi,\eta)=\widebar{w}_1^{\psi,\eta},\,\psi,\eta\in H,$ is the second–order Fréchet differential of $w^{s,\phi}_{t_1}$. Indeed, considering that $Dw^{s,\phi}_{t_1}\in\mathcal{L}(H;\mathcal{H})$,  $D\widebar{B}(w^{s,\phi}_{t_1})\in \mathcal{L}(H;{H}_{\ts})$ and $D^2\widebar{B}(w_{t_1}^{s,\phi})\in\mathcal{L}(H,H;{H}_{\ts})$, the fact that $D^2w^{s,\phi}_{t_1}$ is bilinear directly follows from \eqref{claro2}. As for the boundedness, by \eqref{meglio}-\eqref{meglio1} (see also \eqref{direct}-\eqref{direct_2}) and  \eqref{alpha-h}   we can compute, applying Bochner's theorem to \eqref{frec_w2}, for some constant $C_2=C_2(d,T)>0$, 
\begin{multline}\label{ste3}
	\norm{D^2w^{s,\phi}_{t_1}\left(\psi,\eta\right)}_{2}
	\le C_0\sqrt{\Delta}\left(\norm{Dw^{s,\phi}_{t_1}\psi}_2\norm{Dw^{s,\phi}_{t_1}\eta}_2
	+\norm{D^2w_{t_1}^{s,\phi}(\psi,\eta)}_2
	\right)
	\\
	\le C_0\sqrt{{T/N}} \left(C_2\norm{\psi}_2\norm{\eta}_2+\norm{D^2w_{t_1}^{s,\phi}(\psi,\eta)}_2\right), \quad \mathbb{P}-\text{a.s., } \psi,\eta\in {H}.
\end{multline}
Hence 
\begin{equation}\label{bound_der3}
\norm{D^2w^{s,\phi}_{t_1}\left(\psi,\eta\right)}_{2}
\le \left(1-C_0\sqrt{T/N}\right)^{-1}C_0C_2\sqrt{T/N}\norm{\psi}_2\norm{\eta}_2,	\quad \mathbb{P}-\text{a.s., }\psi,\eta\in {H}.
\end{equation}
We now observe that, by Taylor's formula applied to $D\widebar{B}$ (cfr. \eqref{save2}), from \eqref{frec_w}-\eqref{frec_w2} we have, for every $h\in H$, 
	\begin{equation}\label{male}
		\norm{Dw^{s,\phi+h}_{t_1}-Dw^{s,\phi}_{t_1}-D^2w^{s,\phi}_{t_1}h}_{\mathcal{L}\left(H;\mathcal{H}\right)}\le  \mathbf{\upperRomannumeral{1}}_1+\mathbf{\upperRomannumeral{2}}_1+\mathbf{\upperRomannumeral{3}}_1+\mathbf{\upperRomannumeral{4}}_1,
	\end{equation}
where we set 
\begin{align*}
	\mathbf{\upperRomannumeral{1}}_1&=\sup_{\norm{\eta}_{2}\le 1}
\mathbb{E}\left[\norm{\int_{s}^{t_1}D\widebar{B}\left(w^{s,\phi}_{t_1}\right)\left(Dw_{t_1}^{s,\phi+h}\eta-Dw_{t_1}^{s,\phi}\eta-D^2w^{s,\phi}_{t_1}\left(h,\eta\right)\right)\left(r,\cdot\right)\,\dd r}^2_2\right]^{\frac{1}{2}},
	\\
		\mathbf{\upperRomannumeral{2}}_1&=\sup_{\norm{\eta}_{2}\le 1}\mathbb{E}\left[\norm{\int_{s}^{t_1}\left(D^2\widebar{B}\left(w^{s,\phi}_{t_1}\right)\left(w_{t_1}^{s,\phi+h}-w^{s,\phi}_{t_1}-Dw^{s,\phi}_{t_1}h,Dw^{s,\phi}_{t_1}\eta\right)
		\right)\left(r,\cdot\right)\,\dd r}_2^2\right]^\frac{1}{2},
		\\
		\mathbf{\upperRomannumeral{3}}_1&=\sup_{\norm{\eta}_{2}\le 1}
	\mathbb{E}\left[\norm{\int_{s}^{t_1}\left(D\widebar{B}\left(w^{s,\phi+h}_{t_1}\right)-D\widebar{B}\left(w^{s,\phi}_{t_1}\right)\right)\left(Dw_{t_1}^{s,\phi+h}\eta-Dw_{t_1}^{s,\phi}\eta\right)\left(r,\cdot\right)\,\dd r}^2_2\right]^{\frac{1}{2}},
	\\
	\mathbf{\upperRomannumeral{4}}_1&=\sup_{\norm{\eta}_{2}\le 1}
		\mathbb{E}\left[\bigg|\hspace{-.09em}\bigg|\int_{s}^{t_1}\left(\int_{0}^{1}\left(D^2\widebar{B}\left(w^{s,\phi}_{t_1}+v\left(w^{s,\phi+h}_{t_1}-w^{s,\phi}_{t_1}\right)\right)
		\right.\right.\right.\\&\hspace{16em}	\left.\left.\left.
		-D^2\widebar{B}\left(w^{s,\phi}_{t_1}\right)\right)\left(w^{s,\phi+h}_{t_1}-w^{s,\phi}_{t_1}\right)\dd v\right)
			Dw_{t_1}^{s,\phi}\eta\,\left(r,\cdot\right)\dd r\bigg|\hspace{-.09em}\bigg|^2_2\right]^{\frac{1}{2}}.
\end{align*}
By \eqref{meglio} (see, in particular, \eqref{direct})
\begin{multline*}
	\left|\mathbf{\upperRomannumeral{1}}_1\right|\le C_0\sqrt{T/N}\sup_{\norm{\eta}_{2}\le 1}\mathbb{E}\left[	\norm{Dw^{s,\phi+h}_{t_1}\eta-Dw^{s,\phi}_{t_1}\eta-D^2w^{s,\phi}_{t_1}\left(h,\eta\right)}_2^2\right]^{\frac{1}{2}}
	\\= C_0\sqrt{T/N}	\norm{Dw^{s,\phi+h}_{t_1}-Dw^{s,\phi}_{t_1}-D^2w^{s,\phi}_{t_1}h}_{\mathcal{L}\left(H;\mathcal{H}\right)}.
\end{multline*}
Moreover, considering   \eqref{ho_alpha1}-\eqref{alpha-h} (see also \eqref{unavolta}) and Corollary \ref{lip}, which we can apply with $q=2(1+\beta)$ because $\phi,h\in H\subset \mathcal{H}^q$ (see also ), for some $C_1=C_1(\beta,d,T)>0$ we can write
\begin{multline*}
	\left|\mathbf{\upperRomannumeral{3}}_1\right|
	\le \sqrt{\Delta}
	\sup_{\norm{\eta}_{2}\le 1}\mathbb{E}\left[\norm{
\left(D\widebar{B}\left(w^{s,\phi+h}_{t_1}\right)-D\widebar{B}\left(w^{s,\phi}_{t_1}\right)\right)\left(Dw_{t_1}^{s,\phi+h}\eta-Dw_{t_1}^{s,\phi}\eta\right)	
}_{2,\ts}^2\right]^{\frac{1}{2}}
	\\
	\le \norm{D^2\widebar B}_\infty C_2\sqrt{T/N} \sup_{\norm{\eta}_{2}\le 1}\mathbb{E}\left[\norm{w^{s,\phi+h}_{t_1}-w^{s,\phi}_{t_1}}^{2\left(1+\beta\right)}_2\norm{\eta}_2^2\right]^\frac{1}{2}
	\le C_0C_{1}^{1+\beta}C_2\sqrt{T/N}\norm{h}_2^{1+\beta},
\end{multline*}
where we also use the mean value theorem on $D\widebar{B}$ and \eqref{massu}.
As for $\mathbf{\upperRomannumeral{2}}_1$, by \eqref{alpha-h}-\eqref{meglio1} we compute
\begin{align*}
	\left|\mathbf{\upperRomannumeral{2}}_1\right|
	&\le 
	\sqrt{\Delta}\sup_{\norm{\eta}_{2}\le 1}\mathbb{E}\left[
	\norm{
D^2\widebar{B}\left(w^{s,\phi}_{t_1}\right)\left(w_{t_1}^{s,\phi+h}-w^{s,\phi}_{t_1}-Dw^{s,\phi}_{t_1}h,Dw^{s,\phi}_{t_1}\eta\right)	
}^2_{2,\ts}
	\right]^{\frac{1}{2}}
	\\&
	\le C_0C_2\sqrt{\Delta}\sup_{\norm{\eta}_{2}\le 1}\mathbb{E}\left[\norm{w_{t_1}^{s,\phi+h}-w^{s,\phi}_{t_1}-Dw^{s,\phi}_{t_1}h}_2^2\norm{\eta}_2^2\right]^{\frac{1}{2}}
	\\&\le 
	C_0C_2\sqrt{T/N}\norm{w_{t_1}^{s,\phi+h}-w^{s,\phi}_{t_1}-Dw^{s,\phi}_{t_1}h}_{\mathcal{H}}
	=
	\text{o}\left(\norm{h}_2\right).
\end{align*}
Finally, again by \eqref{alpha-h}-\eqref{meglio1} (see also \eqref{ho_alpha22}) and Corollary \ref{lip}, employed with $q=2(1+\beta)$, we have 
\begin{align*}
	\left|\mathbf{\upperRomannumeral{4}}_1\right|
	&\le 
	\sqrt{\Delta}\mathbb{E}\left[\left(\int_{0}^{1}
		\norm{
		D^2\widebar{B}\left(w^{s,\phi}_{t_1}+v\left(w^{s,\phi+h}_{t_1}-w^{s,\phi}_{t_1}\right)\right)
		-D^2\widebar{B}\left(w^{s,\phi}_{t_1}\right)}_{\mathcal{L}(H,H;H_{\ts})}\dd v\right)^2
\right.	\\&\hspace{28em}\left.\times 
	\norm{w_{t_1}^{s,\phi+h}-w_{t_1}^{s,\phi}}_{2}^2
	\norm{Dw_{t_1}^{s,\phi}\eta}_{2}^2\right]^{\frac{1}{2}}
\\&	\le C_0C_2\sqrt{T/N}\sup_{\norm{\eta}_{2}\le 1}\mathbb{E}\left[\norm{w_{t_1}^{s,\phi+h}-w_{t_1}^{s,\phi}}_2^{2\left(1+\beta\right)}\norm{\eta}_2^2\right]^{\frac{1}{2}}\le C_0C_{1}^{1+\beta}C_{2}\sqrt{T/N}\norm{h}_2^{1+\beta}.
\end{align*}
Going back to \eqref{male}, we conclude that 
\begin{equation}\label{ma3}
	\norm{Dw^{s,\phi+h}_{t_1}-Dw^{s,\phi}_{t_1}-D^2w^{s,\phi}_{t_1}h}_{\mathcal{L}\left(H;\mathcal{H}\right)}
	\le\left(1-C_0\sqrt{T/N}\right)^{-1} \left(\mathbf{\upperRomannumeral{2}}_1+\mathbf{\upperRomannumeral{3}}_1+\mathbf{\upperRomannumeral{4}}_1\right)
	=\text{o}\left(\norm{h}_2\right),\quad h\in H.
\end{equation}
This shows that $D^2w^{s,\phi}_{t_1}$ is the second–order Fréchet differential of $w^{s,\phi}_{t_1}$, as desired.

Next, consider 
\begin{equation}\label{add2}
w
	=D^2w_{t_1}^{s,\phi}\left(\psi,\eta\right)\\+
	\int_{t_1}^{t_2}\Big(D^2\widebar{B}\left(w^{s,\phi}_{t_2}\right)\left(Dw^{s,\phi}_{t_2}\psi,Dw_{t_2}^{s,\phi}\eta\right)
	+D\widebar{B}\left(w^{s,\phi}_{t_2}\right)w\Big)\left(r,\cdot\right)\dd r,\quad \psi,\eta \in H.
\end{equation}
Arguing as in the previous step, we infer the well–posedness of this equation in $\mathcal{H}$: we denote by $\widebar{w}_2^{\psi,\eta}\in\mathcal H$ its unique solution, for every $\psi,\,\eta\in H.$

Given $\psi,\,\eta\in H$, we now show that $\widebar{w}_2^{\psi,\eta}$ is the unique solution of \eqref{claro2} with $t_2$ instead of $t$. 
By the Volterra–type property of $D^2\widebar{B}$ [resp., $D\widebar{B}$] in \eqref{V0D1} [resp., \eqref{VO_D}] and \eqref{add2} we have,  $\mathbb{P}-$a.s.,
\[
\restr{\widebar{w}_2^{\psi,\eta}}{\left(0,t_1\right)}=\restr{D^2w^{s,\phi}_{t_1}(\psi,\eta)}{\left(0,t_1\right)}.
\]
Moreover, since  $Dw^{s,\phi}_{t_2}\psi$ satisfies \eqref{add}, we infer that, $\mathbb{P}-$a.s.,
\[
	\restr{D{w}_{t_2}^{s,\phi}\psi}{\left(0,t_1\right)}=\restr{D{w}_{t_1}^{s,\phi}\psi}{\left(0,t_1\right)},
\]
with an analogous result holding for $\eta$. 
Consequently, recalling also \eqref{wow} and Remark \ref{classical}, by the property of $D^2\widebar{B}$ [resp., $D\widebar{B}$] in \eqref{DPD1} [resp., \eqref{D_P_D}], from \eqref{add2} we obtain, $\mathbb{P}-$a.s., 
\begin{align}\label{dai2}
	\notag\widebar{w}_2^{\psi,\eta}
	&=
	\int_{s}^{t_1}\Big(D^2\widebar{B}\left(w^{s,\phi}_{t_1}\right)\left(Dw^{s,\phi}_{t_1}\psi,Dw_{t_1}^{s,\phi}\eta\right)
	+
	D\widebar{B}\left(w^{s,\phi}_{t_1}\right)D^2w^{s,\phi}_{t_1}(\psi,\eta)\Big)\left(r,\cdot\right)\dd r
	\\	&\qquad \notag\qquad \qquad \qquad 	+\int_{t_1}^{t_2}\Big(D^2\widebar{B}\left(w^{s,\phi}_{t_2}\right)\left(Dw^{s,\phi}_{t_2}\psi,Dw_{t_2}^{s,\phi}\eta\right)
	+D\widebar{B}\left(w^{s,\phi}_{t_2}\right)\widebar{w}_2^{\psi,\eta}\Big)\left(r,\cdot\right)\dd r
\\	&=
	\int_{s}^{t_2}\Big(D^2\widebar{B}\left(w^{s,\phi}_{t_2}\right)\left(Dw^{s,\phi}_{t_2}\psi,Dw_{t_2}^{s,\phi}\eta\right)
	+
	D\widebar{B}\left(w^{s,\phi}_{t_2}\right)\widebar{w}_2^{\psi,\eta}\Big)\left(r,\cdot\right)\dd r,
\end{align}
where we also use the fact that  $D^2w_{t_1}^{s,\phi}(\psi,\eta)$ solves \eqref{claro2} with $t_1$ instead of $t$.
Hence $\widebar{w}_2^{\psi,\eta}$ solves \eqref{claro2} with $t$ replaced by $t_2$. In order to prove that it is in fact the unique solution of this equation,  we consider another random variable $\widetilde{w}\in\mathcal{H}$ satisfying \eqref{dai2}. Then, by \eqref{V0D1}-\eqref{DPD1},
\begin{equation}\label{uni2}
	1_{(0,t_1)}\widetilde{w}=1_{(0,t_1)}\left(
		\int_{s}^{t_1}\Big(D^2\widebar{B}\left(w^{s,\phi}_{t_1}\right)\left(Dw^{s,\phi}_{t_1}\psi,Dw_{t_1}^{s,\phi}\eta\right)
	+
	D\widebar{B}\left(w^{s,\phi}_{t_1}\right)1_{(0,t_1)}\widetilde{w}\Big)\left(r,\cdot\right)\dd r
	\right).
\end{equation}
We observe that also $1_{(0,t_1)}\widebar{w}^{\psi,\eta}_1\in\mathcal{H}$ satisfies \eqref{uni2}. Therefore we can perform the same computations as in \eqref{metterla:tardi} to deduce that 
\[
1_{(0,t_1)}\widetilde{w}=1_{(0,t_1)}\widebar{w}^{\psi,\eta}_1,\quad \mathbb{P}-\text{a.s.}
\]
Going back to \eqref{dai2}, by the previous equality we have, $\mathbb{P}-$a.s.,
\begin{align*}
	\widetilde{w}&=	\int_{s}^{t_1}\Big(D^2\widebar{B}\left(w^{s,\phi}_{t_1}\right)\left(Dw^{s,\phi}_{t_1}\psi,Dw_{t_1}^{s,\phi}\eta\right)
	+
	D\widebar{B}\left(w^{s,\phi}_{t_1}\right)\widetilde{w}\Big)\left(r,\cdot\right)\dd r
	\\&
	\hspace{12em}
	+\int_{t_1}^{t_2}\Big(D^2\widebar{B}\left(w^{s,\phi}_{t_2}\right)\left(Dw^{s,\phi}_{t_2}\psi,Dw_{t_2}^{s,\phi}\eta\right)
	+D\widebar{B}\left(w^{s,\phi}_{t_2}\right)\widetilde{w}\Big)\left(r,\cdot\right)\dd r
\\&	=
	\widebar{w}_1^{\psi,\eta}	+\int_{t_1}^{t_2}\Big(D^2\widebar{B}\left(w^{s,\phi}_{t_2}\right)\left(Dw^{s,\phi}_{t_2}\psi,Dw_{t_2}^{s,\phi}\eta\right)
	+D\widebar{B}\left(w^{s,\phi}_{t_2}\right)\widetilde{w}\Big)\left(r,\cdot\right)\dd r.
\end{align*}
It follows that $\widetilde{w}$ satisfies \eqref{add2}: by uniqueness, we obtain $\widetilde{w}=\widebar{w}^{\psi,\eta}_2$. Hence $\widebar{w}_2^{\psi,\eta}$ is the unique solution of \eqref{claro2} in $\mathcal{H}$ with $t_2$ instead of $t$.

We define the operator $D^2w^{s,\phi}_{t_2}\colon H\times H\to \mathcal{H}$  by $D^2w^{s,\phi}_{t_2}(\psi,\eta)=\widebar{w}_2^{\psi,\eta},\,\psi,\eta\in H,$ and claim that it is the second–order Fréchet differential of $w^{s,\phi}_{t_2}$. Indeed, as we have argued for $D^2w^{s,\phi}_{t_1}$, the map $D^2w^{s,\phi}_{t_2}$ is bilinear thanks to the the well–posedness of \eqref{dai2}. As for the boundedness, arguing as in \eqref{ste3}, by \eqref{bound_der3}-\eqref{add2} we can write, for every  $\psi,\,\eta \in H$, $ \mathbb{P}-\text{a.s.}$,
\begin{align*}
	\norm{Dw_{t_2}^{s,\phi}(\psi,\eta)}_2&\le 
	\norm{D^2w_{t_1}^{s,\phi}(\psi,\eta)}_2
	\\
	&\qquad \qquad 
	+ \int_{t_1}^{t_2}\norm{\left(
	D^2\widebar{B}\left(w^{s,\phi}_{t_2}\right)\left(Dw^{s,\phi}_{t_2}\psi,Dw_{t_2}^{s,\phi}\eta\right)
	+D\widebar{B}\left(w^{s,\phi}_{t_2}\right)Dw_{t_2}^{s,\phi}(\psi,\eta)\right)(r,\cdot)
}_2 \dd r
	\\&
	\le C_0C_2\left(\left(1-C_0\sqrt{T/N}\right)^{-1}\sqrt{T/N}+\sqrt{\Delta}\right)\norm{\psi}_2\norm{\eta}_2
	+
	\sqrt{\Delta}C_0\norm{D^2w_{t_2}^{s,\phi}(\psi,\eta)}_2,\quad
\end{align*}
whence 
\begin{equation}\label{bound_ste3}
	\norm{Dw_{t_2}^{s,\phi}(\psi,\eta)}_2\le 2C_0C_2\left(1-C_0\sqrt{T/N}\right)^{-2}\sqrt{T/N}\norm{\psi}_2\norm{\eta}_2,\quad \mathbb{P}-\text{a.s.}, \,\psi,\,\eta \in H.
\end{equation}
Moreover, combining  \eqref{add} with \eqref{add2}, we can argue as in \eqref{male} to infer that 
\[
\norm{Dw^{s,\phi+h}_{t_2}-Dw^{s,\phi}_{t_2}-D^2w^{s,\phi}_{t_2}h}_{\mathcal{L}\left(H;\mathcal{H}\right)}=
		\text{o}\left(\norm{h}_2\right),\quad h\in H,
\]
which shows that $D^2w^{s,\phi}_{t_2}$ is the second–order Fréchet differential of $w^{s,\phi}_{t_2}$, as desired.

	This reasoning can be repeated $N-$times to deduce that the operator 
	 $D^2w^{s,\phi}_{t}\colon H\times H\to \mathcal{H}$  defined by $D^2w^{s,\phi}_{t}(\psi,\eta)=\widebar{w}_N^{\psi,\eta},$ where $\widebar{w}_N^{\psi,\eta}$ is the unique solution of \eqref{claro2} in $\mathcal{H}$, for every $\psi,\,\eta\in H$, is the second–order Fréchet differential of $w^{s,\phi}_t$. In particular, the first bound in \eqref{alpha-h2} is true, because (cfr. \eqref{bound_der3}-\eqref{bound_ste3})
	\begin{align}\label{bound_derw2}
		&\notag\norm{D^2w^{s,\phi}_t\left(\psi,\eta\right)}_{2}\le NC_0C_2\left(1-C_0\sqrt{T/N}\right)^{-N}\sqrt{T/N}\norm{\psi}_2\norm{\eta}_2\eqqcolon\widetilde{C}\norm{\psi}_2\norm{\eta}_2,\\
		&\qquad  \mathbb{P}-\text{a.s., }\phi,\psi,\eta\in {H}.
	\end{align}
	
	As for the second inequality in \eqref{alpha-h2},  by  \eqref{meglio}, \eqref{alpha-h}, \eqref{meglio1}, \eqref{unavolta}, \eqref{frec_w2} and \eqref{bound_derw2} we compute, for every $\phi,\psi,\eta,\theta\in {H}$, $\mathbb{P}-$a.s.,
	\begin{align*}
	&\norm{D^2w_{t_1}^{s,\phi}\left(\eta,\theta\right)-D^2w_{t_1}^{s,\psi}\left(\eta,\theta\right)}_2
	\\& \,\,=
\bigg|\hspace{-.09em}\bigg|\int_{s}^{t_1}\Big(
	D^2\widebar{B}\left(w^{s,\phi}_t\right)\left(Dw^{s,\phi}_{t}\eta,Dw_{t}^{s,\phi}\theta\right)
	-D^2\widebar{B}\left(w^{s,\psi}_t\right)\left(Dw^{s,\psi}_{t}\eta,Dw_{t}^{s,\psi}\theta\right)
	\\&\hspace{13em}+D\widebar{B}\left(w^{s,\phi}_t\right)D^2w^{s,\phi}_{t_1}\left(\eta,\theta\right)
	-D\widebar{B}\left(w^{s,\psi}_t\right)D^2w^{s,\psi}_{t_1}\left(\eta,\theta\right)\Big)\left(r,\cdot\right)\dd r\bigg|\hspace{-.09em}\bigg|_2
	\\&\,\, \le \sqrt{\Delta}\Big(
	\norm{\left(D^2\widebar{B}\left(w^{s,\phi}_t\right)-D^2\widebar{B}\left(w^{s,\psi}_t\right)\right)\left(Dw^{s,\phi}_{t}\eta,Dw_{t}^{s,\phi}\theta\right)}_{2,\ts }
\\&\,\,\quad +\!
	\norm{D^2\widebar{B}\left(w^{s,\psi}_t\right)\!\left(\left(Dw^{s,\phi}_{t}-Dw^{s,\psi}_t\right)\eta,Dw_{t}^{s,\phi}\theta\right)}_{2,\ts }\!
	+\!	\norm{D^2\widebar{B}\left(w^{s,\psi}_t\right)\!\left(Dw^{s,\psi}_t\eta,\left(Dw_{t}^{s,\phi}-Dw_{t}^{s,\psi}\right)\theta\right)}_{2,\ts }
	\\&\,\, \quad+
		\norm{\left(D\widebar{B}\left(w^{s,\phi}_t\right)-D\widebar{B}\left(w^{s,\psi}_t\right)\right)D^2w^{s,\phi}_{t_1}\left(\eta,\theta\right)}_{2,\ts}
	 +\norm{D\widebar{B}\left(w^{s,\psi}_t\right)\left(D^2w^{s,\phi}_{t_1}\left(\eta,\theta\right)-D^2w^{s,\psi}_{t_1}\left(\eta,\theta\right)\right)}_{2,\ts}\Big)
	\\&\,\,\le 
	C_0\sqrt{T/N}\Big(\left(\widetilde{C}+3C_2\right)\norm{w^{s,\phi}_t-w^{s,\psi}_t}^\beta_2\norm{\eta}_2\norm{\theta}_2
	+\norm{D^2w_{t_1}^{s,\phi}\left(\eta,\theta\right)-D^2w_{t_1}^{s,\psi}\left(\eta,\theta\right)}_2\Big)
	,
\end{align*}	
whence 
\begin{equation*}
\norm{D^2w_{t_1}^{s,\phi}\left(\eta,\theta\right)-D^2w_{t_1}^{s,\psi}\left(\eta,\theta\right)}_2
\le 
\left(1-C_0\sqrt{T/N}\right)^{-1}
 C_0\left(\widetilde{C}+3C_2\right)\sqrt{T/N}
\norm{w^{s,\phi}_{t}-w^{s,\psi}_{t}}_2^{\beta}\norm{\eta}_2\norm{\theta}_2.
\end{equation*}
By \eqref{add2}, we sequentially iterate this computation to obtain the second inequality in \eqref{alpha-h2} with $$C_3=\max\{\widetilde{C},N\left(1-C_0\sqrt{T/N}\right)^{-N}C_0(\widetilde{C}+3C_2)\sqrt{T/N}\}.$$
Thus, taking expectations and using   Corollary \ref{lip} with $q=2$, by Jensen's inequality we deduce  that, for some constant $c>0$, 
\begin{multline*}
	\norm{D^2w_{t}^{s,\phi}-D^2w_{t}^{s,\psi}}_{\mathcal{L}\left(H,H;\mathcal{H}\right)}=\sup_{\norm{\eta}_{2},\norm{\theta}_2\le 1}\mathbb{E}\left[\norm{D^2w_{t}^{s,\phi}\left(\eta,\theta\right)-D^2w_{t}^{s,\psi}\left(\eta,\theta\right)}_2^2\right]^\frac{1}{2}
	\\\le 
	C_3\mathbb{E}\left[\norm{w^{s,\phi}_{t}-w^{s,\psi}_{t}}_2^{2\beta}\right]^{\frac{1}{2}}
	\le c\norm{\phi-\psi}_2^\beta,\quad \phi,\,\psi\in {H}.
\end{multline*}
This shows that $D^2w^{s,\cdot}_t\in C^{\beta}\left(H;\mathcal{L}(H,H;\mathcal{H})\right)$, completing the proof.
\end{proof}

\section{The Kolmogorov equation}\label{sec_KO}
Recall the definition of the map $\sigma\colon[0,T]\to\mathcal{L}(\mathbb{R}^d;H)  $ in \eqref{kernel}. Given $u\colon\left[0,T\right]\times H\to \mathbb{R}$ and a terminal condition $\Phi\colon H\to \mathbb{R}$, in this section we investigate the following \emph{Kolmogorov backward equation} in integral form:
\begin{align}\label{kolm}
	\notag&u\left(t,\phi\right)=\Phi\left(\phi\right)+\int_{t}^{T}\left\langle\nabla u\left(r,\phi\right),B\left(r,\phi\right)\right\rangle_H \dd r+\frac{1}{2}\int_{t}^{T}\text{Tr}\left(D^2u\left(r,\phi\right)\sigma\left(r\right)\sigma\left(r\right)^\ast\right)\dd r,\\&\qquad t\in\left[0,T\right],\,\phi\in\Lambda.
\end{align}
Our aim is to find a solution of \eqref{kolm} via the random variables $w^{t,\phi}_T\in \mathcal{H}$ satisfying \eqref{eq_B1} for every $t\in [0,T]$ and $\phi\in H$.
This is done in Theorem \ref{Kolm_thm}, for which we need  a couple of  preparatory results. 
\begin{lemma}\label{lemma_time_stint}
	There exists a constant $C_{\alpha,d}>0$ such that
	\begin{equation}\label{tr0}
		\norm{\int_{s}^{t}\left(\sigma\left(t\right)-\sigma\left(r\right)\right)\dd W_r}_\mathcal{H}\le C_{\alpha,d}\left|t-s\right|^\alpha,\quad 0\le s \le t \le T.
	\end{equation}
\end{lemma}
\begin{proof}
	Fix $0\le s \le t \le T$ and denote by $(e_k)_{k=1,\dots,d}$ the canonical basis of $\mathbb{R}^d$.  Using straightforward substitutions, by \eqref{kernel} we compute, for every $k=1,\dots, d,$
	\begin{multline}\label{ancora}
		\norm{\left(\sigma\left(t\right)-\sigma\left(r\right)\right)e_k}^2_2=\int_{0}^{T}\left|k_2\left(\xi-t\right)1_{\{\xi>t\}}-k_2\left(\xi-r\right)1_{\{\xi>r\}}\right|^2\dd \xi
		\\
		=\int_{0}^{t-r}\left|k_2\left(\xi\right)\right|^2\dd \xi+ \int_{0}^{T-t}\left|k_2\left(\xi+t-r\right)-k_2\left(\xi\right)\right|^2\dd \xi,\quad r\in\left[s,t\right].
	\end{multline}
Recalling  that (see \eqref{k_f}) $k_2\left(u\right)=\frac{1}{\Gamma(\alpha)}u^{\alpha-1},\,\alpha\in(1/2,1),\,u>0,$ for every $r\in\left[s,t\right]$ we have
\[
	\int_{0}^{t-r}\left|k_2\left(\xi\right)\right|^2\dd \xi=\frac{1}{(\Gamma(\alpha))^2(2\alpha-1)}\left|t-r\right|^{2\alpha-1},
\] 
and \[
	\int_{0}^{T-t}\left|k_2\left(\xi+t-r\right)-k_2\left(\xi\right)\right|^2\dd \xi\le\frac{1}{(\Gamma(\alpha))^2} \left(\int_{0}^{\infty}\left(\left(\xi+1\right)^{\alpha-1}-\xi^{\alpha-1}\right)^2\dd \xi\right)
	\left|t-r\right|^{2\alpha-1}.
\]
Therefore the discussion at the end of Page 98 in \cite{DP} ensures that \eqref{tr0} holds with $$
C_{\alpha,d}=\frac{\sqrt{d}}{\Gamma\left(\alpha\right)}\left(\frac{1}{2\alpha}\right)^{\frac{1}{2}}\left(\frac{1}{2\alpha-1}+\int_{0}^{\infty}\left(\left(\xi+1\right)^{\alpha-1}-\xi^{\alpha-1}\right)^2\dd \xi\right)^{\frac{1}{2}},$$ completing the proof.
\end{proof}
The following lemma analyzes some properties of the solution $w^{s,\phi}_t\in \mathcal{L}^p_t$ of \eqref{eq_B1} in the framework of Remark \ref{rem_p}. Recall that $\mathcal{L}_t^p= L_t^p\left(\Omega;L^p\right), \text{ where } L^p=L^p\big(0,T;\mathbb{R}^d\big)$, and that $L^p_{\ts}=L^p\big(\left(0,T\right)\times \left(0,T\right);\mathbb{R}^d\big)$. 
\begin{lemma}\label{p_lemma}
	Suppose that $B\colon\Lambda\to L^p_{\ts}$ satisfies Assumption \ref{B_1} and \eqref{ass_b_1'}, for some $p\in\big[2,\left(1-\alpha\right)^{-1}\big)$. Then there exists a constant $C_{1,p}=C_{1,p}\left(\alpha, d,T\right)>0$ such that
	\begin{equation}\label{bound_p}
		\norm{w^{s,\phi}_t}_{\mathcal{L}^p}\le C_{1,p}\left(1+\norm{\phi}_p\right),\quad 0\le s \le t\le T,\, \phi\in L^p.
	\end{equation}
	Furthermore, for every $\phi\in L^p$, there is  a constant $C_{\phi,p}=C_{\phi,p}(\alpha,d,T)>0$ such that 
	\begin{equation}\label{unpo+}
		\norm{w^{s,\phi}_t-\phi}_{\mathcal{L}^p}\le C_{\phi,p}\sqrt{t-s},\quad 0\le s \le t \le T.
	\end{equation}
\end{lemma}
\noindent When $p=2$, the hypotheses of Lemma \ref{p_lemma} reduce to Assumption \ref{B_1} and $\norm{\cdot}_{\mathcal{L}^p}=\norm{\cdot}_{\mathcal{H}}$.
\begin{proof}
	Fix $0\le s \le t\le T$ and $\phi\in L^p$. Recall that, under the hypotheses of the lemma, the unique solution  $w^{s,\phi}_t\in\mathcal{H}$ of \eqref{eq_B1} belongs to the space $\mathcal{L}^p_t$, see Remark \ref{rem_p}. \\
	Consider $N=N(d,p,T)\in \mathbb{N}$ so big that $C_{0,p}(2T/N)^{1-\frac{1}{p}}<1$, where $C_{0,p}$ is the constant appearing in \eqref{ass_b_1'}. Take an equispaced partition $\{t_k\}_{k=0}^N$ of $\left[s,t\right]$ with  $t_0=s$ and $t_N=t$: its mesh $\Delta\le T/N$. By \eqref{eq_B1}-\eqref{ass_b_1'} we have, using Bochner's theorem and Jensen's inequality, 
	\[
	\norm{w^{s,\phi}_{t_1}}_{\mathcal{L}^p}\le \norm{\phi}_p+C_{0,p}{\left(2\Delta\right)}^{1-\frac{1}{p}}\left(1+\norm{w^{s,\phi}_{t_1}}_{\mathcal{L}^p}\right)+\norm{\int_{s}^{t_1}\sigma\left(r\right)\dd W_r}_{\mathcal{L}^p},
	\]
	which in turn implies, by \eqref{PZbanach}, for some constant $c=c(d,p,T)>0$,
	\[
		\norm{w^{s,\phi}_{t_1}}_{\mathcal{L}^p}\le\left(1-C_{0,p}\left(2T/N\right)^{1-\frac{1}{p}}\right)^{-1}\left(\norm{\phi}_p+c\norm{k_2}_p+C_{0,p}\left(2T/N\right)^{1-\frac{1}{p}}\right).
	\]
	At this point, invoking  $N-$times the cocycle property in \eqref{cocycle} we obtain \eqref{bound_p}.
	
	As for \eqref{unpo+}, using \eqref{PZbanach}-\eqref{ass_b_1'} we compute, for some constant $C=C(d,p)>0,$ recalling the notation $\Sigma_{s,t}$ introduced in \eqref{serve_rem},
	\begin{align*}
		\norm{w^{s,\phi}_t-\phi}_{\mathcal{L}^p}
		&\le
		\mathbb{E}\left[\left( \int_{s}^{t}\norm{\widebar{B}\left(r,w^{s,\phi}_t\right)}_p\dd r\right)^p\right]^{\frac{1}{p}}+\norm{\Sigma_{s,t}}_{\mathcal{L}^p}\\&\le
		\left(t-s\right)^{1-\frac{1}{p}}\mathbb{E}\left[\int_{s}^t\dd r\int_{0}^{T}\left|\widebar{B}\left(w^{s,\phi}_t\right)\right|^p\left(r,\xi\right)\dd \xi\right]^{\frac{1}{p}}+C\norm{k_2}_p\sqrt{t-s}\\&
		\le 
		\sqrt{t-s}\left(C\norm{k_2}_p+2^{1-\frac{1}{p}}T^{\frac12-\frac1p}C_{0,p}\left(1+\norm{w^{s,\phi}_t}_{\mathcal{L}^p}\right)\right).
	\end{align*}
	Thus, by \eqref{bound_p} the proof is complete.
\end{proof}

We are now ready to prove the main result of the paper, which shows the connection between the solution $w^{t,\phi}_T,\,t\in[0,T],\,\phi\in H,$ of \eqref{eq_B1} and the backward Kolmogorov equation in integral form \eqref{kolm}.
\begin{theorem}\label{Kolm_thm}
	Suppose that $B\colon\Lambda \to L^p_{\ts}$ satisfies Assumption \ref{B_3} and \eqref{ass_b_1'}, for some $p\in\big(2,\left(1-\alpha\right)^{-1}\big)$. In addition, let the function $r\mapsto B(r,\phi)$ belong to $C\big([0,T];H\big)$, for every $\phi\in \Lambda$.  Fix $\Phi\in C^{2+\beta}_b\left(H\right)$ and define the map $u\colon \left[0,T\right]\times H\to\mathbb{R}$ by 
	\begin{equation}\label{sol_kol}
		u\left(t,\phi\right)=\mathbb{E}\left[\Phi\left(w^{t,\phi}_T\right)\right],\quad t\in\left[0,T\right],\,\phi\in H,
	\end{equation}
	where $w^{t,\phi}_T\in \mathcal{H}$ is the unique solution of \eqref{eq_B1}. Then $u\in L^\infty\big(0,T; C_b^{2+\beta}\left(H\right)\big)\cap C([0,T]\times H;\mathbb{R})$ and solves the Kolmogorov backward equation in integral form \eqref{kolm}. 
\end{theorem}
\begin{proof}
	The fact that the function $u$ defined in \eqref{sol_kol} belongs to $L^\infty\big(0,T; C_b^{2+\beta}\left(H\right)\big)\cap C([0,T]\times H;\mathbb{R})$ is one  of the results contained in Lemma~\ref{Reg_sol} (see Appendix \ref{appendix1}). Consequently, here we only focus on proving that $u$ solves \eqref{kolm}.
	
	Fix $0\le s <t\le T$ and $\phi \in \Lambda$. Since $\Lambda\subset \mathcal{H}^q_s,\,q\ge2, $ we can use \eqref{markoviala} in Corollary \ref{lip} to write
	\begin{equation}\label{mark}
		u\left(s,\phi\right)=\mathbb{E}\left[\mathbb{E}\left[\Phi\left(w^{s,\phi}_T\right)\Big| \mathcal{F}_t\right]\right]
		=
		\mathbb{E}\left[\restr{\mathbb{E}\left[\Phi\left(w_T^{t,\psi}\right)\right]}{\psi=w^{s,\phi}_t}\right]
		=\mathbb{E}\left[u\left(t,w^{s,\phi}_t\right)\right].
	\end{equation}
	Taylor's formula applied to the mapping $u\left(t,\cdot\right)\in C^{2+\beta}_b\left(H\right)$ yields, denoting by $h=w^{s,\phi}_t-\phi\in \mathcal{H}$,
	\begin{align}\label{expre_rem}
		\notag&u\left(t,w^{s,\phi}_t\right)-u\left(t,\phi\right)=
		\left\langle\nabla u\left(t,\phi\right),h\right\rangle_H
		+\frac{1}{2}
		\left\langle D^2u\left(t,\phi\right)h,h\right\rangle_H+r_{u\left(t,\cdot\right)}\left(\phi,w^{s,\phi}_t\right),\quad \text{ where }
		\\ &\qquad 
		r_{u\left(t,\cdot\right)}\left(x,y\right)=	\int_{0}^{1}\left(1-r\right)\left\langle \left(D^2u\left(t,x+r\left(y-x\right)\right)-D^2u\left(t,x\right)\right)\left(y-x\right), y-x\right\rangle_H\dd r,\quad x,\,y\in H.
	\end{align}
To keep the notation simple, in this proof we denote by $\widebar{B}_{s,t}(w^{s,\phi}_t)=\int_{s}^{t}\widebar{B}(r,w^{s,\phi}_t)\,\dd r\in\mathcal{H}$. Using the expression in  \eqref{eq_B1} for $h=w^{s,\phi}_t-\phi$ and noticing that $\mathbb{E}[\Sigma_{s,t}]=0\in H$ by \cite[Proposition 4.28]{DP}, we take expectations in the previous chain of equalities to obtain, from \eqref{mark},
\begin{align}\label{ste1}
\!\!\!	\notag u\left(s,\phi\right)\!-\!u\left(t,\phi\right)=&\left\langle \!\nabla u\left(t,\phi\right), \mathbb{E}\left[\int_{s}^{t}\widebar{B}\left(r,w^{s,\phi}_t\right)\dd r\right]\right\rangle_H
	\\& +\!\frac{1}{2}\mathbb{E}\left[\left\langle \!D^2u\left(t,\phi\right) \left(\!\widebar{B}_{s,t}\left(w^{s,\phi}_t\right)\!+\Sigma_{s,t}\right)\!
	,\!\widebar{B}_{s,t}\left(w^{s,\phi}_t\right)+\Sigma_{s,t}
	\right\rangle_H\right]\!\!
+\mathbb{E}\left[r_{u\left(t,\cdot\right)}\!\left(\phi,w^{s,\phi}_t\right)\right].
\end{align}

For all $N\in \mathbb{N}$, consider an equispaced partition $\{t^{(N)}_k\}_{k=0}^{N}$ of $[s,T]$ with mesh $\Delta_N$, where $t^{(N)}_0=s$ and $t^{(N)}_N=T$. By \eqref{ste1}, we have
\begin{align}\label{ste2}
	u&\left(s,\phi\right)-\Phi\left(\phi\right)=\sum_{k=1}^{N}\left(u\left(t_{k-1}^{(N)},\phi\right)-u\left(t_{k}^{(N)},\phi\right)\right)\notag\\
	&\notag=\sum_{k=1}^{N}
	\left\langle \nabla u\left(t_{k}^{(N)},\phi\right), \mathbb{E}\left[\widebar{B}_{t_{k-1}^{(N)},t_{k}^{(N)}}\left(w^{t_{k-1}^{(N)},\phi}_{t_{k}^{(N)}}\right)\right]\right\rangle_H
	\\
	&\notag\quad +\frac{1}{2}\sum_{k=1}^{N}
	\mathbb{E}\left[\left\langle D^2u\left(t_{k}^{(N)},\phi\right) \left(\widebar{B}_{t_{k-1}^{(N)},t_{k}^{(N)}}\left(w^{t_{k-1}^{(N)},\phi}_{t_{k}^{(N)}}\right)+\Sigma_{t_{k-1}^{(N)},t_{k}^{(N)}}\right)
	,\widebar{B}_{t_{k-1}^{(N)},t_{k}^{(N)}}\left(w^{t_{k-1}^{(N)},\phi}_{t_{k}^{(N)}}\right)+\Sigma_{t_{k-1}^{(N)},t_{k}^{(N)}}
	\right\rangle_H\right]
	\\
	&\quad+\sum_{k=1}^{N}\mathbb{E}\left[r_{u\left(t_k^{(N)},\cdot\right)}\left(\phi,w^{t_{k-1}^{(N)},\phi}_{t_k^{(N)}}\right)\right] \eqqcolon\mathbf{\upperRomannumeral{1}}^N+\mathbf{\upperRomannumeral{2}}^N+\mathbf{\upperRomannumeral{3}}^N.
\end{align}
In the sequel, we omit the superscript $N$ from the points of the partition to ease notation, i.e., we write $t_k$ for $t_k^{(N)}$. Firstly, we analyze $\mathbf{\upperRomannumeral{1}}^N$, which we decompose using the properties of the Bochner's integral as follows:
\begin{align*}
\notag	\mathbf{\upperRomannumeral{1}}^N&=
	\sum_{k=1}^{N}\left\langle
	\nabla u\left(t_{k},\phi\right), B\left(t_{k},\phi\right)
	 \right\rangle_H\left(t_k-t_{k-1}\right)
	\\&\quad+\sum_{k=1}^{N}\mathbb{E}\left[
	 \int_{t_{k-1}}^{t_{k}}\left\langle \nabla u\left(t_{k},\phi\right),\widebar{B}\left(r,w^{t_{k-1},\phi}_{t_{k}}\right)-B\left(r,\phi\right)\right\rangle_H\dd r\right]\notag\\&
	 \quad+\sum_{k=1}^{N}
	 \int_{t_{k-1}}^{t_{k}}\left\langle \nabla u\left(t_{k},\phi\right),B\left(r,\phi\right)-B\left(t_k,\phi\right)\right\rangle_H\dd r
	 \eqqcolon
	\mathbf{\upperRomannumeral{1}}^N_1+\mathbf{\upperRomannumeral{1}}^N_2+\mathbf{\upperRomannumeral{1}}^N_3.
\end{align*} 
Note that $\mathbf{\upperRomannumeral{1}}_1^N\to \int_{s}^{T}\langle\nabla u (r,\phi), B(r,\phi)\rangle_H\dd r$ as $N\to \infty$ by Lemma \ref{Reg_sol} in Appendix \ref{appendix1}. Next, Jensen's inequality, \eqref{ass_b_1'}, \eqref{unpo+}  and the continuous immersion $L^p\big({\left(t_{k-1},t_k\right)\times \left(0,T\right);\mathbb{R}^d}\big)\hookrightarrow L^2\big({\left(t_{k-1},t_k\right)\times \left(0,T\right);\mathbb{R}^d}\big)$ yield, for some constant $C_{\phi,p}=C_{\phi,p}(\alpha,d,T)>0$,
\begin{align*}
	\left|\mathbf{\upperRomannumeral{1}}^N_2\right|&\le
	\norm{\nabla u}_{\infty}\sqrt{\Delta_N}\sum_{k=1}^{N}\mathbb{E}\left[\left(\int_{t_{k-1}}^{t_k}\dd r \int_{0}^{T}\left|\widebar{B}\left(w^{t_{k-1},\phi}_{t_{k}}\right)-B\left(\phi\right)\right|^2(r,\xi)\,\dd \xi\right)^{\frac{1}{2}} 
\right]
	 \\
	 &
	 \le 
	 T^{\frac{1}{2}-\frac{1}{p}}\left({\Delta_N}\right)^{1-\frac{1}{p}}	\norm{\nabla u}_{\infty}
	 \sum_{k=1}^{N}\mathbb{E}\left[\norm{\widebar{B}\left(w^{t_{k-1},\phi}_{t_{k}}\right)-B\left(\phi\right)}_{p,\ts}\right]
	 \\&
	 \le T^{\frac{1}{2}-\frac{1}{p}}C_{0,p}\left({\Delta_N}\right)^{1-\frac{1}{p}}\norm{\nabla u}_{\infty}\sum_{k=1}^{N}\mathbb{E}\left[\norm{w^{t_{k-1},\phi}_{t_{k}}-\phi}_p\right]
	\le
	T^{\frac{3}{2}-\frac{1}{p}}C_{0,p}C_{\phi,p}\norm{\nabla u}_{\infty}
	\left({\Delta_N}\right)^{\frac{1}{2}-\frac{1}{p}}\underset{N\to \infty}{\longrightarrow}0.
\end{align*}
Here, we set $\norm{\nabla u}_\infty=\sup_{t\in [0,T]}\sup_{\phi\in H}\norm{\nabla u (t,\phi)}_2$.
Regarding $\mathbf{\upperRomannumeral{1}}^N_3$, we define the modulus of continuity of the map $B(\cdot,\phi)\colon [0,T]\to H$ by
\[
\mathfrak{w}\left(B(\cdot,\phi),\delta\right)=\sup_{\left|u-v\right|\le \delta}. \norm{B\left(u,\phi\right)-B\left(v,\phi\right)}_2,\quad \delta>0.
\]
Since, by hypothesis,  $B(\cdot,\phi)$ is continuous on the compact $[0,T]$, it is also uniformly continuous, hence we infer that 
$
\left|\mathbf{\upperRomannumeral{1}}_3^N\right|
\le 
T \norm{\nabla u}_\infty \mathfrak{w}\left(B\left(\cdot, \phi\right),\Delta_N\right)
\underset{N\to \infty}{\longrightarrow}0.
$
Therefore, we have just shown that 
\begin{equation}\label{1}
	\lim_{N\to \infty}\mathbf{\upperRomannumeral{1}}^N=\int_{s}^{T}\left\langle\nabla u (r,\phi), B(r,\phi)\right\rangle_H\dd r.
\end{equation}

Now we investigate $\mathbf{\upperRomannumeral{2}}^N,$ which we split as follows:
\begin{align*}
	2\mathbf{\upperRomannumeral{2}}^N=&
	\sum_{k=1}^{N}
	\mathbb{E}\left[\left\langle D^2u\left(t_{k},\phi\right) \widebar{B}_{t_{k-1},t_{k}}\left(w^{t_{k-1},\phi}_{t_{k}}\right)
	,\widebar{B}_{t_{k-1},t_{k}}\left(w^{t_{k-1},\phi}_{t_{k}}\right)
	\right\rangle_H\right]
	\\&+\sum_{k=1}^{N}
	\mathbb{E}\left[\left\langle D^2u\left(t_{k},\phi\right) \widebar{B}_{t_{k-1},t_{k}}\left(w^{t_{k-1},\phi}_{t_{k}}\right)
	,\Sigma_{t_{k-1},t_{k}}
	\right\rangle_H\right]\\
	&+\sum_{k=1}^{N}
	\mathbb{E}\left[\left\langle D^2u\left(t_{k},\phi\right) 
	\Sigma_{t_{k-1},t_{k}}
	,
	\widebar{B}_{t_{k-1},t_{k}}\left(w^{t_{k-1},\phi}_{t_{k}}\right)
	\right\rangle_H\right]
	\\&+\sum_{k=1}^{N}
	\mathbb{E}\left[\left\langle D^2u\left(t_{k},\phi\right) \Sigma_{t_{k-1},t_{k}}
	,\Sigma_{t_{k-1},t_{k}}
	\right\rangle_H\right]
	\eqqcolon\mathbf{\upperRomannumeral{2}}^N_1+\mathbf{\upperRomannumeral{2}}^N_2+\mathbf{\upperRomannumeral{2}}^N_3+\mathbf{\upperRomannumeral{2}}^N_4.
\end{align*}
Let us  set $\norm{D^2u}_\infty=\sup_{t\in [0,T]}\sup_{\phi\in H}\norm{D^2u(t,\phi)}_{\mathcal{L}(H;H)}$. By \eqref{ass_b_1'}-\eqref{bound_p}, arguing similarly to $\mathbf{\upperRomannumeral{1}}_2^N$ we have, for some $c>0$,
\begin{align*}
	\left|\mathbf{\upperRomannumeral{2}}_1^N\right|&\le \norm{D^2u}_\infty\sum_{k=1}^N\mathbb{E}\left[\norm{\widebar{B}_{t_{k-1},{t_k}}\left(w_{t_k}^{t_{k-1},\phi}\right)}_2^2\right]\\&
	\le T^{1-\frac{2}{p}}\Delta^{1-\frac{2}{p}}_N\norm{D^2u}_\infty\sum_{k=1}^N\mathbb{E}\left[\norm{\widebar{B}\left(w_{t_k}^{t_{k-1},\phi}\right)}^2_{p,\ts}\right]\left(t_k-t_{k-1}\right)
	\le
	c\Delta^{1-\frac{2}{p}}_N\norm{D^2u}_\infty\left(1+\norm{\phi}^2_p
	 %2cT^{2-\frac{2}{p}}\Delta^{1-\frac{2}{p}}_N\norm{D^2u}_\infty\!\left(\!1\!+\!\left(1+\norm{\phi}_p\right)^2%\mathbb{E}\left[\norm{w_{t_k}^{t_{k-1},\phi}}_p^2\right]
	 \right).
\end{align*}
Moreover, by H\"older's inequality and \eqref{iso}, for some  $\tilde{c}>0$,
\begin{equation*}
	\left|\mathbf{\upperRomannumeral{2}}^N_2\right|
	\le \norm{D^2u}_\infty\sum_{k=1}^N
	\norm{\widebar{B}_{t_{k-1},{t_k}}\left(w_{t_k}^{t_{k-1},\phi}\right)}_{\mathcal{H}}\norm{\Sigma_{t_{k-1},t_{k}}}_{\mathcal{H}}
	 \le 
 \norm{D^2u}_\infty\widetilde{c}\,T^{\frac{3}{2}-\frac{1}{p}}\norm{k_2}_2\Delta_N^{\frac{1}{2}-\frac{1}{p}}\left(1+\norm{\phi}_p\right).
\end{equation*}
  Since the second bound holds for $\mathbf{\upperRomannumeral{2}}^N_3$, too, we see that $\mathbf{\upperRomannumeral{2}}^N_i\to 0$ as $N\to \infty$, $i=1,2,3.$  \\As for $\mathbf{\upperRomannumeral{2}}^N_4$, we write it as the following sum:
\begin{align*}
	\mathbf{\upperRomannumeral{2}}^N_4
	=&
	\sum_{k=1}^N	\mathbb{E}\left[\left\langle D^2u\left(t_{k},\phi\right) \int_{t_{k-1}}^{t_k}\sigma\left(t_k\right)\dd W_r
	,\int_{t_{k-1}}^{t_k}\sigma\left(t_k\right)\dd W_r
	\right\rangle_H\right]\\&+
		\sum_{k=1}^N	\mathbb{E}\left[\left\langle D^2u\left(t_{k},\phi\right) \int_{t_{k-1}}^{t_k}\left(\sigma\left(r\right)-\sigma\left(t_k\right)\right)\dd W_r
	,\int_{t_{k-1}}^{t_k}\sigma\left(t_k\right)\dd W_r
	\right\rangle_H\right]\\&+
		\sum_{k=1}^N	\mathbb{E}\left[\left\langle D^2u\left(t_{k},\phi\right) \int_{t_{k-1}}^{t_k}\sigma\left(t_k\right)\dd W_r
	,\int_{t_{k-1}}^{t_k}\left(\sigma\left(r\right)-\sigma\left(t_k\right)\right)\dd W_r
	\right\rangle_H\right]\\&+	\sum_{k=1}^N	\mathbb{E}\left[\left\langle D^2u\left(t_{k},\phi\right) \int_{t_{k-1}}^{t_k}\left(\sigma\left(r\right)-\sigma\left(t_k\right)\right)\dd W_r
	,\int_{t_{k-1}}^{t_k}\left(\sigma\left(r\right)-\sigma\left(t_k\right)\right)\dd W_r
	\right\rangle_H\right]\\\eqqcolon&
	\mathbf{\upperRomannumeral{2}}^N_{4,1}+\mathbf{\upperRomannumeral{2}}^N_{4,2}+\mathbf{\upperRomannumeral{2}}^N_{4,3}+\mathbf{\upperRomannumeral{2}}^N_{4,4}.
\end{align*}
By \cite[Proposition $4.30$]{DP}, we have, for every $k=1,\dots,N$,
\[
	D^2u\left(t_{k},\phi\right) \int_{t_{k-1}}^{t_k}\sigma\left(t_k\right)\dd W_r=\int_{t_{k-1}}^{t_k}D^2u\left(t_{k},\phi\right)\sigma\left(t_k\right)\dd W_r,\quad \mathbb{P}-\text{a.s.},
\]
whence, by \cite[Corollary $4.29$]{DP} and  Lemma \ref{Reg_sol}, 
\begin{align*}
	\notag&\mathbf{\upperRomannumeral{2}}^N_{4,1}
	=\sum_{k=1}^N\text{Tr}\left( D^2u\left(t_k,\phi\right)\sigma\left(t_k\right)\sigma\left(t_k\right)^\ast\right)\left(t_k-t_{k-1}\right)\underset{N\to\infty}{\longrightarrow}\int_{s}^{T}\text{Tr}\left( D^2u\left(r,\phi\right)\sigma\left(r\right)\sigma\left(r\right)^\ast\right)\dd r.
\end{align*}
Furthermore,  H\"older's inequality, \eqref{tr0} in Lemma \ref{lemma_time_stint} and \cite[Proposition $4.20$]{DP} yield, for $i=2,3$, for some constants $c_1,c_2>0$,
\begin{align*}
	\left|\mathbf{\upperRomannumeral{2}}_{4,i}^N\right| 
	\le c_1\norm{k_2}_2
	\norm{D^2u}_{\infty}\sqrt{\Delta_N}
	\sum_{k=1}^N\norm{\int_{t_{k-1}}^{t_k}\left(\sigma\left(r\right)-\sigma\left(t_k\right)\right) \dd W_r}_\mathcal{H}\le Tc_2\norm{k_2}_2\norm{D^2u}_{\infty}\Delta_N^{\alpha-\frac{1}{2}}\underset{N\to\infty}{\longrightarrow}0.
\end{align*}
Analogous estimates show that $\mathbf{\upperRomannumeral{2}}_{4,4}^N\to 0$ as $N\to \infty$, as well.
Thus,
\begin{equation}\label{2}
	\lim_{N\to \infty}\mathbf{\upperRomannumeral{2}}^N
	=\frac{1}{2}
\int_{s}^{T}\text{Tr}\left( D^2u\left(r,\phi\right)\sigma\left(r\right)\sigma\left(r\right)^\ast\right)\dd r.
\end{equation}

At last we study the remainder term $\mathbf{\upperRomannumeral{3}}^N$ in \eqref{ste2}. To do this, we employ the fact that  $D^2u\left(t,\cdot\right)\colon H\to \mathcal{L}\left(H;H\right)$  is $\beta-$H\"older continuous uniformly in time, see \eqref{unif_time} in Lemma \ref{Reg_sol}. We choose $\tilde{\beta}\in\left(0,\beta\right)$ such that $2+\tilde{\beta}<p$; by the expression of $r_{u(t_k,\cdot)}$ in \eqref{expre_rem} we deduce that
\begin{align}\label{3}
	\notag\left|\mathbf{\upperRomannumeral{3}}^N\right|&\le 
	\sum_{k=1}^N\int_{0}^{1}\mathbb{E}\left[\norm{D^2u\left(t_k,\phi+r\left(w_{t_k}^{t_{k-1},\phi}-\phi\right)-D^2u\left(t_k,\phi\right)\right)}_{\mathcal{L}\left(H;H\right)}
	\norm{w_{t_k}^{t_{k-1},\phi}-\phi}_2^2\right] \dd r
	\\&\notag\le C
	\sum_{k=1}^N\mathbb{E}\left[\norm{w_{t_k}^{t_{k-1},\phi}-\phi}_2^{2+{\tilde{\beta}}}\right]
	\le
	CT^{\left(\frac{1}{2}-\frac{1}{p}\right)\left(2+\tilde{\beta}\right)}\sum_{k=1}^N\mathbb{E}\left[\norm{w_{t_k}^{t_{k-1},\phi}-\phi}_p^{2+{\tilde{\beta}}}\right]
	\\&\le  C\sum_{k=1}^{N}\left(t_k-t_{k-1}\right)^{1+\frac{\tilde{\beta}}{2}}
\underset{N\to \infty}{\longrightarrow}0,
\end{align}
where in the last passage we use  Lemma \ref{p_lemma} and Jensen's inequality. Here $C_{}>0$ is a constant allowed to change from line to line. Combining \eqref{1}, \eqref{2}, \eqref{3} in \eqref{ste2}, we obtain
\[
	u\left(s,\phi\right)-\Phi\left(\phi\right)
	=
	\int_{s}^{T}\left\langle\nabla u \left(r,\phi\right), B\left(r,\phi\right)\right\rangle_H \dd r
	+
	\frac{1}{2}
	\int_{s}^{T}\text{Tr}\left( D^2u\left(r,\phi\right)\sigma\left(r\right)\sigma\left(r\right)^\ast\right)\dd r,
\]
i.e., \eqref{kolm}. Thus, the proof is complete.
\end{proof}
\begin{rem}\label{rem_diff}
	Under the hypotheses of Theorem \ref{Kolm_thm}, for every $\phi\in\Lambda$ the function $u\left(\cdot, \phi\right)\colon[0,T]\to \mathbb{R}$ defined in \eqref{sol_kol} is absolutely continuous on $[0,T]$, because the integrands on the right--hand side of \eqref{kolm} are bounded on $[0,T]$. Thus, the fundamental theorem of calculus shows that $u\colon [0,T]\times H \to \mathbb{R}$ satisfies the following \emph{Kolmogorov backward equation} in differential form:
	\begin{equation*}\label{kolm_diff}
		\begin{cases}
			\partial_tu\left(t,\phi\right)+\left\langle\nabla u\left(t,\phi\right),B\left(t,\phi\right)\right\rangle_H+\frac{1}{2}\emph{Tr}\left(D^2u\left(t,\phi\right)\sigma\left(t\right)\sigma\left(t\right)^\ast\right) =0,&\text{for a.e. }t\in\left(0,T\right),\,\phi\in\Lambda,\\
			u\left(T,\phi\right)=\Phi\left(\phi\right),\quad \phi\in H.
		\end{cases}
	\end{equation*}
\end{rem}
\begin{rem}\label{vabene}
All the arguments and computations leading to Theorem \ref{Kolm_thm} continue to hold when the power $\alpha$ of the kernel $k_2$ in \eqref{k_f} varies in $[1,\frac{3}{2})$, i.e., $k_2$ is the continuous kernel in $\mathbb{R}_+$ given by
\[
		k_2(t)=\frac{1}{\Gamma(\alpha)}t^{\alpha-1},\quad t\ge0,\text{ for some }\alpha\in\left[1,\frac{3}{2}\right).
\]
We have however decided to present the theory in the case $\alpha\in (\frac12,1)$ to emphasize the fact that our approach is able to handle rough kernels with explosions at $t=0$.
\end{rem}
\begin{ex}\label{ex1}
	Given two continuous maps $A\colon[0,T]\to\mathbb{R}^{d\times d}$ and $b\colon[0,T]\to \mathbb{R}^d$,  define  $B\colon \Lambda\to H_{\ts}$ by 	(cfr.~\eqref{B})
	\begin{equation}\label{continuare}
		B(w)\colon [0,T]\times [0,T] \to \mathbb{R}^d \quad \text{ such that }\quad  B(w)(t,\xi)=1_{\{\xi>t\}}k_2(\xi-t)\,\left(A(t)w(t)+b(t)\right),\,t,\xi\in[0,T], 
	\end{equation}
	$\text{for every }w\in\Lambda.$
We now show that $B$ satisfies all the hypotheses of Theorem \ref{Kolm_thm}. 

For every $t\in (0,T]$ and $r\in (0,t)$, from the definition in \eqref{continuare} it is immediate to see that $B(w)(r,\xi)=0, \,\xi\in(0,r)$, and that $B(w)(r,\cdot)$ depends on $w$ only via $\restr{w}{(0,t)}.$ Denote by $\norm{A}_\infty=\sup_{t\in[0,T]}|A(t)|$ and by $\norm{b}_{\infty}=\sup_{t\in[0,T]}|b(t)|$, where $|A(t)|$ is the operator norm in $\mathbb{R}^{d\times d}.$ Computing,  for every $w_1,\,w_2\in \Lambda$,
	\begin{multline*}
	\norm{B(w_1)}^2_{\ts}\le \int_{0}^T\left(\int_{0}^T\left|k_2(\xi-t)\right|^21_{\{\xi>t\}}\left(\norm{b}_\infty+\norm{A}_\infty\left|w_1(t)\right|\right)^2\dd\xi\right)\dd t
	\\\le 2T\max\left\{\norm{b}^2_\infty, \norm{A}^2_\infty\right\}\norm{k_2}_2^2(1+\norm{w_1}^2_2),
	\end{multline*}
	and
	\[
	\norm{B(w_2)-B(w_1)}^2_{\ts}\le \norm{A}^2_\infty\int_{0}^T\left(\int_{0}^T\left|k_2(\xi-t)\right|^21_{\{\xi>t\}}\left|w_2(t)-w_1(t)\right|^2\dd\xi\right)\dd t\le \norm{A}^2_\infty\norm{k_2}_2^2\norm{w_2-w_1}^2_2,
	\]
	we deduce that Assumption \ref{B_1} is satisfied. Since the previous computations can be repeated for every $p\in (2,(1-\alpha)^{-1})$, then  condition \eqref{ass_b_1'} in Remark \ref{rem_p} is verified, as well. 
	\\
	As for Assumption \ref{B_2}, evidently  the operator $DB(w_1)\in \mathcal{L}\big(\Lambda_2;H_{\ts}\big)$ defined by 
	\begin{equation}\label{cont1}
		[DB(w_1)(w_2)](t,\xi)=1_{\{\xi>t\}}k_2(\xi-t)A(t)w_2(t),\quad t,\xi\in[0,T],\,w_2\in \Lambda_2,
	\end{equation}
	is the $\Lambda_2-$Fréchet differential of $B$ in $w_1$, for any $w_1 \in\Lambda$. Indeed,
	\begin{equation*}
		{B(w_1+h)-B(w_1)-DB(w_1)(h)}=0,\quad w_1, h\in\Lambda. 
	\end{equation*}
	Moreover, from \eqref{cont1} we have, for every $w_1,w_2\in\Lambda$,
	\[
	\norm{DB(w_1)(w_2)}_{\ts}\le \norm{A}_{\infty}\norm{k_2}_2\norm{w_2}_2,\qquad 
		\norm{DB(w_1)-DB(w_2)}_{\mathcal{L}(\Lambda_2;H_{\ts})}=0,
	\]
	which in particular gives \eqref{ho_alpha} with $\gamma=1$. \\
	The requirements of Assumption \ref{B_3} are trivially satisfied (with $\beta=1$) because, given the affine structure of this example, $D^2B(w_1)=0\in\mathcal{L}(\Lambda_2,\Lambda_2;H_{\ts}),\,w_1\in\Lambda.$
\\
	In conclusion,  for every $w\in \Lambda$, the map $t\mapsto B(t,w)=B(w)(t,\cdot)$ is continuous from $[0,T]$ to $H$. Indeed, denoting by $\tilde{b}(t)$ the $\mathbb{R}^d-$valued continuous function $A(t)w(t)+b(t)$, by \eqref{ancora} and the two following equations we have, for any $r,t\in [0,T]$,
	\begin{align*}
	&	\norm{B(t,w)-B(r,w)}^2_2=\int_{0}^{T}
	\left|k_2\left(\xi-t\right)1_{\{\xi>t\}}\tilde{b}(t)-k_2\left(\xi-r\right)1_{\{\xi>r\}}\tilde{b}(r)\right|^2\dd \xi
\\&
\le 2\big|\!\big|{\tilde{b}}\big|\!\big|^2_\infty\!\int_{0}^{T}\!	\left|k_2\left(\xi-t\right)\!1_{\{\xi>t\}}\!-\!k_2\left(\xi-r\right)\!1_{\{\xi>r\}}\right|^2\dd \xi\!+\!2\norm{k_2}_2^2\left|\tilde{b}(t)-\tilde{b}(r)\right|^2\!\!\le\! L\!\left(\!\left|t-r\right|^{2\alpha-1}\!\!+\!\left|\tilde{b}(t)-\tilde{b}(r)\right|^2\right),
	\end{align*}
	for some constant $L>0$.
\end{ex}
	\section{The mild Kolmogorov equation}\label{mde}
A classical approach to the study of the Kolmogorov equation is its mild formulation, see for example \cite[Section 6.5]{DP2} and \cite[Section 9.5]{DP}. Contrary to the strategy adopted in the previous section, where we have constructed a solution to \eqref{kolm} via a stochastic equation (cfr. Theorem \ref{Kolm_thm}), for the  mild Kolmogorov equation we look for a \emph{direct} solution. With the term \emph{direct}, we mean a solution which is determined by a fixed point argument, hence which does not rely on the underlying stochastic PDE. 

In this section, we first present a formal reasoning leading to the mild form of \eqref{kolm}, see \eqref{mild}. After that, in Subsection \ref{grad_est} we explain some difficulties in proving the well–posedness of such a mild formulation, which are essentially due to the structure of the noise.  Since it not the purpose of this section to present a general theory with abstract hypotheses, we limit ourselves to observe that the mild Kolmogorov equation cannot be solved for a class of interesting drifts $b$  using common techniques (cfr. Lemma \ref{neg}). Finally, in Subsection \ref{reg_noise}, we highlight the theoretical importance of the mild Kolmogorov equation. In particular, we sketch a procedure –relying on the mild form– typically used to prove uniqueness in law for a stochastic PDE under weak regularity requirements on the coefficients. 
We only mention that studying the relation between the transition semigroup of an SDE and the corresponding  mild Kolmogorov equation can also be used  for numerical applications, as recently investigated by \cite{FLR} in the Brownian case and \cite{BO} in the case of isotropic, stable Lévy processes. 

Let 	$\mathcal{C}=C_{b}\left(  H;\mathbb{R}\right)$  and consider the backward Kolmogorov equation in differential form, formally written as
\begin{equation}\label{strong}
	\begin{cases}
	\partial_{s}v\left(  s,x\right)  +\left\langle b\left(  s,x\right)  ,\nabla v\left(
	s,x\right)  \right\rangle_H +\frac{1}{2}\text{Tr}\left(
	D^{2}v\left(  s,x\right)  \sigma\left(  s\right)\sigma\left(  s\right)^\star  \right) =0,\qquad s\in\left[  0,T\right),\,x\in H,\\
	v\left(
	T,x\right)  =\phi\left(  x\right),\quad \phi\in\mathcal{C}.
\end{cases}
\end{equation}
Here, $H$ and $\sigma$ are those of the previous sections (see, in particular, \eqref{kernel}), whereas the drift $b\colon [0,T]\times H\to H$ is a bounded measurable map which could be non--smooth. \\
We reformulate \eqref{strong} in order to study it in the space $\mathcal{C} $. Let $u\left(  t,x\right)  \coloneqq v\left(  T-t,x\right)  $: $u$ solves the forward equation
\begin{equation}\label{differ eq}
	\begin{cases}
		\partial_{t}u\left(  t,x\right)  =\mathcal{A}_{T-t} u\left(
		t,x\right)  +\left\langle b\left(  T-t,x\right)  , \nabla u\left(  t,x\right)
		\right\rangle_H,\qquad  t\in\left(0,T\right],\,x\in H,\\
		u\left(  0,x\right)
		=\phi\left(  x\right),\qquad \phi\in\mathcal{C},
	\end{cases}
\end{equation}
where we set 
\[
\mathcal{A}_{T-t} f\left(  x\right)  = \frac{1}{2}\text{Tr}\left(
D^{2}f\left(  x\right)  \sigma\left(
T-t\right) \sigma\left(  T-t\right) ^\ast \right).
\]
Fix $s\in\left[  0,T\right]$. For every $t\in [s,T]$, we define the linear evolution operator $R_T\left(  t,s\right)\colon \mathcal{C}\to \mathcal{C}$ by 
\[
\left(  R_T\left(  t,s\right)  \phi\right)  \left(  x\right)  = 
\mathbb{E}
\left[  \phi\left(  x+\int_{s}^{t}\sigma\left(  T-r\right)  \dd W_{r}\right)
\right],\quad x\in H,\,\phi\in\mathcal{C},
\]
where $W$ is an $\mathbb{R}^d-$valued, standard Brownian motion as the one introduced in Section \ref{section abstract}. Consider the auxiliary equation
\begin{equation}\label{aux_eq}
	\begin{cases}
		\partial_{t}z\left(  t,x\right)   =\mathcal{A}_{T-t}  z\left(
		t,x\right),  \qquad t\in\left(s,T\right],\,x\in H, \\
		z\left(  s,x\right)    =\phi\left(  x\right),\quad \phi\in\mathcal{C};
	\end{cases}
\end{equation}
if $\phi\in{C}_b^{2+\beta}(H)$, then Theorem \ref{Kolm_thm} and Remark \ref{rem_diff} imply that the function $(R_T(t,s)\phi)(x)$ solves this Cauchy problem  for almost every $t\in(s,T)$, for every $x\in \Lambda$.
%We know that this Cauchy
%problem has a unique smooth solution $z^{s,\psi}\left(  t,\cdot\right)  $
%which depends linearly on $\psi$, hence there exists a family of linear
%operators $U_{T}\left(  t,s\right),\,0\leq s\leq t\leq T,$ mapping $\mathcal{C}$ into $\mathcal{C}$ defined by
%\[
%z^{s,\psi}\left(  t,\cdot\right)  =U_{T}\left(  t,s\right)  \psi,\quad \psi\in \mathcal{C}.
%\]
At this point, we can introduce the mild formulation of the Kolmogorov equation \eqref{kolm}:
\begin{equation}\label{mild}
	u\left(  t,x\right)  =\left(  R_{T}\left(  t,0\right)  \phi\right)  \left(
	x\right)  +\int_{0}^{t}\left(  R_{T}\left(  t,s\right)  \left\langle b\left(
	T-s,\cdot\right)  , \nabla u\left(  s,\cdot\right)  \right\rangle_H \right)  \left(
	x\right)  \dd s,\quad \phi\in\mathcal{C}.
\end{equation}

Note that, heuristically speaking, \eqref{mild} corresponds to the Kolmogorov equation \eqref{differ eq}. Indeed, if $u\left(
t,x\right)  $ solves \eqref{mild}, then a formal application of Leibnitz integral rule and \eqref{aux_eq} yield
\begin{align*}
	\partial_{t}u (t,\cdot) & =\partial_{t}R_{T}\left(  t,0\right)  \phi+R_{T}\left(
	t,t\right)  \left\langle b\left(  T-t,\cdot\right)  ,\nabla u\left(  t,\cdot\right)
	\right\rangle_H+\int_{0}^{t}\partial_{t}R_{T}\left(  t,s\right)  \left\langle
	b\left(  T-s,\cdot\right)  , \nabla u\left(  s,\cdot\right)  \right\rangle_H \dd s\\
	& =\mathcal{A}_{T-t}  R_{T}\left(  t,0\right)  \phi+\left\langle
	b\left(  T-t,\cdot\right)  ,\nabla u\left(  t,\cdot\right)  \right\rangle_H+\int%
	_{0}^{t}\mathcal{A}_{T-t} R_{T}\left(  t,s\right)  \left\langle
	b\left(  T-s,\cdot\right)  ,\nabla u\left(  s,\cdot\right)  \right\rangle_H \dd s\\
	& =\mathcal{A}_{T-t}  u+\left\langle b\left(  T-t,\cdot\right)
	,\nabla u\left(  t,\cdot\right)  \right\rangle_H.
\end{align*}

As we have already mentioned, the aim is to prove directly, i.e., by a fixed point argument not relying on a stochastic equation, that \eqref{mild} admits a solution of class, e.g., $C\left(\left[  0,T\right]  ;\mathcal{C}\right)  $.  In this regards,  the regularity properties of the evolution operator $R_{T}\left(  t,s\right)  $ are paramount, hence we now discuss them. \\
According to \cite[Proposition 4.28]{DP}, the $H-$valued random variable $\int_{s}
^{t}\sigma\left(  T-r\right)  \dd W_{r}$ is Gaussian, centered, with covariance
operator%
\begin{align}\label{esagera}
	Q_{T}\left(  t,s\right) =\int_{s}^{t}\sigma\left(  T-r\right)
	\sigma\left(  T-r\right)^\ast  \dd r
	=\int_{T-t}^{T-s}\sigma\left(  \tau\right)  \sigma\left(
	\tau\right)^\ast  \dd \tau.
\end{align}
%meaning that $Q_{T}\left(  t,s\right)  h=\int_{s}^{t}	\left\langle \sigma\left(  T-r\right)  ,h\right\rangle_H\sigma\left(  T-r\right)
%\dd r,\,h\in
%H$. 
This covariance operator is not trivial as it would be in the case of constant $\sigma$. In fact, in such a case it would be easy to see that $R_T\left(t,s\right)  \phi,\,\phi\in\mathcal{C},$ is differentiable in the direction $\sigma$ (and only in
this direction). In our framework with a  time–varying $\sigma$, the question of the directions of differentiability of $R_T\left(  t,s\right)  \phi,\,\phi\in\mathcal{C},$ is much more complex. Nevertheless, it has to be addressed, because the directional differentiability of $R_T\left(  t,s\right)\phi $ is essential to solve	directly \eqref{mild}. This may be seen in various ways, one of which is the change of variable
\[
\theta_{T}\left(  t,x\right)  =\left\langle b\left(  T-t,x\right)  , \nabla u\left(
t,x\right)  \right\rangle_H,
\]
that leads to the study of the equation
\begin{multline}\label{integral theta}
	\theta_{T}\left(  t,x\right)  =\left\langle b\left(  T-t,x\right)  , \nabla \left(
	R_{T}\left(  t,0\right)  \phi\right)  \left(  x\right)  \right\rangle_H
	\\+\int_{0}^{t}\left\langle b\left(  T-t,x\right)  ,\nabla\left(  R_{T}\left(
	t,s\right)  \theta_{T}\left(  s,\cdot\right)  \right)  \left(  x\right)
	\right\rangle_H\dd s,\quad \phi\in\mathcal{C}.
\end{multline}
If we can prove that, for some $C,\epsilon>0$,
\begin{equation}\label{gradient estimate}
	\sup_{x\in H}\left\vert \left\langle b\left(  T-t,x\right)  , \nabla \left(
	R_{T}\left(  t,s\right)  \psi\right)  \left(  x\right)  \right\rangle_H
	\right\vert \leq\frac{C}{\left\vert t-s\right\vert ^{1-\epsilon}}\left\Vert
	\psi\right\Vert _{\infty},\quad 0\leq s< t\leq T,\,\psi\in \mathcal{C},
\end{equation}
then we may try to set up a fixed point
argument for the $\theta_{T}-$equation \eqref{integral theta} in a suitable space of bounded, 
measurable functions. This would in turn give a solution for equation \eqref{mild}
by simply setting%
\[
u\left(  t,x\right)  =\left(  R_{T}\left(  t,0\right)  \phi\right)  \left(
x\right)  +\int_{0}^{t}\left( R_{T}\left(  t,s\right)  \theta_{T}\left(
s,\cdot\right)  \right)  \left(  x\right)  \dd s.
\]

\subsection{The gradient estimate}\label{grad_est}

Using the Gaussian structure of the $H-$valued random variable 
\[
Z_{T}\left(  t,s\right)  = \int_{s}^{t}\sigma\left(  T-r\right)  \dd W_{r},\quad 0\le s < t \le T,
\]
and denoting by $Q_T(t,s)^{-1}$ the pseudo–inverse of $Q_T(t,s)$, one can prove –via the Cameron Martin formula (see, e.g., \cite[Theorem $2.23$]{DP})– that 
\[
\left\langle b\left(  T-t,x\right)  ,\nabla\left(  R_{T}\left(  t,s\right)
\psi\right)  \left(  x\right)  \right\rangle_H
=\mathbb{E}\left[  \left\langle
Q_{T}\left(  t,s\right)  ^{-1}b\left(  T-t,x\right)  ,Z_{T}\left(  t,s\right)
\right\rangle_H \psi\left(  x+Z_{T}\left(  t,s\right)  \right)  \right],\quad \psi\in \mathcal{C},
\]
if
\[
b\left(  T-t,x\right)  \in \text{Range}\left(  Q_{T}\left(  t,s\right)  \right)  .
\]
This is not the most general condition to obtain the existence of such directional derivative. Indeed, we could split $Q_{T}\left(
t,s\right)  ^{-1}$ and use the fact that $Q_{T}\left(  t,s\right)
^{-1/2}Z_{T}\left(  t,s\right)  $ has good properties, which reduces the problem to
investigating $b\left(  T-t,x\right)\in\text{Range }\big(Q_{T}\left(  t,s\right) ^{1/2}\big)  $. However, 
handling the square root is even more difficult and thus, for the time being,
we analyze the more restrictive condition.

When the previous holds, arguing as in \eqref{iso}, for some $c>0$ we have
\begin{align*}
	\sup_{x\in H}\left\vert \left\langle b\left(  T-t,x\right)  ,\nabla \left(
	R_{T}\left(  t,s\right)  \psi\right)  \left(  x\right)  \right\rangle_H
	\right\vert  & \leq\left\Vert \psi\right\Vert _{\infty}\sup_{x\in H}%
	\mathbb{E}\left[  \left\vert \left\langle Q_{T}\left(  t,s\right)
	^{-1}b\left(  T-t,x\right)  ,Z_{T}\left(  t,s\right)  \right\rangle_H
	\right\vert \right] \\
	& \leq c
	\left\Vert \psi\right\Vert _{\infty}\norm{k_2}_2\left(  t-s\right)  ^{1/2}\sup_{x\in
		H}\left\Vert Q_{T}\left(  t,s\right)  ^{-1}b\left(  T-t,x\right)  \right\Vert
	_{2}.
\end{align*}
Therefore a sufficient condition for the gradient estimate
\eqref{gradient estimate} is%
\[
\sup_{x\in H}\left\Vert Q_{T}\left(  t,s\right)  ^{-1}b\left(  T-t,x\right)
\right\Vert _{2}\leq\frac{C}{\left\vert t-s\right\vert ^{\frac{3}{2}-\epsilon
}},\quad 0\le s <t\le T,\text{ for some }C>0.
\]

For a general $b$, standing the potentially very strong degeneracy of
$Q_{T}\left(  t,s\right)  $, we do not see any hope to prove the gradient
estimate \eqref{gradient estimate}. A particular case that, a priori, may look
promising, is when the Volterra drift is  of the same kind as the noise part,
namely (cfr. \eqref{kernel})
\begin{align*}
[	b\left(  t,x\right) ](\xi) = \bar{\beta}\left(  x\right)k_2\left(\xi-  t\right)1_{\{t<\xi\}}
=
[\sigma(t)\bar{\beta}(x)]\left(\xi\right)
,\quad\xi\in[0,T], \text{ for some }
	\bar{\beta} \in \mathcal{B}_b\left(H;\mathbb{R}^d\right).
\end{align*}
In this case, since $b\left(  T-t,x\right)  =
 \sigma\left(  T-t\right) \bar{\beta}\left(  x\right)$, we need to prove that
\begin{equation}
	\sigma\left(  T-t\right) e_k \in \text{Range}\left(  Q_{T}\left(  t,s\right)  \right),\quad k=1,\dots, d,
	\label{sigma property}%
\end{equation}
and that
\[
\left\Vert Q_{T}\left(  t,s\right)  ^{-1}\sigma\left(  T-t\right) e_k \right\Vert
_{2}
\leq\frac{C}{\left\vert t-s\right\vert ^{\frac{3}{2}-\epsilon}},\quad 0\le s <t\le T,\,k=1,\dots,d, \text{ for some }C>0,
\]
where  $(e_k)_{k=1,\dots,d}$ is the  canonical basis of $\mathbb{R}^d$. 
Recalling that, by \eqref{esagera}, $Q_{T}\left(  t,s\right)  =\int_{T-t}^{T-s}\sigma\left(
\tau\right)\sigma\left(  \tau\right)^\ast  \dd\tau$, apparently we could
think that \eqref{sigma property} is true. But it is not, as the necessary condition given by the next lemma shows. 

\begin{lemma}\label{neg}
	Let $0\le s < t \le T$ and suppose that $f\in \emph{Range }(Q_T(t,s))\subset H$. Then $f=g$ almost everywhere in $(0,T)$, where $g\colon(0,T)\to \mathbb{R}^d$ is a continuous function such that $g=0$ in $(0,T-t)$.
\end{lemma}
\begin{proof}
	Fix $0\le s < t\le T$. Consider $f\in \text{Range }(Q_T(t,s))$, so that there exists $v\in H$ such that, by \eqref{esagera},
$
		f=\int_{T-t}^{T-s}\sigma\left(\tau\right)\sigma\left(\tau\right)^\ast v_{} \dd \tau.
$
In particular, for every  $k=1,\dots,d$, denoting by $\boldsymbol{\cdot}$ the scalar product in $\mathbb{R}^d$, by the standard properties of Bochner's integral we obtain
\[
f\boldsymbol{\cdot} e_k=
\left(\int_{T-t}^{T-s}\left( \sigma\left(\tau\right)^\ast \!v_{}\right) \,\,k_2\left(\cdot-\tau\right)1_{\{\cdot>\tau\}} \dd \tau\right) 
\boldsymbol{\cdot} e_k
=
\int_{T-t}^{T-s}\left\langle \sigma\left(\tau\right)e_k,v_{}\right\rangle_H k_2\left(\cdot-\tau\right)1_{\{\cdot>\tau\}} \dd \tau
.
\]
Furthermore, recalling \eqref{k_f}, $\text{for a.e. }\xi \in (0,T)$ we have
	\begin{equation*}
		\left(f\boldsymbol{\cdot} e_k\right)\left(\xi\right)
		=\frac{1}{\Gamma\left(\alpha\right)}
		\int_{T-t}^{T-s}1_{\left\{\tau<\xi\right\}}
		\left\langle \sigma\left(\tau\right)e_k,v_{}\right\rangle_H
		 \left(\xi-\tau\right)^{\alpha-1}\dd \tau
		. 
	\end{equation*}
	We denote by $g_k$ the function appearing on the right--hand side of the previous equation, i.e.,
	\[
	g_k\left(\xi\right) =\frac{1}{\Gamma\left(\alpha\right)}
	\int_{T-t}^{T-s}1_{\left\{\tau<\xi\right\}}
	\left\langle \sigma\left(\tau\right)e_k,v_{}\right\rangle_H
	\left(\xi-\tau\right)^{\alpha-1}\dd \tau,\quad \xi\in\left(0,T\right).
	\]
	We want to show the continuity of $g_k$ on the interval $[T-t,T)$: this ensures that $g_k$ is continuous on the whole $(0,T)$, since trivially $g_k=0$ on $(0,T-t]$. We first  write
	\begin{equation*}
		g_k\left(\xi\right)
		=
		\int_{0}^{\xi}1_{\left\{\tau>T-t\right\}}
			\left\langle \sigma\left(\tau\right)e_k,v_{}\right\rangle_H
			\left(\xi-\tau\right)^{\alpha-1}
		\dd \tau,\quad\xi \in [T-t,T-s],
	\end{equation*}
 and notice that, as  $\sigma(\cdot) e_k \in C([0,T];H)$ (see \eqref{ancora} in the proof of Lemma \ref{lemma_time_stint}), the mapping $\langle\sigma(\cdot)e_k, v_{}\rangle_H$ is continuous on $ [0,T]$. Therefore we invoke \cite[Theorem $2.2$ (\lowerRomannumeral{1}), Chapter $2$]{GP} to conclude that $g_k$ is continuous on $[T-t,T-s]$.  Secondly, since
	\[
	g_k\left(\xi\right)=\int_{T-t}^{T-s}	\left\langle \sigma\left(\tau\right)e_k,v_{}\right\rangle_H
	 \left(\xi-\tau\right)^{\alpha-1}\dd \tau,\quad \xi \in \left[T-s,T\right),
	\]
	the continuity of $g_k$ on $[T-s,T)$ can be inferred employing the dominated convergence theorem. Thus, $g_k$ is continuous on $(0,T)$. This shows that  the components $f\boldsymbol{\cdot}e_k,\,k=1,\dots, d,$ of the function $f\colon[0,T]\to \mathbb{R}^d$ are almost everywhere equal on $(0,T)$ to continuous functions $g_k$, which completes  the proof.
\end{proof}
\begin{rem}
	Lemma \ref{neg} prevents us from choosing another interesting  drift $b(t,x)$, namely 
	\[
	[b\left(t,x\right)]\left(\xi\right)=\bar{\beta}\left(x\right)1_{\left(t,T\right)}\left(\xi\right),\quad \xi \in\left[0,T\right],\text{ for some }\bar{\beta}\in \mathcal{B}_b(H;\mathbb{R}^d).
	\]
\end{rem} 

\subsection{Concerning regularization by noise via the Kolmogorov equation}\label{reg_noise}
Among the interests of the Kolmogorov equation, there is the theory of regularization by noise: both in finite and infinite dimensions, it has been	shown that a sufficiently regular solution to the Kolmogorov equation allows to prove suitable uniqueness results for the underlying stochastic differential equation (see examples in  \cite{DFPR, DFRV, FlaSF, Stroock, Veret}). In contrast with Sections \ref{sec_abstract}-\ref{sec_KO}, one
deals with a stochastic PDE 
\begin{equation}\label{SPDE}
	\dd X_t=b(t,X_t)\,\dd t+ \sigma\left(t\right)\dd W_t,\quad X_0=x\in H,
\end{equation}
which  –a priori– is not well posed, because $b\colon[0,T]\times H\to H$ is subject to weak regularity assumptions not including Lipschitz continuity. The aim is to prove the uniqueness in law of a mild solution to \eqref{SPDE}. A typical approach to achieve this takes the following steps:

\begin{enumerate}
	\item\label{pr_1} Write the Kolmogorov equation in mild form \eqref{mild} associated with \eqref{SPDE}
	and prove the  existence of solutions by a fixed point argument.
	
	\item \label{pr_2} Possibly after a regularization procedure (see an example in infinite dimensions
	in \cite[Theorem 2.9, Section 2.3.3]{FlaSF}), apply It\^{o} formula to
	$u\left(  T-t,X_{t}\right)  $, where $u$ solves \eqref{mild} and $X_{t}$ is any solution of \eqref{SPDE},  prove
	that the local martingale term is a martingale and obtain an expression for 
	\[
	\mathbb{E}\left[\phi \left(X_{t}\right)\right].
	\]
	In this way, one deduces that two solutions have the same marginals. A control on
	the gradient of $u$, like the one discussed in Subsection \ref{grad_est}, may help in this step to	prove that the local martingale term is a martingale.
	
	\item \label{pr_3}Apply specific arguments (see \cite{Stroock}, \cite{Trevisan}) to obtain uniqueness in law.
\end{enumerate}
Under suitable assumptions on $b$ which guarantee the well–posedness of \eqref{mild} (whence Step \ref{pr_1} follows), the details of Steps \ref{pr_2}-\ref{pr_3} will be the subject of a future research.
\begin{appendices}
\section{Regularity  of the solution (\ref{sol_kol}) of the Kolmogorov equation }\label{appendix1}
In this appendix, we present an auxiliary lemma, namely Lemma \ref{Reg_sol}, containing regularity results about the solution $u\colon [0,T] \times H\to \mathbb{R}$ of the Kolmogorov backward equation \eqref{kolm} defined in \eqref{sol_kol}. Such a lemma plays a key role in the proof of Theorem \ref{Kolm_thm}.
\begin{lemma}\label{Reg_sol}
	Suppose that $\Phi\in C_b^{2+\beta}\left(H\right)$ and that Assumption \ref{B_3} holds. Then, the map $u\colon[0,T]\times H\to \mathbb{R}$ defined in \eqref{sol_kol} belongs to $L^\infty\big(0,T; C_b^{2+\beta}\left(H\right)\big)\cap C([0,T]\times H;\mathbb{R})$. In particular, there exists a constant $C_{d,T,\beta,\Phi}>0$ such that 
	\begin{equation}\label{unif_time}
		\norm{D^2u\left(t,\phi\right)-D^2u\left(t,\psi\right)}_{\mathcal{L}\left(H;H\right)}\le C_{d,T,\beta,\Phi} \norm{\phi-\psi}^\beta_2,\quad \phi,\psi\in H,\,t\in\left[0,T\right].
	\end{equation}

Furthermore, the map $(t,\phi, \psi)\mapsto \langle \nabla u (t,\phi), \psi \rangle_H$ [resp., $(t,\phi, \psi,\eta)\mapsto\langle D^2u(t,\phi)\psi,\eta\rangle_H$] is  continuous in $[0,T]\times H\times H$ [resp., $[0,T]\times H\times H\times H$].
\end{lemma}
\begin{proof}
We start off by proving that $u\in  C([0,T]\times H;\mathbb{R})$. Consider $t\in[0,T],\,\phi\in H$ and two sequences $(t_n)_n\subset[0,T]$ and $(\phi_n)_n\subset H$ such that $t_n\to t$ and $\phi_n\to \phi $ as $n\to \infty$. Since $\nabla \Phi\colon H\to H$ is bounded, by the mean value theorem we compute, recalling the definition of $u$ in \eqref{sol_kol},
\begin{multline}\label{riprendere}
	\left|u(t_n,\phi_n)-u(t,\phi)\right|\le \mathbb{E}\left[\left|\Phi\left(w_T^{t_n,\phi_n}\right)-\Phi\left(w_T^{t_n,\phi}\right)\right|\right]
	+
	\mathbb{E}\left[\left|\Phi\left(w_T^{t_n,\phi}\right)-\Phi\left(w_T^{t,\phi}\right)\right|\right]
	\\
	\le \norm{\nabla\Phi}_\infty\left(\norm{w_T^{t_n,\phi_n}-w_T^{t_n,\phi}}_{\mathcal{H}}
	+
	\norm{w_T^{t_n,\phi}-w_T^{t,\phi}}_{\mathcal{H}}
	\right).
\end{multline}
By \eqref{initial_c} in Corollary \ref{lip}, we infer that $\lim_{n\to \infty}\big|\!\big|w_T^{t_n,\phi_n}-w_T^{t_n,\phi}\big|\!\big| _{\mathcal{H}}=0$. As for $\big|\!\big|{w_T^{t_n,\phi}-w_T^{t,\phi}}\big|\!\big|_{\mathcal{H}}$, we first assume that $t_n> t$.
Then, by the flow property in \eqref{cocycle} and Corollary \ref{lip} we have, for some constants $c_1,c_2>0$ which might depend on $\phi$,
\begin{equation*}
		\norm{w_T^{t_n,\phi}-w_T^{t,\phi}}_{\mathcal{H}}=
			\norm{w_T^{t_n,\phi}-w_T^{t_n,w_{t_n}^{t,\phi}}}_{\mathcal{H}}\le 
			c_1\norm{w_{t_n}^{t,\phi}-\phi}_{\mathcal{H}}\le c_2\sqrt{\left|t_n-t\right|},
\end{equation*}
where the last inequality is due to Lemma \ref{p_lemma}, see \eqref{unpo+}. An analogous argument shows that the previous bound holds even in   the case $t_n\le t$, therefore $\lim_{n\to \infty}	\big|\!\big|{w_T^{t_n,\phi}-w_T^{t,\phi}}\big|\!\big|_{\mathcal{H}}=0$.
Going back to \eqref{riprendere}, we conclude that $\lim_{n\to\infty}	\left|u(t_n,\phi_n)-u(t,\phi)\right|=0$, hence $u\colon [0,T]\times H\to \mathbb{R}$ is continuous, as desired.

We now prove that $u\in L^\infty\big(0,T; C_b^{2+\beta}\left(H\right)\big)$.
Since $\Phi \in C_b^{2+\beta}\left(H\right)$, there exists a constant $C_\Phi>0$ such that 
\begin{equation}\label{phi_2}
	\norm{D^2\Phi\left(\phi\right)-D^2\Phi\left(\psi\right)}_{\mathcal{L}\left(H;H\right)}\le C_\Phi\norm{\phi-\psi}^\beta_2,\quad \phi,\psi \in H.
\end{equation}	
Obviously, from the boundedness of $\Phi$ we have $\norm{u}_\infty= \sup_{t\in [0,T]}\sup_{\phi\in H}\left|u(t,\phi)\right|<\infty$. First, we want to show that, for every $t\in\left[0,T\right]$, $u(t,\cdot)\in C_b^1(H)$, with 
	\begin{equation}\label{grad_u}
		\left\langle \nabla u\left(t,\phi\right) , \psi\right\rangle_H=\mathbb{E}\left[\left\langle \nabla \Phi\left(w^{t,\phi}_T\right), Dw^{t,\phi}_T\psi\right\rangle_H\right]
		,\quad \phi,\,\psi\in H.
	\end{equation}
To see this, by Taylor's formula applied to $\Phi$ we compute, for every $\phi,\,h\in H$,
\begin{align}\label{ba1}
\notag&\mathbb{E}\left[\left|\Phi\left(w^{t,\phi+h}_T\right)-\Phi\right(w^{t,\phi}_T\left)
	- \left\langle \nabla \Phi\left(w^{t,\phi}_T\right), Dw^{t,\phi}_Th\right\rangle_H\right|\right]
	\\&\notag\le \norm{\nabla\Phi}_\infty
		\mathbb{E}\left[
		\norm {w^{t,\phi+h}_T-w^{t,\phi}_T-Dw^{t,\phi}_Th}_2\right]\\&\notag
	\qquad\qquad\qquad+\mathbb{E}\left[\left|\int_{0}^{1}\left\langle
	\nabla\Phi\left(w^{t,\phi}_T+r\left(w^{t,\phi+h}_T-w^{t,\phi}_T\right)\right)-\nabla\Phi\left(w^{t,\phi}_T\right),w^{t,\phi+h}_T-w^{t,\phi}_T
	\right\rangle_H\dd r \right|\right]
	\\&\le \norm{\nabla \Phi}_\infty \norm{w^{t,\phi+h}_T-w^{t,\phi}_T-Dw^{t,\phi}_Th}_\mathcal{H}
	+\norm{D^2\Phi}_\infty\norm{w^{t,\phi+h}_T-w^{t,\phi}_T}_\mathcal{H}^2
	=\text{o}\left(\norm{h}_2\right).
\end{align}
Here, for the second inequality we use the Lipschitz continuity of the map $\nabla \Phi \colon H\to H$ --guaranteed by the mean value theorem-- and for the third equality we invoke Corollary \ref{lip} and Theorem \ref{diff_w_thm}. This shows \eqref{grad_u}, from which 
we deduce  the continuity of the function $\nabla u(t,\cdot)\colon H \to H$. In particular, by \eqref{alpha-h},	there exists a constant $C_1=C_1(d,T)$ such that $\norm{\nabla u}_\infty\le C_1\norm{\nabla\Phi}_\infty$.
\\
We also note that, arguing as in \eqref{ba1} and thanks to the estimates of $\big|\!\big|{w^{t,\phi+h}_T-w^{t,\phi}_T-Dw^{t,\phi}_Th}\big|\!\big|_\mathcal{H}$ in the proof of Theorem \ref{diff_w_thm} (see, for instance, \eqref{ma2}-\eqref{ma4}),  for every $M>0$ we have
\begin{equation}\label{sghe}
\sup_{t\in[0,T]}\sup_{\norm{\phi}_2,\norm{\psi}_2\le M}\mathbb{E}\left[\left|\Phi\left(w^{t,\phi+h\psi}_T\right)-\Phi\right(w^{t,\phi}_T\left)
-h \left\langle \nabla \Phi\left(w^{t,\phi}_T\right), Dw^{t,\phi}_T\psi\right\rangle_H\right|\right]=\text{o}\left({h}\right),\quad h\in\mathbb{R}.
\end{equation}
which gives the continuity of the map $(t,\phi,\psi)\mapsto \langle \nabla u (t,\phi), \psi \rangle_H$ in $[0,T]\times H\times H$ as $u\in C([0,T]\times H;\mathbb{R})$.

Secondly, we claim that $u\left(t,\cdot\right)$ is twice Fréchet differentiable in $H$, with 
\begin{align}\label{D^2u}
	 &\notag\left\langle D^2u\left(t,\phi\right)\psi,\eta\right\rangle_H
	=
	\mathbb{E}\left[\left\langle D^2\Phi\left(w^{t,\phi}_T\right)Dw_T^{t,\phi}\psi,Dw_T^{t,\phi}\eta\right\rangle_H
	+\left\langle
	\nabla \Phi\left(w_T^{t,\phi}\right), D^2w_T^{t,\phi}\left(\psi,\eta\right)
	\right\rangle_H\right]
	,\\&\qquad \phi,\psi,\eta \in H.
\end{align}
 Indeed, recalling \eqref{grad_u}, an application of Taylor's formula on $\nabla \Phi$ yields
\begin{align}\label{aste1}
\notag&\left|\left\langle \nabla u\left(t,\phi+h\right)-\nabla u \left(t,\phi\right)- D^2u\left(t,\phi\right)h, \psi\right\rangle_H\right|\\
&\notag\quad=
\Big|\mathbb{E}\Big[\left\langle \nabla \Phi\left(w^{t,\phi+h}_T\right), Dw_T^{t,\phi+h}\psi\right\rangle_H
-
\left\langle \nabla \Phi\left(w^{t,\phi}_T\right), Dw_T^{t,\phi}\psi\right\rangle_H\\&\notag \hspace{10em}-\left\langle D^2\Phi\left(w_T^{t,\phi}\right)Dw_T^{t,\phi}h,Dw_T^{t,\phi}\psi\right\rangle_H-\left\langle
\nabla \Phi\left(w_T^{t,\phi}\right), D^2w_T^{t,\phi}\left(h,\psi\right)
\right\rangle_H\Big]\Big|
\\\notag&
\quad \le \mathbb{E}\left[\left|
\left\langle D^2\Phi\left(w^{t,\phi}_T\right)\left(w^{t,\phi+h}_T-w^{t,\phi}_T-Dw^{t,\phi}_Th\right), Dw_T^{t,\phi}\psi\right\rangle_H\right|\right]
\\&\qquad\quad  \notag
+\mathbb{E}\left[\left|\left\langle \nabla \Phi\left(w_T^{t,\phi}\right),Dw_T^{t,\phi+h}\psi-Dw_T^{t,\phi}\psi-D^2w_T^{t,\phi}\left(h,\psi\right)\right\rangle_H\right|\right]
\\&\qquad\quad  +
\mathbb{E}\left[\left|\left\langle \nabla \Phi\left(w_T^{t,\phi+h}\right)-\nabla \Phi\left(w_T^{t,\phi}\right),\left(Dw_T^{t,\phi+h}-Dw_T^{t,\phi}\right)\psi\right\rangle_H\right|
\right]+R_\Phi\left(\phi,\psi, h\right)\notag\\&\quad \eqqcolon
\left(\mathbf{\upperRomannumeral{1}}_1+\mathbf{\upperRomannumeral{2}}_1+\mathbf{\upperRomannumeral{3}}_1+R_{\Phi}\right)\left(\phi,\psi, h\right)
, 
\end{align}
for every $\phi,\psi, h\in H.$ Here, we denote by 
\begin{align*}
	R_{\Phi}\left(\phi,\psi,h\right)=\!\mathbb{E}\left[\left|\left\langle\int_{0}^{1}\!\left(D^2\Phi\left(w_T^{t,\phi}+r\left(w_T^{t,\phi+h}-w_T^{t,\phi}\right)\right)\!-\!D^2\Phi\left(w_T^{t,\phi}\right)\right)\left(w_T^{t,\phi+h}-w_T^{t,\phi}\right)\dd r, Dw^{t,\phi}_T\psi\right\rangle_H\right|\right].
\end{align*}
Using \eqref{alpha-h}, \eqref{phi_2} and Corollary \ref{lip},  for some constant $c_3>0$ we compute 
\begin{align*}
	R_{\Phi}\left(\phi,\psi,h\right)&\le C_\Phi{C_1} \mathbb{E}\left[\norm{w_T^{t,\phi+h}-w_T^{t,\phi}}^{1+\beta}_2\right] \norm{\psi}_2
	\le 
	C_\Phi{C_1} \norm{w_T^{t,\phi+h}-w_T^{t,\phi}}^{1+\beta}_\mathcal{H}\norm{\psi}_2
	\\&\le c_3 \norm{\psi}_2\norm{h}_2^{1+\beta},\quad \phi,\psi,h\in H,
\end{align*}
where we also employ Jensen's inequality noticing that $1+\beta\le 2$.
Next,
\begin{align*}
\left|	\mathbf{\upperRomannumeral{1}}_1\left(\phi,\psi,h\right)\right|
\le {C_1}\norm{D^2\Phi}_\infty \norm{\psi}_2\norm{w^{t,\phi+h}_T-w^{t,\phi}_T-Dw^{t,\phi}_Th}_\mathcal{H},\quad \phi,\psi, h \in H,
\end{align*}
and 
\begin{align*}
\left|\mathbf{\upperRomannumeral{2}}_1\left(\phi,\psi,h\right)\right|
\le 
\norm{\nabla \Phi}_\infty
\norm{\psi}_2
\norm{Dw_T^{t,\phi+h}-Dw_T^{t,\phi}-D^2w_T^{t,\phi}\left(h,\cdot\right)}_{\mathcal{L}\left(H;\mathcal{H}\right)},\quad \phi,\psi,h\in H.
\end{align*}
Finally, by Corollary \ref{lip}   and \eqref{alpha-h}  (recall that,  under Assumption \ref{B_3}, we  take $\gamma=\beta$ in \eqref{meglio}, see \eqref{unavolta})
\[
\left|	\mathbf{\upperRomannumeral{3}}_1\left(\phi,\psi,h\right)\right|
\le 
C_1\norm{D^2\Phi}_\infty\norm{\psi}_2\mathbb{E}\left[\norm{w_T^{t,\phi+h}-w_T^{t,\phi}}_2^{1+\beta}\right]
\le 
\tilde{c}\norm{D^2\Phi}_\infty\norm{\psi}_2 \norm{h}^{1+\beta}_2,\quad \phi,\psi,h \in H,
\]
for some $\tilde{c}>0$.
Going back to \eqref{aste1}, by Theorem \ref{diff_w2_thm}, the previous estimates let us write, for some constant $C>0,$
\begin{align}\label{ba2}
	\notag&\norm{\nabla u\left(t,\phi+h\right)-\nabla u \left(t,\phi\right)- D^2u\left(t,\phi\right)h}_2
	=
	\sup_{\norm{\psi}_2\le 1} \left|\left\langle \nabla u\left(t,\phi+h\right)-\nabla u \left(t,\phi\right)- D^2u\left(t,\phi\right)h,\psi\right\rangle_H\right|\\&\quad \notag
	\le C\left(\norm{w^{t,\phi+h}_T-w^{t,\phi}_T-Dw^{t,\phi}_Th}_\mathcal{H}+\norm{Dw_T^{t,\phi+h}-Dw_T^{t,\phi}-D^2w_T^{t,\phi}\left(h,\cdot\right)}_{\mathcal{L}\left(H;\mathcal{H}\right)}+\norm{h}_2^{1+\beta}\right)\\&\quad =\text{o}\left(\norm{h}_2\right),\quad\phi, h \in H,  
\end{align}
which proves \eqref{D^2u}. In particular, by \eqref{alpha-h}-\eqref{alpha-h2}, there is a constant $C_2=C_2(d,T)>0$ such that 
\[
	\norm{D^2u}_\infty
	%=\sup_{t\in\left[0,T\right]}\sup_{\phi\in H} \norm{D^2u\left(t,\phi\right)}_{\mathcal{L}\left(H;H\right)}
	\le 
	C_2\left(\norm{D^2\Phi}_\infty+\norm{\nabla \Phi}_\infty\right).
\]
In addition, arguing as in \eqref{ba2} (see also \eqref{sghe}) and thanks to the estimates of $\big|\!\big|Dw_T^{t,\phi+h}-Dw_T^{t,\phi}-D^2w_T^{t,\phi}\left(h,\cdot\right)\big|\!\big|_{\mathcal{L}\left(H;\mathcal{H}\right)}$ in the proof of Theorem \ref{diff_w2_thm} (see, for instance, \eqref{ma3}),  for every $M>0$ we have
\begin{align*}
	\sup_{t\in[0,T]}\sup_{\norm{\phi}_2,\norm{\psi}_2,\norm{\eta}_2\le M} \left|\left\langle \nabla u\left(t,\phi+h\psi\right)-\nabla u \left(t,\phi\right)-h D^2u\left(t,\phi\right)\psi,\eta\right\rangle_H\right|=\text{o}\left(h\right),\quad  h\in\mathbb{R}.
\end{align*}
Since we have proved that $\langle \nabla u(t,\phi),\psi\rangle_H$ is continuous in $[0,T]\times H\times H$, the previous equation ensures that the map $(t,\phi,\psi,\eta)\mapsto\langle D^2u(t,\phi)\psi,\eta\rangle_H$ is continuous in $[0,T]\times H\times H\times H$, as desired.

In conclusion, we prove that $u\left(t,\cdot \right)\in C_b^{2+\beta}\left(H\right)$. From \eqref{D^2u}, for every $\phi_1,\phi_2\in H,$
\begin{align*}
	&\left\langle \left(D^2u\left(t,\phi_1\right) -D^2u\left(t,\phi_2\right)\right)\psi,\eta\right\rangle_H\\&\quad =	
	\mathbb{E}\left[\left\langle \left(D^2\Phi\left(w^{t,\phi_1}_T\right)-D^2\Phi\left(w^{t,\phi_2}_T\right)\right)Dw_T^{t,\phi_1}\psi,Dw_T^{t,\phi_1}\eta\right\rangle_H\right]
	\\&\qquad 	
	+
	\mathbb{E}\left[\left\langle D^2\Phi\left(w^{t,\phi_2}_T\right)\left(Dw_T^{t,\phi_1}-Dw_T^{t,\phi_2}\right)\psi,Dw_T^{t,\phi_1}\eta\right\rangle_H\right]
	\\&\qquad +
	\mathbb{E}\left[\left\langle D^2\Phi\left(w^{t,\phi_2}_T\right)Dw_T^{t,\phi_2}\psi,\left(Dw_T^{t,\phi_1}-Dw_T^{t,\phi_2}\right)\eta\right\rangle_H\right]
	\\&\qquad +\mathbb{E}\left[\left\langle
	\nabla \Phi\left(w_T^{t,\phi_1}\right)-\nabla \Phi\left(w_T^{t,\phi_2}\right), D^2w_T^{t,\phi_1}\left(\psi,\eta\right)
	\right\rangle_H\right]
		\\&\qquad +\mathbb{E}\left[\left\langle
	\nabla \Phi\left(w_T^{t,\phi_2}\right), \left(D^2w_T^{t,\phi_1}-D^2w_T^{t,\phi_2}\right)\left(\psi,\eta\right)
	\right\rangle_H\right]	\\&\qquad\eqqcolon
	\left(\mathbf{\upperRomannumeral{1}}_2+\mathbf{\upperRomannumeral{2}}_2+\mathbf{\upperRomannumeral{3}}_2+\mathbf{\upperRomannumeral{4}}_2+\mathbf{\upperRomannumeral{5}}_2\right)\left(\phi_1,\phi_2,\psi,\eta\right),\quad \psi,\eta\in H.
\end{align*}
To keep notation the short, in what follows we consider arbitrary $\psi, \eta \in H$, we do not write $\left(\phi_1,\phi_2,\psi, \eta\right)$ and we denote by $c=c(d,T,\beta)>0$  a constant that might change from line to line. Observe that, by  \eqref{alpha-h}-\eqref{phi_2}, Corollary \ref{lip} and Jensen's inequality,
\[
	\left|\mathbf{\upperRomannumeral{1}}_2\right|\le c\,C_{\Phi}\norm{\psi}_2\norm{\eta}_2\norm{\phi_1-\phi_2}_2^\beta.
\] 
Moreover, by \eqref{alpha-h} (see also \eqref{unavolta}),
\[
	\left|\mathbf{\upperRomannumeral{2}}_2\right|\le  c\norm{D^2\Phi}_\infty\norm{\psi}_2\norm{\eta}_2\norm{\phi_1-\phi_2}^{\beta}_2.
\]
An analogous estimate holds for $|\mathbf{\upperRomannumeral{3}}_2|$, too. As for the remaining addends, by \eqref{alpha-h2} we have 
\[
	\left|\mathbf{\upperRomannumeral{4}}_2\right|\le  c \norm{D^2\Phi}_\infty\norm{\psi}_2\norm{\eta}_2 \norm{\phi_1-\phi_2}_2,
\] 
and
\[
	\left|\mathbf{\upperRomannumeral{5}}_2\right|\le c\norm{\nabla \Phi}_\infty\norm{\psi}_2\norm{\eta}_2\norm{\phi_1-\phi_2}_2^\beta.
\]
Thus, the function $D^2u\left(t,\cdot\right)\colon H\to \mathcal{L}\left(H;H\right)$ is $\beta-$H\"older continuous uniformly in time and the proof is complete.
\end{proof}
\end{appendices}
	
\end{document}